\documentclass[reqno]{amsart}
\usepackage{amsmath}
\usepackage{paralist}
\usepackage{graphics} %% add this and next lines if pictures should be in esp format
\usepackage{epsfig} %For pictures: screened artwork should be set up with an 85 or 100 line screen
\usepackage{graphicx}  \usepackage{epstopdf}%This is to transfer .eps figure to .pdf figure; please compile your paper using PDFLeTex or PDFTeXify.
\usepackage{amssymb}
\usepackage{amsthm}
\usepackage{epstopdf}
\usepackage{verbatim}
\usepackage{mathrsfs}
\usepackage{mathtools}
\usepackage{pstricks}
\usepackage{cleveref}
\usepackage{relsize}
\usepackage{tikz}
\usetikzlibrary{matrix}
\usepackage{subcaption}
\usepackage{pgfplots}
\usepackage{fixltx2e}
\usepackage{bm}
\usepackage{appendix}

\theoremstyle{plain}
\newtheorem{thrm}{Theorem}[section]
\newtheorem{lmm}[thrm]{Lemma}

\newtheorem{prpstn}[thrm]{Proposition}

\theoremstyle{definition}
\newtheorem{dfntn}[thrm]{Definition}
\newtheorem{cnjctr}[thrm]{Conjecture}

\newtheorem{rmrk}[thrm]{Remark}

\theoremstyle{plain}
  \def\xR{\mathbb{R}}

\def\xCzero{{\rm C}^{0}}
\def\xCone{{\rm C}^{1}} 
\def\xCtwo{{\rm C}^{2}} 
\def\xCinfty{{\rm C}^{\infty}} 
\def\xCn#1{{\rm C}^#1}
\def\xHone{{\rm H}^{1}}
\def\xHtwo{{\rm H}^{2}}

\def\xLone{{\rm L}^{1}}
\def\xLtwo{{\rm L}^{2}} 
\def\xLinfty{{\rm L}^{\infty}} 
\def\xLn#1{{\rm L}^#1}
\def\xHmone{{\rm H}^{-1}} % Sobolev space H^{-1}.
\def\xLoneFtwo{{\rm L}_{1,{\rm F}}^2} % Defined in \eqref{def_L_1F^2}

\newtheorem{lemma}{Lemma}[subsection]

\crefname{equation}{}{} %customize the way a type show up by \cref
\Crefname{equation}{}{} %customize the way a type show up by \Cref

\title[The turnpike property in semilinear control] %Use the shortened version of the full title
{The turnpike property in semilinear control}

\thanks{This project has received funding from the European Research Council (ERC) under the European Union’s Horizon 2020 research and innovation programme (grant agreement No 694126-DYCON).\\
	We acknowledge professor Enrique Zuazua for his helpful remarks on the manuscript. We thank the referees for their interesting comments.}

\begin{document}
	\maketitle
	
	% Enter the first author's name and address:
	\centerline{\scshape Dario Pighin}
	\medskip
	{\footnotesize
		% please put the address of the first author
		\centerline{Departamento de Matem\'aticas, Universidad Aut\'onoma de Madrid}
		%   \centerline{Other lines}
		\centerline{28049 Madrid, Spain}
	} % Do not forget to end the {\footnotesize by the sign }
	\medskip
	{\footnotesize
		% please put the address of the first author
		\centerline{Chair of Computational Mathematics, Fundaci\'on Deusto}
		%   \centerline{Other lines}
		\centerline{University of Deusto, 48007, Bilbao, Basque Country, Spain}
	} % Do not forget to end the {\footnotesize by the sign }

	\begin{abstract}
		An exponential turnpike property for a semilinear control problem is proved. The state-target is assumed to be small, whereas the initial datum
		can be arbitrary.
		
		Turnpike results are also obtained for large targets, requiring that the control acts everywhere. In this case, we prove the convergence of the infimum of the averaged time-evolution functional towards the steady one.
		
		Numerical simulations are performed.
	\end{abstract}
	
	%The title of your section 1
\section*{Introduction}

In this manuscript, the long time behaviour of semilinear optimal control problems as the time-horizon tends to infinity is analyzed. Our results are global, meaning that we do not require smallness of the initial datum for the governing state equation.

In \cite{PZ2}, A. Porretta and E. Zuazua studied turnpike property for control problems governed by a semilinear heat equation, with dissipative nonlinearity. In particular, \cite[Theorem 1]{PZ2} yields the existence of a solution to the optimality system fulfilling the turnpike property, under smallness conditions on the initial datum and the target. Our first goal is to
\begin{enumerate}
	\item prove that in fact the (exponential) turnpike property is satisfied by the optimal control and state;
	\item remove the smallness assumption on the initial datum.
\end{enumerate}
We keep the smallness assumption on the target. This leads to the smallness and uniqueness of the steady optima (see \cite[subsection 3.2]{PZ2}), whence existence and uniqueness of the turnpike follows. We also treat the case of large targets, under the added assumption that control acts everywhere. In this case, we prove a weak turnpike result, which stipulates that the averaged infimum of the time-evolution functional converges towards the steady one. We also provide an $\xLtwo$ bound of the time derivative of optimal states, uniformly in the time horizon.

Generally speaking, in turnpike theory a time-evolution optimal control problem is considered together with its steady version. The ``turnpike property'' is verified if the time-evolution optima remain close to the steady optima up to some thin initial and final boundary layers.

An extensive literature is available on this subject. A pioneer on the topic has been John von Neumann \cite{von1945model}. In econometrics turnpike phenomena have been widely investigated by several scholars including P. Samuelson and L.W. McKenzie \cite{Samuelson1,Samuelson1972,LS,mckenzie1963turnpike,mckenzie1976turnpike,carlson2012infinite,haurie1976optimal}. Long time behaviour of optimal control problems has been studied by P. Kokotovic and collaborators \cite{wilde1972dichotomy,anderson1987optimal}, by  R.T. Rockafellar \cite{rockafellar1973saddle} and by A. Rapaport and P. Cartigny \cite{rapaport2004turnpike,rapaport2005competition}. A.J. Zaslavski wrote a book \cite{zaslavski2006turnpike} on the subject. A turnpike-like asymptotic simplification has been obtained in the context of optimal design of the diffusivity matrix for the heat equation \cite{allaire2010long}. In the papers \cite{damm2014exponential,grune2016relation,grune2018turnpike,trelat2018integral}, the concept of (measure) turnpike is related to the dissipativity of the control problem.

Recent papers on long time behaviour of Mean Field games \cite{cardaliaguet2012long,cardaliaguet2013long,GPT} motivated new research on the topic. A special attention has been paid in providing an exponential estimate, as in the work \cite{porretta2013long} by A. Porretta and E. Zuazua, where linear quadratic control problems were considered. These results have later been extended in \cite{trelat2015turnpike,PZ2,zamorano2018turnpike,trelat2018steady,gruneexponential,grune2019sensitivity} to control problems governed by a nonlinear state equation and applied to optimal control of the Lotka-Volterra system \cite{ibanez2017optimal}. Recently, turnpike property has been studied around nonsteady trajectories \cite{trelat2018steady,faulwasser2019towards,Grune2018,TLA}. The turnpike property is intimately related to asymptotic behaviour of the Hamilton-Jacobi equation \cite{kouhkouh2018dynamic,THJ}.

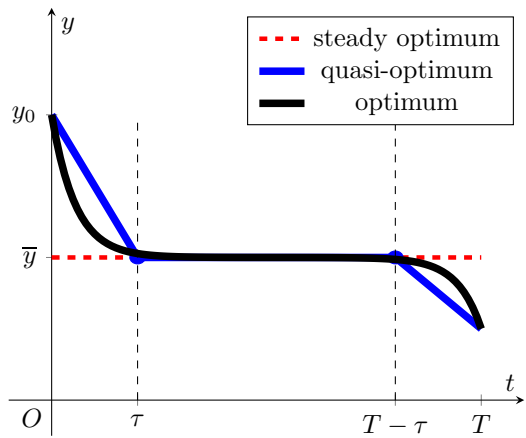
\begin{figure}
	\begin{center}
		\begin{tikzpicture}[
		declare function={
			func(\x)= (\x < 2) * (2*(2-\x)+4)   +
			and(\x >= 2, \x < 8) * (4)     +
			(\x >= 8) * (-(\x-8) + 4)
			;
		}
		]
		\begin{axis}[
		axis x line=middle, axis y line=middle,
		ymin=-1, ymax=11, ytick={4, 8}, yticklabels={$\overline{y}$, $y_0$}, ylabel=$y$,
		xmin=-1, xmax=11, xtick={2,8,10}, xticklabels={$\tau$,$T-\tau$,$T$}, xlabel=$t$,
		domain=0:10,samples=101, % added
		]
		
		{\addplot [dashed, red,line width=0.06cm] {4};}
		\addlegendentry{steady optimum}
		{\addplot [blue,line width=0.1cm] {func(x)};}
		\addlegendentry{quasi-optimum}
		\draw [fill, blue] (30,50) circle [radius=1.8];
		\draw [dashed] (30,10) -- (30,90);
		\draw [fill, blue] (90,50) circle [radius=1.8];
		\draw [dashed] (90,10) -- (90,90);
		{\addplot [black,line width=0.1cm] {4*exp(-1.8*\x)-2*exp(-1.8*(10-\x))+4};}
		\addlegendentry{optimum}
		\node [below left, black] at (10,10) {$O$};
		\end{axis}
		\end{tikzpicture}
		%[ADD] Fix uncover.
		%			\begin{tikzpicture}
		%								%the legend must be outside the "axis", to be located outside the plot-square.
		%				\matrix [draw,below left] at (current bounding box.north east) {
		%					\draw [-] (0.0, 3) -- (0.3,3);
		%					\node [right, black] at (0.06,3) {steady optimum}; \\
		%					%					\node [label=right:Foo] {}; \\
		%					%					\node [label=right:Bar] {}; \\
		%					%					\node [label=right:Baz] {}; \\
		%					%					\node [label=right:Foobar] {}; \\
		%				};
		%			\end{tikzpicture}
		\caption{quasi-optimal turnpike strategies}
		\label{quasi-optimal_intro}
	\end{center}
\end{figure}

Note that for a general optimal control problem, even in absence of a turnpike result, we can construct turnpike strategies (see \cite[Remark 7]{hernandez2017greedy}) as in \cref{quasi-optimal_intro}:
\begin{enumerate}
	\item in a short time interval $[0,\tau]$ drive the state from the initial configuration $y_0$ to a turnpike $\overline{y}$;
	\item in a long time arc $[\tau,T-\tau]$, remain on $\overline{y}$;
	\item in a short final arc $[T-\tau,T]$, use to control to match the required terminal condition at time $t=T$.
\end{enumerate}
In general, the corresponding control and state are not optimal, being not smooth. However, they are easy to construct.

The proof of turnpike results is harder than the above construction. In fact, to prove turnpike results, one has to ensure that there is not another time-evolving strategy which is significantly better than the above one. In case the turnpike property is verified, the above strategy is quasi-optimal.

\subsection*{Statement of the main results}

We consider the semilinear optimal control problem:
\begin{equation}\label{functional_slt}
\min_{u\in \xLtwo((0,T)\times \omega)}J_{T}(u)=\frac12 \int_0^T\int_{\omega} |u|^2 dxdt+\frac{\beta}{2}\int_0^T\int_{\omega_0} |y-z|^2 dxdt,
\end{equation}
where:
\begin{equation}\label{semilinear_internal_1_slt}
\begin{dcases}
y_t-\Delta y+f\left(y\right)=u\chi_{\omega}\hspace{2.8 cm} & \mbox{in} \hspace{0.10 cm}(0,T)\times\Omega\\
y=0  & \mbox{on}\hspace{0.10 cm} (0,T)\times \partial \Omega\\
y(0,x)=y_0(x)  & \mbox{in}\hspace{0.10 cm}  \Omega.
\end{dcases}
\end{equation}
As usual, $\Omega$ is a regular bounded open subset of $\xR^n$, with $n=1,2,3$. The nonlinearity $f$ is $\xCn{3}$ nondecreasing, with $f\left(0\right)=0$. The action of the control is localized by multiplication by $\chi_{\omega}$, characteristic function of the open subregion $\omega\subseteq \Omega$. The target $z$ is assumed to be in $\xLinfty(\omega_0)$. The well poosedeness and regularity properties of the state equation are studied in \Cref{appendixsec:sml_heat_well_posedeness}.
%Since that the nonlinearity is nondecreasing, the semilinear problem \cref{semilinear_internal_1_slt} is well-posed \cite[chapter 5]{barbu2010nonlinear}
%%or a s a gradient flow of the Dirichlet energy .\cite{boccardo2013elliptic}
%. Namely, given an initial datum $y_0\in \xLinfty\left(\Omega\right)$ and a control $u\in \xLtwo((0,T)\times \omega)$, there exists a unique solution
%\begin{equation*}
%y\in \xCzero([0,T];\xLtwo\left(\Omega\right))\cap \xLtwo(0,T;\xHone_0\left(\Omega\right)).
%\end{equation*}
$\omega_0\subseteq \Omega$ is an open subset and $\beta\geq 0$ is a weighting parameter. As $\beta$ increases, the distance between the optimal state and the target decreases.

By the direct method in the calculus of variations \cite{ESC,troltzsch2010optimal}, there exists a global minimizer of \cref{functional_slt}. As we shall see, uniqueness can be guaranteed, provided that the initial datum and the target are small enough in the uniform norm.
%[CHECK]

Taking the G\^ateaux differential of the functional \cref{functional_slt} and imposing the Fermat stationary condition, we realize that any optimal control reads as $u^T=-q^T\chi_{\omega}$, where $\left(y^T,q^T\right)$ solves
\begin{equation}\label{semilinear_internal_parabolic_1}
\begin{dcases}
y^T_t-\Delta y^T+f(y^T)=-q^T\chi_{\omega}\hspace{2.8 cm} & \mbox{in} \hspace{0.10 cm}(0,T)\times\Omega\\
y^T=0  & \mbox{on}\hspace{0.10 cm} (0,T)\times \partial \Omega\\
y^T(0,x)=y_0(x)  & \mbox{in}\hspace{0.10 cm}  \Omega\\
-q^T_t-\Delta q^T+f^{\prime}(y^T)q^T=\beta(y^T-z)\chi_{\omega_0}\hspace{2.8 cm} & \mbox{in} \hspace{0.10 cm}(0,T)\times \Omega\\
q^T=0  & \mbox{on}\hspace{0.10 cm} (0,T)\times\partial \Omega\\
q^T(T,x)=0 & \mbox{in} \hspace{0.10 cm}\Omega.\\
\end{dcases}
\end{equation}

In order to study the turnpike, we need to study the steady version of \cref{semilinear_internal_1_slt}-\cref{functional_slt}:
\begin{equation}\label{steady_functional_slt}
\min_{u_s\in \xLtwo\left(\Omega\right)}J_{s}(u_s)=\frac12 \int_{\omega} |u_s|^2 dx+\frac{\beta}{2}\int_{\omega_0} |y_s-z|^2 dx,
\end{equation}
where:
\begin{equation}\label{semilinear_internal_elliptic_1_slt}
\begin{dcases}
-\Delta y_s+f(y_s)=u_s\chi_{\omega}\hspace{2.8 cm} & \mbox{in} \hspace{0.10 cm}\Omega\\
y_s=0  & \mbox{on}\hspace{0.10 cm} \partial \Omega.
\end{dcases}
\end{equation}
Under the same assumptions required for the problem \cref{semilinear_internal_1_slt}-\cref{functional_slt}, for any given control $u_s\in \xLtwo\left(\Omega\right)$, there exists a unique state $y_s\in \xHtwo\left(\Omega\right)\cap \xHone_0\left(\Omega\right)$ solution to \cref{semilinear_internal_elliptic_1_slt} (see e.g. \cite{boccardo2013elliptic}).

By adapting the techniques of \cite{ESC}, we have the existence of a global minimizer $\overline{u}$ for \cref{steady_functional_slt}. The corresponding optimal state is denoted by $\overline{y}$. If the target is sufficiently small in the uniform norm, the optimal control is unique (see \cite[subsection 3.2]{PZ2}). Furthermore any optimal control satisfies $\overline{u}=-\overline{q}\chi_{\omega}$, where the pair $\left(\overline{y},\overline{q}\right)$ solves the steady optimality system
\begin{equation}\label{semilinear_internal_elliptic_2}
\begin{dcases}
-\Delta \overline{y}+f(\overline{y})=-\overline{q}\chi_{\omega}\hspace{2.8 cm} & \mbox{in} \hspace{0.10 cm}\Omega\\
\overline{y}=0  & \mbox{on}\hspace{0.10 cm} \partial \Omega\\
-\Delta \overline{q}+f^{\prime}(\overline{y})\overline{q}=\beta(\overline{y}-z)\chi_{\omega_0}\hspace{2.8 cm} & \mbox{in} \hspace{0.10 cm}\Omega\\
\overline{q}=0  & \mbox{on}\hspace{0.10 cm} \partial \Omega.
\end{dcases}
\end{equation}

Consider the control problem \cref{semilinear_internal_elliptic_1_slt}-\cref{steady_functional_slt}. By \cite[section 3]{PZ2}, there exists $\delta >0$ such that if the \textbf{initial datum} and the target fulfill the \textbf{smallness condition}
\begin{equation}\label{smallness_condition}
\|y_0\|_{\xLinfty\left(\Omega\right)}\leq \delta \hspace{0.3 cm} \mbox{and} \hspace{0.3 cm} \|z\|_{\xLinfty\left(\omega_0\right)}\leq \delta,
\end{equation}
there exists a solution $\left(y^T,q^T\right)$ to the Optimality System
\begin{equation*}\label{semilinear_internal_parabolic_1_PZ2}
\begin{dcases}
y^T_t-\Delta y^T+f(y^T)=-q^T\chi_{\omega}\hspace{2.8 cm} & \mbox{in} \hspace{0.10 cm}(0,T)\times\Omega\\
y^T=0  & \mbox{on}\hspace{0.10 cm} (0,T)\times \partial \Omega\\
y^T(0,x)=y_0(x)  & \mbox{in}\hspace{0.10 cm}  \Omega\\
-q^T_t-\Delta q^T+f^{\prime}(y^T)q^T=\beta(y^T-z)\chi_{\omega_0}\hspace{2.8 cm} & \mbox{in} \hspace{0.10 cm}(0,T)\times \Omega\\
q^T=0  & \mbox{on}\hspace{0.10 cm} (0,T)\times\partial \Omega\\
q^T(T,x)=0 & \mbox{in} \hspace{0.10 cm}\Omega\\
\end{dcases}
\end{equation*}
satisfying for any $t\in [0,T]$
\begin{equation*}
\|q^T(t)-\overline{q}\|_{\xLinfty\left(\Omega\right)}+\|y^T(t)-\overline{y}\|_{\xLinfty\left(\Omega\right)}\leq K\left[\exp\left(-\mu t\right)+\exp\left(-\mu(T-t)\right)\right],
\end{equation*}
where $K$ and $\mu$ are $T$-independent.

In the aforementioned result, the turnpike property is satisfied by one solution to the optimality system. Since our problem may be not convex, we cannot directly assert that such solution of the optimality system is the unique minimizer (optimal control) for \cref{semilinear_internal_elliptic_1_slt}-\cref{steady_functional_slt}.

%\begin{rmrk}\label{remark_needofsmalltargets}
%	The above method fails when the target is large.	As a matter of fact, a key point in our method is the smallness of steady optima $(\overline{u},\overline{y})$, for small targets. Then, by triangle inequality, we can replace the smallness condition
%	\begin{equation*}
%	\|y^T(t)-\overline{y}\|_{\xLinfty\left(\Omega\right)}\leq \delta,
%	\end{equation*}
%	with the simpler one
%	\begin{equation*}
%	\|y^T(t)\|_{\xLinfty\left(\Omega\right)}\leq \delta^{\prime}, \hspace{0.6 cm}\mbox{with}\hspace{0.16 cm}\delta^{\prime}\coloneqq\frac{\delta}{2}.
%	\end{equation*}
%	This cannot be done if the target is large. Indeed, in that case the steady optima may be large as well.
%\end{rmrk}

\subsubsection*{Large initial data and small targets}

We start by keeping the running target small, but allowing the initial datum for \cref{semilinear_internal_1_slt} to be large.

\begin{thrm}\label{th_slt_1}
	Consider the control problem \cref{semilinear_internal_1_slt}-\cref{functional_slt}, with nondecreasing nonlinearity $f$. Let $u^T$ be a minimizer of \cref{functional_slt}. For any $\varepsilon >0$, there exists $\rho_{\varepsilon}>0$ such that for every initial datum $y_0\in \xLinfty\left(\Omega\right)$ and target $z$ verifying
	\begin{equation}\label{smallness_target_any_initial_data}
	\|z\|_{\xLinfty\left(\omega_0\right)}\leq \rho_{\varepsilon},
	\end{equation}
	we have
	\begin{equation}\label{exp_turnpike_2}
	\|u^T(t)-\overline{u}\|_{\xLinfty\left(\omega\right)}+\|y^T(t)-\overline{y}\|_{\xLinfty\left(\Omega\right)}\leq K_{\varepsilon}\exp\left(-\mu t \right)+\varepsilon\exp\left(-\mu(T-t)\right),\hspace{0.6 cm}\forall t\in [0,T],
	\end{equation}
	for some $T$-independent constants $K_{\varepsilon}=K_{\varepsilon}\left(\Omega,\omega,\omega_0,\left\|y_0\right\|_{\xLinfty\left(\Omega\right)},\varepsilon\right)$ and $\mu=\mu\left(\Omega,\omega,\omega_0\right)>0$, for any choice of $R\geq \left\|y_0\right\|_{\xLinfty\left(\Omega\right)}$.
	%$\mu=\mu(\Omega,\omega,\omega_0,\dots,\dots, )$
	%$K=K(\Omega,\omega,\omega_0,y_0,\dots,\dots, )$.
\end{thrm}

Note that $\rho$ is smaller than the smallness parameter $\delta$ in \eqref{smallness_condition}. Furthermore, the smallness of the target yields the smallness of the final arc, when the state leaves the turnpike to match the final condition for the adjoint.

The main  ingredients our proofs require are:
\begin{enumerate}
	\item prove a $\xLinfty$ bound of the norm of the optimal control, uniform in the time horizon $T>0$ (Lemma \ref{lemma_bound_optima} in \cref{subsec:5.2.1});
	\item proof of the turnpike property for \textit{small data} and \textit{small targets}. Note that, in \cite[section 3]{PZ2}, the authors prove the existence of a solution to the optimality system enjoying the turnpike property. In this preliminary step, for \textit{small data} and \textit{small targets}, we prove that any optimal control verifies the turnpike property (Lemma \ref{lemma_3} in \cref{subsec:5.2.1});
	\item for \textit{small targets} and \textit{any data}, proof of the smallness of $\|y^T(t)\|_{\xLinfty\left(\Omega\right)}$ in time $t$ large (\cref{subsec:5.2.2}). This is done by estimating the critical time $t_s$ needed to approach the turnpike;
	\item conclude concatenating the two former steps (\cref{subsec:5.2.2}).
\end{enumerate}

\Cref{th_slt_1} ensures that the conclusion of \cite[section 3]{PZ2} holds for the optimal pair.

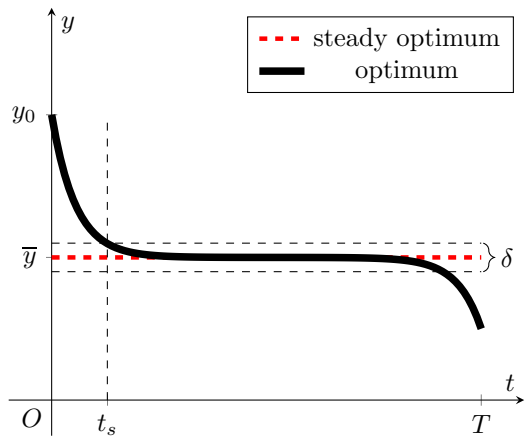
\begin{figure}
	\begin{center}
		\begin{tikzpicture}
		\begin{axis}[
		axis x line=middle, axis y line=middle,
		ymin=-1, ymax=11, ytick={4, 8}, yticklabels={$\overline{y}$, $y_0$}, ylabel=$y$,
		xmin=-1, xmax=11, xtick={1.30,10}, xticklabels={$t_s$,$T$}, xlabel=$t$,
		domain=0:10,samples=101, % added
		]
		
		{\addplot [dashed, red,line width=0.06cm] {4};}
		\addlegendentry{steady optimum}
		{\addplot [black,line width=0.1cm] {4*exp(-1.8*\x)-2*exp(-1.8*(10-\x))+4};}
		\addlegendentry{optimum}
		\draw [dashed] (10,54) -- (110,54);
		\draw [dashed] (10,46) -- (110,46);
		\draw [dashed] (22.963,10) -- (22.963,90);
		\node [below left, black] at (10,10) {$O$};
		\draw [decorate,decoration={brace,amplitude=4.2pt,mirror,raise=4pt},yshift=0pt]
		(108,46) -- (108,54) node [black,midway,xshift=0.46cm] {$\delta$};
		\end{axis}
		\end{tikzpicture}
		%[ADD] Fix uncover.
		%			\begin{tikzpicture}
		%								%the legend must be outside the "axis", to be located outside the plot-square.
		%				\matrix [draw,below left] at (current bounding box.north east) {
		%					\draw [-] (0.0, 3) -- (0.3,3);
		%					\node [right, black] at (0.06,3) {steady optimum}; \\
		%					%					\node [label=right:Foo] {}; \\
		%					%					\node [label=right:Bar] {}; \\
		%					%					\node [label=right:Baz] {}; \\
		%					%					\node [label=right:Foobar] {}; \\
		%				};
		%			\end{tikzpicture}
		\caption{global-local argument}\label{Global-local argument_intro}
	\end{center}
\end{figure}

Let us outline the proof of $\left(3\right)$ (\cref{Global-local argument_intro}), the existence of $\tau$ upper bound for the minimal time needed to approach the turnpike $t_s$.\\
Suppose, by contradiction, that the critical time $t_s$ to approach the turnpike is very large. Accordingly, the time-evolution optimal strategy obeys the following plan:
\begin{enumerate}
	\item stay away from the turnpike for long time;
	\item move close to the turnpike;
	\item enjoy a final time-evolution performance, cheaper than the steady one.
\end{enumerate}
Then, in phase 1, with respect to the steady performance, an extra cost is generated, which should be regained in phase 3. At this point, we realize that this is prevented by validity of the local turnpike property. Indeed, once the time-evolution optima approach the turnpike at some time $t_s$, the optimal pair satisfies the turnpike property for larger times $t\geq t_s$. Hence, for $t\geq t_s$, the time-evolution performance cannot be significantly cheaper than the steady one. Accordingly, we cannot regain the extra-cost generated in phase 1, so obtaining a contradiction.
\begin{rmrk}\label{remark_est_adj_state}
	All estimates we have obtained carry over to the adjoint state. This can be obtained by using the adjoint equation in \eqref{semilinear_internal_parabolic_1} and the equation satisfied by the difference $\varphi^T\coloneqq q^T-\overline{q}$
	\begin{equation*}
	\begin{cases}
	-\varphi^T_t-\Delta \varphi^T+f^{\prime}\left(\overline{y}\right)\varphi^T=\beta\chi_{\omega_0}\left(y^T-\overline{y}\right)+\left(f^{\prime}\left(\overline{y}\right)-f^{\prime}\left(y^T\right)\right)q^T\hspace{2.8 cm} & \mbox{in} \hspace{0.10 cm}(0,T)\times \Omega\\
	\varphi^T=0  & \mbox{on}\hspace{0.10 cm} (0,T)\times\partial \Omega\\
	\varphi^T(T,x)=-\overline{q}(x) & \mbox{in} \hspace{0.10 cm}\Omega.\\
	\end{cases}
	\end{equation*}
	Note that the aforementioned adjoint equations are stable since $f$ is increasing, whence $f^{\prime}\geq 0$.
\end{rmrk}

\begin{rmrk}\label{remark_est_stabilizable_system}
	In this manuscript we addressed a model case, where the state equation \eqref{semilinear_internal_1_slt} is stable with null control. Our analysis is applicable to more general stabilizable systems. Namely, it suffices the existence of a control such that the system stabilizes to zero. This can be seen by combining our techniques with turnpike theory for linear quadratic control \cite{porretta2013long,grune2019sensitivity}.
\end{rmrk}

\subsubsection*{Control acting everywhere: convergence of averages for arbitrary targets}

In \cref{sec:5.3} we deal with large targets, supposing the control acts everywhere (i.e. $\omega=\Omega$). We prove that the averages converge. Furthermore, we obtain an $\xLtwo$ bound for the time derivative of optimal states. The bound is uniform independent of the time horizon $T$, meaning that, if $T$ is large, the time derivative of the optimal state is small for most of the time.

\begin{thrm}\label{theorem_convergence_averages}
	Take an arbitrary initial datum $y_0\in \xLinfty\left(\Omega\right)$ and an arbitrary target $z\in \xLinfty(\omega_0)$. Consider the time-evolution control problem \cref{semilinear_internal_1_slt}-\cref{functional_slt} and its steady version \cref{semilinear_internal_elliptic_1_slt}-\cref{steady_functional_slt}. Assume the nonlinearity $f$ is nondecreasing and $\omega=\Omega$. Then, averages converge
	\begin{equation}\label{theorem_convergence_averages_eq1}
	\frac{1}{T}\inf_{\xLtwo((0,T)\times \Omega)} J_{T}\underset{T\to +\infty}{\longrightarrow}\inf_{\xLtwo\left(\Omega\right)}J_s.
	\end{equation}
	Suppose in addition $y_0\in \xLinfty\left(\Omega\right)\cap \xHone_0\left(\Omega\right)$. Let $u^T$ be an optimal control for \cref{semilinear_internal_1_slt}-\cref{functional_slt} and let $y^T$ be the corresponding state, solution to \cref{semilinear_internal_1_slt}, with control $u^T$ and initial datum $y_0$. Then, the $\xLtwo$ norm of the time derivative of the optimal state is bounded uniformly in $T$
	\begin{equation}\label{theorem_convergence_averages_eq2}
	\left\|y_t^T\right\|_{\xLtwo((0,T)\times \Omega)}\leq K,
	\end{equation}
	the constant $K$ being $T$-independent.
\end{thrm}

The proof of \Cref{theorem_convergence_averages}, available in \cref{sec:5.3}, is based on the following representation formula for the time-evolving functional (Lemma \ref{lemma_functional rewritten}):
\begin{align}\label{lemma_functional rewritten_eq1_intro}
J_{T}(u)&= \int_0^T J_s\big(-\Delta y(t,\cdot) +f\left(y(t,\cdot)\right)\big) dt\nonumber\\
&\;\hspace{0.33 cm}+\frac12 \int_0^T\int_{\Omega} \left|y_t(t,x)\right|^2 dxdt\nonumber\\
&\;\hspace{0.33 cm}	+\frac12\int_{\Omega}\left[\left\|\nabla y(T,x)\right\|^2+2F\left(y(T,x)\right)-\left\|\nabla y_0(x)\right\|^2-2F\left(y_0(x)\right)\right]dx,
\end{align}
where $F\left(y\right)\coloneqq \int_0^yf\left(\xi\right)d\xi$ and for a.e. $t\in (0,T)$, $J_s\big(-\Delta y(t,\cdot) +f\left(y(t,\cdot)\right)\big)$ denotes the evaluation of the steady functional $J_s$ at control $u_s(\cdot)\coloneqq -\Delta y(t,\cdot) +f\left(y(t,\cdot)\right)$ and $y$ is the state associated to control $u$ solving
\begin{equation}\label{semilinear_internal_6_slt_intro}
\begin{dcases}
y_t-\Delta y+f\left(y\right)=u\hspace{2.8 cm} & \mbox{in} \hspace{0.10 cm}(0,T)\times\Omega\\
y=0  & \mbox{on}\hspace{0.10 cm} (0,T)\times \partial \Omega\\
y(0,x)=y_0(x)  & \mbox{in}\hspace{0.10 cm}  \Omega.
\end{dcases}
\end{equation}
Note that the above formula is valid for initial data $y_0\in \xLinfty\left(\Omega\right)\cap \xHone_0\left(\Omega\right)$. However, by the regularizing effect of \cref{semilinear_internal_6_slt_intro} and the properties of the control problem, one can reduce to the case of smooth initial data.

By means of \cref{lemma_functional rewritten_eq1_intro}, the functional $J_T$ can be seen as the sum of three terms:
\begin{enumerate}
	\item $\int_0^T J_s\big(-\Delta y(t,\cdot) +f\left(y(t,\cdot)\right)\big) dt$, which stands for the ``steady'' cost at a.e. time $t\in (0,T)$ integrated over $(0,T)$;
	\item $\frac12 \int_0^T\int_{\Omega} \left|y_t(t,x)\right|^2 dxdt$, which penalizes the time derivative of the functional;
	\item $\frac12\int_{\Omega}\left[\left\|\nabla y(T,x)\right\|^2+2F\left(y(T,x)\right)-\left\|\nabla y_0(x)\right\|^2-2F\left(y_0(x)\right)\right]dx$, which depends on the terminal values of the state.
\end{enumerate}
Choose now an optimal control $u^T$ for \cref{semilinear_internal_1_slt}-\cref{functional_slt} and plug it in \cref{lemma_functional rewritten_eq1_intro}. By Lemma \ref{lemma_bound_optima}, the term\\
$\frac12\int_{\Omega}\left[\left\|\nabla y(T,x)\right\|^2+2F\left(y(T,x)\right)-\left\|\nabla y_0(x)\right\|^2-2F\left(y_0(x)\right)\right]dx$ can be estimated uniformly in the time horizon. At the optimal control, the term $\int_0^T\int_{\Omega} \left|y_t(t,x)\right|^2 dxdt$ has to be small and the ``steady'' cost\\
$\int_0^T J_s\big(-\Delta y(t,\cdot) +f\left(y(t,\cdot)\right)\big) dt$ is the dominant addendum. This is the basic idea of our approach to prove turnpike results for large targets.

The rest of the manuscript is organized as follows. In \cref{sec:5.2} we prove \Cref{th_slt_1}. In \cref{sec:5.3}, we prove \Cref{theorem_convergence_averages}. In \cref{sec:5.4} we perform some numerical simulations. The appendix is mainly devoted to the proof of the uniform bound of the optima (Lemma \ref{lemma_bound_optima}) and a PDE result needed for Lemma \ref{lemma_functional rewritten}.

\section{Proof of \Cref{th_slt_1}}
\label{sec:5.2}

\subsection{Preliminary Lemmas}
\label{subsec:5.2.1}

As announced, we firstly exhibit an upper bound of the norms of the optima in terms of the data. Note that the Lemma below yields a uniform bound for large targets as well.

\begin{lmm}\label{lemma_bound_optima}
	Consider the control problem \cref{semilinear_internal_1_slt}-\cref{functional_slt}. Let $R>0$, $y_0\in \xLinfty\left(\Omega\right)$ and $z\in \xLinfty(\omega_0)$, satisfying $\|y_0\|_{\xLinfty\left(\Omega\right)}\leq R$ and $\|z\|_{\xLinfty(\omega_0)}\leq R$. Let $u^T$ be an optimal control for \cref{semilinear_internal_1_slt}-\cref{functional_slt}. Then, $u^T$ and $y^T$ are bounded and
	\begin{equation}\label{lemma_bound_optima_estimate_1}
	\|u^T\|_{\xLinfty((0,T)\times\omega)}+\left\|y^T\right\|_{\xLinfty((0,T)\times\Omega)}\leq K\left[\|y_0\|_{\xLinfty\left(\Omega\right)}+\|z\|_{\xLinfty(\omega_0)}\right],
	\end{equation}
	where the constant $K$ is independent of the time horizon $T$, but it depends on $R$.
	%K=K(\Omega,\omega,\omega_0,R,\dots,\dots)
	%WARNING.
	%It is false:
	%\left\{\|u^T\|_{\xLtwo\left((0,T)\times \Omega\right)}\leq K\left[1+\|y_0\|_{\xLinfty\left(\Omega\right)}+\|z\|_{\xLinfty\left(\omega_0\right)}\right],
	%the constant $K$ being independent of the time horizon $T$.\right\}
\end{lmm}
The proof is postponed to the Appendix.

The second ingredient for the proof of \Cref{th_slt_1} is the following Lemma.

\begin{lmm}\label{lemma_3}
	Consider the control problem \cref{semilinear_internal_1_slt}-\cref{functional_slt}. Let $y_0\in \xLinfty\left(\Omega\right)$ and $z\in \xLinfty(\omega_0)$. There exists $\delta>0$  such that, if
	\begin{equation}\label{smallness_target_datum_0}
	\|z\|_{\xLinfty\left(\omega_0\right)}\leq \delta\hspace{1 cm}\mbox{and}\hspace{1 cm}\|y_0\|_{\xLinfty\left(\Omega\right)}\leq \delta,
	\end{equation}
	the functional \cref{functional_slt} admits a unique global minimizer $u^T$. Furthermore, for every $\varepsilon >0$ there exists $\delta_{\varepsilon}>0$ such that, if
	\begin{equation}\label{smallness_target_datum_varepsilon}
	\|z\|_{\xLinfty\left(\omega_0\right)}\leq \delta_{\varepsilon}\hspace{1 cm}\mbox{and}\hspace{1 cm}\|y_0\|_{\xLinfty\left(\Omega\right)}\leq \delta_{\varepsilon},
	\end{equation}
	the functional \cref{functional_slt} admits a unique global minimizer $u^T$ and
	\begin{equation}\label{lemma_3_turnpike_estimate_varepsilon}
	\|u^T(t)-\overline{u}\|_{\xLinfty\left(\omega\right)}+\|y^T(t)-\overline{y}\|_{\xLinfty\left(\Omega\right)}\leq \varepsilon\left[\exp\left(-\mu t\right)+\exp\left(-\mu(T-t)\right)\right],\hspace{0.6 cm}\forall t\in [0,T],
	\end{equation}
	$(\overline{u},\overline{y})$ being the optimal pair for \cref{steady_functional_slt}. The constants $\delta_{\varepsilon}$ and $\mu>0$ are independent of the time horizon.
\end{lmm}
\begin{rmrk}\label{remark_Riccati}
	In \eqref{lemma_3_turnpike_estimate_varepsilon} the rate $\mu$ is given by
	\begin{eqnarray}\label{formulas_22_PZ2}
	\left\|\mathscr{E}(t)-\widehat{E}\right\|_{\mathscr{L}\left(\xLtwo\left(\Omega\right),\xLtwo\left(\Omega\right)\right)}&\leq&C\exp\left(-\mu t\right),\nonumber\\
	\left\|\exp\left(-t M\right)\right\|_{\mathscr{L}(\xLtwo\left(\Omega\right),\xLtwo\left(\Omega\right))}&\leq &\exp\left(-\mu t\right),\hspace{0.3 cm}M\coloneqq -\Delta +f^{\prime}\left(\overline{y}\right)+\widehat{E}\chi_{\omega}.
	\end{eqnarray}
	where $\mathscr{E}$ and $\widehat{E}$ denote respectively the differential and algebraic Riccati operators (see \cite[equation (22)]{PZ2}) and $\Delta : \xHone_0\left(\Omega\right)\longrightarrow \xHmone\left(\Omega\right)$ is the Dirichlet Laplacian.
\end{rmrk}
%\begin{rmrk}
%	Inequality \eqref{lemma_3_turnpike_estimate_varepsilon} yields the smallness of the state in the final arc, when the state leaves the turnpike. This is due to the smallness of the target $z$ assumed in \eqref{smallness_target_datum_0}. Indeed, if $z$ is small, the optimal steady state $\overline{y}$ and adjoint $\overline{q}$ are small as well \cite[subsection 3.2]{PZ2}. Therefore, the final condition of the time-evolution adjoint in \eqref{semilinear_internal_parabolic_1} is small and the , thus ensuring the smallness of the optimal control and state close to the final time.
%\end{rmrk}
\begin{proof}[Proof of Lemma \ref{lemma_3}]
	Consider initial data $y_0$ and target $z$, such that $\left\|y_0\right\|_{\xLinfty\left(\Omega\right)}\leq 1$ and $\left\|z\right\|_{\xLinfty\left(\omega_0\right)}\leq 1$. We introduce the critical ball
	\begin{equation}\label{critical_ball_36}
	B\coloneqq \left\{u\in \xLinfty((0,T)\times {\omega}) \ \Big| \ \|u\|_{\xLinfty\left((0,T)\times \omega\right)}\leq K\left[\|y_0\|_{\xLinfty\left(\Omega\right)}+\|z\|_{\xLinfty\left(\omega_0\right)}\right] \right\},
	\end{equation}
	where $K$ is the constant appearing in \cref{lemma_bound_optima_estimate_1}.\\
	\textit{Step 1} \ \textbf{Strict convexity in $B$ for small data}\\
	By \cite[section 5]{casas2012second} or \cite{ESC}, the second order G\^ateaux differential of $J$ reads as
	\begin{eqnarray*}\label{sec_Frechet_differential}
		\langle d^2J_{T}(u)w,w\rangle&=&\int_0^T\int_{\omega} w^2 dxdt+\beta\int_0^T\int_{\omega_0}|\psi_w|^2 dxdt-\int_0^T\int_{\Omega}f^{\prime\prime}(y)q|\psi_w|^2 dxdt,
	\end{eqnarray*}
	where $y$ solves \cref{semilinear_internal_1_slt} with control $u$ and initial datum $y_0$, $\psi_w$ solves the linearized problem
	\begin{equation}\label{linearized_internal_1}
	\begin{dcases}
	(\psi_w)_t-\Delta \psi_w+f^{\prime}(y)\psi_w=w\chi_{\omega}\hspace{2.8 cm} & \mbox{in} \hspace{0.10 cm}(0,T)\times\Omega\\
	\psi_w=0  & \mbox{on}\hspace{0.10 cm} (0,T)\times \partial \Omega\\
	\psi_w(0,x)=0  & \mbox{in}\hspace{0.10 cm}  \Omega
	\end{dcases}
	\end{equation}
	and
	\begin{equation}\label{adjoint_parabolic_1}
	\begin{dcases}
	-q_t-\Delta q+f^{\prime}(y)q=\beta(y-z)\chi_{\omega_0}\hspace{2.8 cm} & \mbox{in} \hspace{0.10 cm}(0,T)\times \Omega\\
	q=0  & \mbox{on}\hspace{0.10 cm} (0,T)\times\partial \Omega\\
	q(T,x)=0 & \mbox{in} \hspace{0.10 cm}\Omega.\\
	\end{dcases}
	\end{equation}
	Since $f^{\prime}(y)\geq 0$,
	\begin{equation*}
	\|\psi_w\|_{\xLtwo((0,T)\times \Omega)}\leq C_0\|w\|_{\xLtwo((0,T)\times \omega)},
	\end{equation*}
	for some constant $C_0=C_0\left(\Omega\right)$.\\
	Let $u\in B$. By applying a comparison argument to \cref{semilinear_internal_1_slt} and \cref{adjoint_parabolic_1},
	\begin{equation*}
	\|y\|_{\xLinfty\left((0,T)\times \Omega\right)}+\|q\|_{\xLinfty\left((0,T)\times \Omega\right)}\leq C_1\left[\|y_0\|_{\xLinfty\left(\Omega\right)}+\|z\|_{\xLinfty\left(\omega_0\right)}\right],
	\end{equation*}
	with $C_1=C_1\left(\Omega,\beta\right)$. Hence,
	\begin{equation*}
	\int_0^T\int_{\Omega}\left|f^{\prime\prime}(y)q|\psi_w|^2\right| dxdt\leq C_0^2C_1C_2\left[\|y_0\|_{\xLinfty\left(\Omega\right)}+\|z\|_{\xLinfty\left(\omega_0\right)}\right]\int_0^T\int_{\omega} w^2 dxdt,
	\end{equation*}
	where $C_2\coloneqq \max_{\left[-2C_1,2C_1\right]}\left|f^{\prime\prime}\right|$ and we have used that $\left\|y_0\right\|_{\xLinfty\left(\Omega\right)}\leq 1$ and $\left\|z\right\|_{\xLinfty\left(\omega_0\right)}\leq 1$. Therefore
	\begin{equation*}
	\langle d^2J_{T}(u)w,w\rangle\geq\beta\int_0^T\int_{\omega_0}|\psi_w|^2 dxdt+\left\{ 1-C_0^2C_1C_2\left[\|y_0\|_{\xLinfty\left(\Omega\right)}+\|z\|_{\xLinfty\left(\omega_0\right)}\right]\right\}\int_0^T\int_{\omega}|w|^2 dxdt,
	\end{equation*}
	If $\|y_0\|_{\xLinfty\left(\Omega\right)}$ and $\|z\|_{\xLinfty\left(\omega_0\right)}$ are small enough, we have
	\begin{equation*}
	\langle d^2J_{T}(u)w,w\rangle\geq\frac12\int_0^T\int_{\omega}|w|^2 dxdt,
	\end{equation*}
	whence the strict convexity of $J$ in the critical ball $B$.
	Now, by \cref{lemma_bound_optima_estimate_1} and \cref{critical_ball_36}, if $\|y_0\|_{\xLinfty\left(\Omega\right)}$ and $\|z\|_{\xLinfty\left(\omega_0\right)}$ are small enough, any optimal control $u^T$ belongs to $B$. Then, there exists a unique solution to the optimality system, with control in the critical ball $B$ and such control coincides with $u^T$ the unique global minimizer of \cref{functional_slt}.\\
	\textit{Step 2} \ \textbf{Conclusion}\\ 
	First of all, by \cite[subsection 3.2]{PZ2}, if $\|z\|_{\xLinfty\left(\omega_0\right)}$ are small enough, the linearized optimality system satisfies the turnpike property. Now, let $\varepsilon >0$. We apply the fixed-point argument developed in the proof of \cite[Theorem 1 subsection 3.1]{PZ2} to the convex set
	\begin{equation}
	X\coloneqq \left\{\left(\eta, \varphi\right)\in \xLinfty\left(\left(0,T\right)\times \Omega\right)^2 \ \Big| \ \left\|\eta\left(t,\cdot\right)\right\|_{\xLinfty\left(\Omega\right)}+\left\|\eta\left(t,\cdot\right)\right\|_{\xLinfty\left(\Omega\right)}\leq \theta\left[\exp\left(-\mu t\right)+\exp\left(-\mu(T-t)\right)\right] \right\},
	\end{equation}
	for some $\theta \in \left(0, \varepsilon\right]$. Then, we can find $\delta_{\varepsilon}>0$ such that, if
	\begin{equation*}
	\|z\|_{\xLinfty\left(\omega_0\right)}\leq \delta_{\varepsilon}\hspace{1 cm}\mbox{and}\hspace{1 cm}\|y_0\|_{\xLinfty\left(\Omega\right)}\leq \delta_{\varepsilon},
	\end{equation*}
	there exists a solution $\left(y^T,q^T\right)$ to the optimality system such that
	\begin{equation*}
	\|u^T\|_{\xLinfty\left((0,T)\times \omega\right)}<\varepsilon
	\end{equation*}
	and
	\begin{equation*}
	\|u^T(t)-\overline{u}\|_{\xLinfty\left(\omega\right)}+\|y^T(t)-\overline{y}\|_{\xLinfty\left(\Omega\right)}\leq \varepsilon \left[\exp\left(-\mu t\right)+\exp\left(-\mu(T-t)\right)\right],\hspace{0.6 cm}\forall t\in [0,T].
	\end{equation*}
	By Step 1, if $\varepsilon$ is small enough, $u^T\coloneqq -q^T\chi_{\omega}$ is a strict global minimizer for $J_T$. Then, being strict, it is the unique one. This finishes the proof.
	
\end{proof}

In the following Lemma, we compare the value of the time evolution functional \cref{functional_slt} at a control $u$, with the value of the steady functional \cref{steady_functional_slt} at control $\overline{u}$, supposing that $u$ and $\overline{u}$ satisfy a turnpike-like estimate.

\begin{lmm}\label{lemma_estimate_functional_slt}
	Consider the time-evolution control problem \cref{semilinear_internal_1_slt}-\cref{functional_slt} and its steady version \cref{semilinear_internal_elliptic_1_slt}-\cref{steady_functional_slt}. Fix $y_0\in \xLtwo\left(\Omega\right)$ an initial datum and $z\in \xLtwo(\omega_0)$ a target. Let $\overline{u}\in \xLinfty\left(\Omega\right)$ be a control and let $\overline{y}$ be the corresponding solution to \cref{semilinear_internal_elliptic_1_slt}. Let $u\in \xLinfty((0,T)\times \omega)$ be a control and $y$ the solution to \cref{semilinear_internal_1_slt}, with control $u$. Assume
	\begin{equation}\label{lemma_estimate_functional_turnpike_estimate}
	\|u(t)-\overline{u}\|_{\xLinfty\left(\Omega\right)}+\|y(t)-\overline{y}\|_{\xLinfty\left(\Omega\right)}\leq K\left[\exp\left(-\mu t\right)+\exp\left(-\mu(T-t)\right)\right],\hspace{0.6 cm}\forall t\in [0,T],
	\end{equation}
	with $K=K(\Omega,\beta,y_0)$ and $\mu=\mu(\Omega,\beta)$.
	%HYPO: K=K(\Omega,\omega,\omega_0,y_0,\beta,\dots)
	Then,
	\begin{equation}\label{values_ineq}
	\left| J_{T}(u)-TJ_s\left(\overline{u}\right)\right|\leq C\left[1+\|\overline{u}\|_{\xLinfty\left(\omega\right)}+\|z\|_{\xLinfty\left(\omega_0\right)}\right],
	\end{equation}
	the constant $C$ depending only on the above constant $K$, $\beta$ and $\mu$.
	%$K=K(\Omega,\omega,\omega_0,y_0,\beta,\dots)$
\end{lmm}
\begin{proof}[Proof of Lemma \ref{lemma_estimate_functional_slt}]
	%\textit{Step 1} \ \textbf{Conclusion}\\
	We estimate
	\begin{eqnarray*}
		&\;&\left| J_{T}(u)-TJ_s\left(\overline{u}\right)\right|\nonumber\\
		&=&\left|\frac12 \|u\|_{\xLtwo((0,T)\times \omega)}^2+\frac{\beta}{2}\|y-z\|_{\xLtwo((0,T)\times \omega_0)}^2-T\left[\frac12\|\overline{u}\|_{\xLtwo\left(\Omega\right)}^2+\frac{\beta}{2}\|\overline{y}-z\|_{\xLtwo(\omega_0)}^2\right]\right|\nonumber\\
		&=&\left|\frac12\|u-\overline{u}\|_{\xLtwo((0,T)\times \omega)}^2+\frac{\beta}{2}\|y-\overline{y}\|_{\xLtwo((0,T)\times \omega_0)}^2\right.\nonumber\\
		&\;&+\left.\int_0^T\int_{\omega}(u-\overline{u})\overline{u} dxdt+\beta \int_0^T\int_{\omega_0}(y-\overline{y})(\overline{y}-z)dxdt\right|\nonumber\\
		&\leq &C\left[1+\|\overline{u}\|_{\xLinfty\left(\omega\right)}+\|z\|_{\xLinfty(\omega_0)}\right]\left\{\int_0^T\left[\|u-\overline{u}\|_{\xLinfty\left(\omega\right)}^2+\|u-\overline{u}\|_{\xLinfty\left(\omega\right)}\right] dt\right.\nonumber\\
		&\;&\left.+\int_0^T\left[\|y-\overline{y}\|_{\xLinfty(\omega_0)}^2+\|y-\overline{y}\|_{\xLinfty(\omega_0)}\right]dt\right\}\nonumber\\
		&\leq &C\left[1+\|\overline{u}\|_{\xLinfty(\omega)}+\|z\|_{\xLinfty(\omega_0)}\right],
	\end{eqnarray*}
	where the last inequality follows from \cref{lemma_estimate_functional_turnpike_estimate}.
\end{proof}

\begin{lmm}\label{lemma_upper_bound_value}
	Consider the time-evolution control problem \cref{semilinear_internal_1_slt}-\cref{functional_slt} and its steady version \cref{semilinear_internal_elliptic_1_slt}-\cref{steady_functional_slt}. Arbitrarily fix $y_0\in \xLinfty\left(\Omega\right)$ an initial datum and $z\in \xLinfty(\omega_0)$ a target. We have
	\begin{equation}\label{value_function_ineq}
	\inf_{\xLtwo((0,T)\times \omega)} J_{T}\leq T\inf_{\xLtwo\left(\omega\right)}J_s + K,
	\end{equation}
	the constant $K$ being independent of $T>0$.
	%$K=K(\Omega,\omega,\omega_0,\beta,y_0,z,\dots)$
\end{lmm}

The proof is available in \Cref{appendixsec:Convergence of averages}.

The following Lemma (\cref{Global-local argument_lemma_opt_est}) plays a key role in the proof of \Cref{th_slt_1}.\\
Let $u^T$ be an optimal control for \cref{semilinear_internal_1_slt}-\cref{functional_slt}. Let $y^T$ be the corresponding optimal state. For any $\varepsilon >0$, let $\delta_{\varepsilon}$ be given by \cref{smallness_target_datum_varepsilon}. Set
\begin{equation}\label{t_s_definition}
t_s\coloneqq \inf\left\{t\in [0,T] \ | \ \|y^T(t)\|_{\xLinfty\left(\Omega\right)}\leq \delta_{\varepsilon}\right\},
\end{equation}
where we use the convention $\inf(\varnothing)=T$. In the next Lemma, we are going to estimate the minimal time $t_s$.\\

\begin{figure}
	\begin{center}
		\begin{tikzpicture}
		\begin{axis}[
		axis x line=middle, axis y line=middle,
		ymin=-1, ymax=11, ytick={4, 8}, yticklabels={$\overline{y}$, $y_0$}, ylabel=$y$,
		xmin=-1, xmax=11, xtick={1.30,10}, xticklabels={$t_s$,$T$}, xlabel=$t$,
		domain=0:10,samples=101, % added
		]
		
		{\addplot [dashed, red,line width=0.06cm] {4};}
		\addlegendentry{steady optimum}
		{\addplot [black,line width=0.1cm] {4*exp(-1.8*\x)-2*exp(-1.8*(10-\x))+4};}
		\addlegendentry{optimum}
		\draw [dashed] (10,54) -- (110,54);
		\draw [dashed] (10,46) -- (110,46);
		\draw [dashed] (22.963,10) -- (22.963,90);
		\node [below left, black] at (10,10) {$O$};
		\draw [decorate,decoration={brace,amplitude=4.2pt,mirror,raise=4pt},yshift=0pt]
		(108,46) -- (108,54) node [black,midway,xshift=0.46cm] {$\delta$};
		\end{axis}
		\end{tikzpicture}
		%[ADD] Fix uncover.
		%			\begin{tikzpicture}
		%								%the legend must be outside the "axis", to be located outside the plot-square.
		%				\matrix [draw,below left] at (current bounding box.north east) {
		%					\draw [-] (0.0, 3) -- (0.3,3);
		%					\node [right, black] at (0.06,3) {steady optimum}; \\
		%					%					\node [label=right:Foo] {}; \\
		%					%					\node [label=right:Bar] {}; \\
		%					%					\node [label=right:Baz] {}; \\
		%					%					\node [label=right:Foobar] {}; \\
		%				};
		%			\end{tikzpicture}
		\caption{global-local argument employed in the proof of Lemma \ref{lemma_opt_est}}\label{Global-local argument_lemma_opt_est}
	\end{center}
\end{figure}
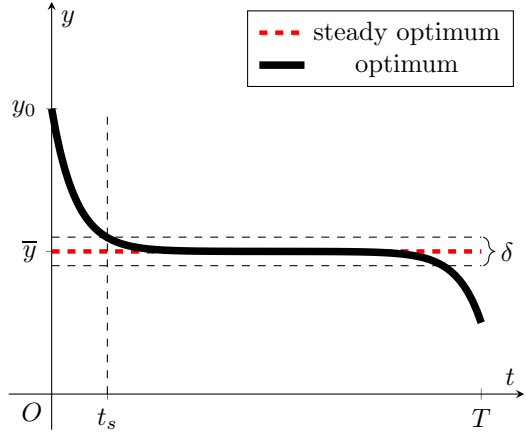

\begin{lmm}[Global attractor property]\label{lemma_opt_est}
	Consider the control problem \cref{semilinear_internal_1_slt}-\cref{functional_slt}. Let $y_0\in \xLinfty\left(\Omega\right)$ and $z\in \xLinfty(\omega_0)$. Let $u^T$ be an optimal control for \cref{semilinear_internal_1_slt}-\cref{functional_slt} and let $y^T$ be the corresponding optimal state. For any $\varepsilon >0$, there exist $\rho_{\varepsilon}=\rho_{\varepsilon}(\Omega,\beta,\varepsilon)$ and $\tau_{\varepsilon}=\tau_{\varepsilon}(\Omega,\beta,y_0,\varepsilon)$, such that if $\|z\|_{\xLinfty\left(\omega_0\right)}\leq \rho_{\varepsilon}$ and $T\geq \tau_{\varepsilon}$,
	\begin{equation}\label{lemma_opt_est_eq1}
	t_s\leq \tau_{\varepsilon}
	\end{equation}
	and
	\begin{equation}\label{lemma_opt_est_eq2}
	\|u^T(t)-\overline{u}\|_{\xLinfty\left(\omega\right)}+\|y^T(t)-\overline{y}\|_{\xLinfty\left(\Omega\right)}\leq \varepsilon\left[\exp\left(-\mu(t-t_s)\right)+\exp\left(-\mu(T-(t-t_s))\right)\right],\hspace{0.6 cm}\forall t\in [t_s,T].
	\end{equation}
	The minimal time $t_s$ is defined in \eqref{t_s_definition}, the constant $\mu$ is given by \eqref{formulas_22_PZ2} and $\delta_{\varepsilon}$ is given by \cref{smallness_target_datum_varepsilon}.
\end{lmm}
\begin{proof}[Proof of Lemma \ref{lemma_opt_est}]
	Let $R\geq \left\|y_0\right\|_{\xLinfty\left(\Omega\right)}$ be arbitrary. Throughout the proof, constant $K_1=K_1(\Omega,\beta)$ is chosen as small as needed, whereas constant $K_2=K_2\left(\Omega,\beta,R\right)$ is chosen as large as needed.\\
	\textit{Step 1} \ \textbf{Estimate of the $\xLinfty$ norm of steady optimal controls}\\
	In this step, we follow the arguments of \cite[subsection 3.2]{PZ2}. Let $\overline{u}\in \xLtwo\left(\Omega\right)$ be an optimal control for \cref{semilinear_internal_elliptic_1_slt}-\cref{steady_functional_slt}. By definition of minimizer (optimal control),
	\begin{equation*}
	\frac12\|\overline{u}\|_{\xLtwo\left(\omega\right)}^2\leq J_{s}\left(\overline{u}\right)\leq J_{s}(0)=\frac{\beta}{2}\|z\|_{\xLtwo\left(\omega_0\right)}^2\leq \frac{\beta\mu_{leb}(\omega_0)}{2}\|z\|_{\xLinfty(\omega_0)}^2.
	\end{equation*}
	Now, any optimal control is of the form $\overline{u}=-\overline{q}\chi_{\omega}$, where the pair $(\overline{y},\overline{q})$ satisfies the optimality system \cref{semilinear_internal_elliptic_2}. Since $n=1,2,3$, by elliptic regularity (see, e.g. \cite[Theorem 4 subsection 6.3.2]{PDE}) and Sobolev embeddings (see e.g. \cite[Theorem 6 subsection 5.6.3]{PDE}), $\overline{q}\in \xCzero(\overline{\Omega})$ and $\|\overline{q}\|_{\xLinfty\left(\Omega\right)}\leq K\|z\|_{\xLinfty\left(\omega_0\right)}$, where $K=K\left(\Omega,\beta\right)$. This yields $\overline{u}\in \xCzero(\overline{\omega})$ and
	\begin{equation}\label{lemma_opt_est_eq3}
	\|\overline{u}\|_{\xLinfty\left(\omega\right)}\leq K\|z\|_{\xLinfty\left(\omega_0\right)}.
	\end{equation}
	\textit{Step 2} \ \textbf{There exist $\rho_{\varepsilon}=\rho_{\varepsilon}(\Omega,\beta,\varepsilon)$ and $\tau_{\varepsilon}=\tau_{\varepsilon}(\Omega,\beta,y_0,\varepsilon)$, such that if $\|z\|_{\xLinfty\left(\omega_0\right)}\leq \rho_{{\varepsilon}}$, then the critical time satisfies $t_s\leq \tau_{\varepsilon}$}\\
	Let $\overline{u}$ be an optimal control for the steady problem. Then, by definition of minimizer (optimal control),
	\begin{equation}\label{lemma_opt_est_eq7}
	J_{T}\left(u^T\right)\leq J_{T}\left(\overline{u}\right)
	\end{equation}
	and, by Lemma \ref{lemma_upper_bound_value},
	\begin{equation}\label{lemma_opt_est_eq8}
	J_T\left(\overline{u}\right) \leq T\inf_{\xLtwo\left(\omega\right)}J_s+K_2.
	\end{equation}
	Now, we split the integrals in $J_T$ into $[0,t_s]$ and $(t_s,T]$
	\begin{eqnarray}\label{lemma_opt_est_eq10}
	J_{T}\left(u^T\right)&=&\frac12\int_0^{t_s}\int_{\omega}|u^T|^2 dt+\frac{\beta}{2}\int_{0}^{t_s}\int_{\omega_0}|y^T-z|^2dxdt\nonumber\\
	&\;&+\frac12\int_{t_s}^{T}\int_{\omega}|u^T|^2 dt+\frac{\beta}{2}\int_{t_s}^{T}\int_{\omega_0}|y^T-z|^2dxdt.
	\end{eqnarray}
	Set:
	\begin{equation*}c_y(t,x)\coloneqq 
	\begin{dcases}
	\frac{f(y^T(t,x))}{y^T(t,x)} \hspace{1.3 cm} & y^T(t,x)\neq 0\\
	f^{\prime}(0) & y^T(t,x)= 0.\\
	\end{dcases}
	\end{equation*}
	Since $f$ is nondecreasing and $f(0)=0$, we have $c_y\geq 0$. Then, Lemma \ref{lemma_reg_nonnegativepotential} (with potential $c_y$ and source term $h\coloneqq u^T\chi_{\omega}$) yields
	%and initial datum $\tilde{y}_0\coloneqq y_0$
	\begin{equation*}
	\frac12\int_0^{t_s}\int_{\omega}|u^T|^2 dt+\frac{\beta}{2}\int_{0}^{t_s}\int_{\omega_0}|y^T-z|^2dxdt\geq K_1\int_0^{t_s}\|y^T(t)\|_{\xLinfty\left(\Omega\right)}^2dt-K_2.
	\end{equation*}		
	Furthermore, by definition of $t_s$, for any $t\in [0,t_s]$, $\|y^T(t)\|_{\xLinfty\left(\Omega\right)}\geq \delta_{\varepsilon}$.
	%		Hence,
	%		\begin{equation*}
	%		\int_0^{t_s}\|y^T(t,\cdot)\|_{\xLinfty\left(\Omega\right)}^2dt\geq t_s \delta_{\varepsilon}^2,
	%		\end{equation*}
	Then,
	\begin{equation}\label{lemma_opt_est_eq11}
	\frac12\int_0^{t_s}\int_{\omega}|u^T|^2 dt+\frac{\beta}{2}\int_{0}^{t_s}\int_{\omega_0}|y^T-z|^2dxdt\geq K_1t_s \delta_{\varepsilon}^2-K_2.
	\end{equation}
	Once again, by definition of $t_s$ and since $\rho_{\varepsilon}\leq \delta_{\varepsilon}$,
	\begin{equation*}
	\|y^T(t_s)\|_{\xLinfty\left(\Omega\right)}\leq \delta_{\varepsilon}\hspace{0.3 cm}\mbox{and}\hspace{0.3 cm}\|z\|_{\xLinfty\left(\omega_0\right)}\leq \delta_{\varepsilon},
	\end{equation*}
	where $\delta_{\varepsilon}$ is given by \cref{smallness_target_datum_varepsilon}. Therefore, by Lemma \ref{lemma_3}, the turnpike estimate \cref{lemma_3_turnpike_estimate_varepsilon} is satisfied in $[t_s,T]$. Lemma \ref{lemma_estimate_functional_slt} applied in $[t_s,T]$ gives
	\begin{eqnarray}\label{lemma_opt_est_eq20}
	&\;&\frac12\int_{t_s}^{T}\int_{\omega}|u^T|^2 dt+\frac{\beta}{2}\int_{t_s}^{T}\int_{\omega_0}|y^T-z|^2dxdt\nonumber\\
	&\geq&(T-t_s)\inf_{\xLtwo\left(\omega\right)}J_s-K_2\left[1+\|\overline{u}\|_{\xLinfty\left(\omega\right)}+\|z\|_{\xLinfty\left(\omega_0\right)}\right]\nonumber\\
	&\geq &(T-t_s)\inf_{\xLtwo\left(\omega\right)}J_s-K_2,
	\end{eqnarray}
	where the last inequality is due to \cref{lemma_opt_est_eq3} and $\|z\|_{\xLinfty\left(\omega_0\right)}\leq \delta_{\varepsilon}$.
	
	At this point, by \cref{lemma_opt_est_eq10}, \cref{lemma_opt_est_eq11} and \cref{lemma_opt_est_eq20}
	\begin{equation}\label{lemma_opt_est_eq24}
	J_{T}\left(u^T\right)\geq  K_1t_s \delta_{\varepsilon}^2+(T-t_s)\inf_{\xLtwo\left(\omega\right)}J_s-K_2.
	\end{equation}
	
	Therefore, by \cref{lemma_opt_est_eq24}, \cref{lemma_opt_est_eq7} and \cref{lemma_opt_est_eq8}
	\begin{equation*}
	K_1t_s \delta_{\varepsilon}^2+(T-t_s)\inf_{\xLtwo\left(\omega\right)}J_s-K_2\leq T\inf_{\xLtwo\left(\omega\right)}J_s+K_2,
	\end{equation*}
	whence
	\begin{equation}\label{lemma_opt_est_estimate}
	t_s\left[K_1 \delta_{\varepsilon}^2-\inf_{\xLtwo\left(\omega\right)}J_s\right]\leq K_2.
	\end{equation}
	Now, by \cref{lemma_opt_est_eq3}, there exists $\rho_{\varepsilon}=\rho_{\varepsilon}(\Omega,\beta,\varepsilon)\leq \delta_{\varepsilon}$ such that, if the target $\|z\|_{\xLinfty(\omega_0)}\leq \rho_{\varepsilon}$, then $\inf_{\xLtwo\left(\omega\right)}J_s\leq \frac{K_1\delta_{\varepsilon}^2}{2}$. This, together with \cref{lemma_opt_est_estimate}, yields
	\begin{equation*}
	t_s\frac{K_1\delta_{\varepsilon}^2}{2}\leq K_2,
	\end{equation*}
	whence
	\begin{equation*}
	t_s\leq \frac{K_2}{\delta_{\varepsilon}^2}.
	\end{equation*}
	Set
	\begin{equation*}
	\tau_{\varepsilon}\coloneqq \frac{K_2}{\delta_{\varepsilon}^2}+1.
	\end{equation*}
	%WARNING
	%In the above equation, we put "+1", because we want t_s<T, \forall T\geq \tau_{\varepsilon}.
	This finishes this step.\\
	\textit{Step 3} \ \textbf{Conclusion}\\
	By Step 2, for any $T\geq \tau_{\varepsilon}$, there exists $t_s\leq \tau_{\varepsilon}$ such that
	\begin{equation}\label{lemma_opt_est_eq_66}
	\|y^T(t_s)\|_{\xLinfty\left(\Omega\right)}\leq \delta_{\varepsilon},
	\end{equation}
	where $\delta_{\varepsilon}$ is given by \cref{lemma_3_turnpike_estimate_varepsilon}. Now, by Bellman's Principle of Optimality, $u^T\hspace{-0.1 cm}\restriction_{(t_s,T)}$ is optimal for \cref{semilinear_internal_1_slt}-\cref{functional_slt}, with initial datum $y^T(t_s)$ and target $z$. Since $\rho_{\varepsilon}\leq \delta_{\varepsilon}$, we also have
	\begin{equation}\label{lemma_opt_est_eq_86}
	\|z\|_{\xLinfty\left(\omega_0\right)}\leq \rho_{\varepsilon}  \leq \delta_{\varepsilon}.
	\end{equation}
	%Hence, by \cref{lemma_opt_est_eq_66} and \cref{lemma_opt_est_eq_86}, \cref{lemma_3_turnpike_estimate_varepsilon} is satisfied with initial time $t_s$ and $\varepsilon=\varepsilon$.
	Then, we can apply Lemma \ref{lemma_3}, getting \cref{lemma_opt_est_eq2}. This completes the proof.
\end{proof}

\subsection{Proof of \Cref{th_slt_1}}
\label{subsec:5.2.2}
We now prove \Cref{th_slt_1}.

\begin{proof}[Proof of \Cref{th_slt_1}]
	%	\textit{Step 1} \ \textbf{Smallness of $\|y^T(t)\|_{\xLinfty\left(\Omega\right)}$ in large time}\\
	%	Arbitrarily fix $\varepsilon>0$. By Lemma \ref{lemma_opt_est}, there exist $\rho_{\varepsilon}=\rho_{\varepsilon}(\Omega,\beta,\varepsilon)$ and $\tau_{\varepsilon}=\tau_{\varepsilon}(\Omega,\beta,y_0,\varepsilon)$, such that if $\|z\|_{\xLinfty\left(\omega_0\right)}\leq \rho_{\varepsilon}$ and $T\geq \tau_{\varepsilon}$,
	%	\begin{equation}\label{th_slt_1_proof_eq60}
	%	\|y^T(t_s)\|_{\xLinfty\left(\Omega\right)}\leq \delta_{\varepsilon},\hspace{1 cm}\mbox{with the upper bound}\hspace{0.16 cm}t_s\leq \tau_{\varepsilon}.
	%	\end{equation}
	%	\textit{Step 1} \ \textbf{Conclusion}\\	
	Arbitrarily fix $\varepsilon>0$. By Lemma \ref{lemma_opt_est}, there exists $\rho_{\varepsilon}\left(\Omega,\beta,\varepsilon\right) >0$ such that if
	\begin{equation}\label{smallness_targetanddatum}
	\|z\|_{\xLinfty\left(\omega_0\right)}\leq \rho_{\varepsilon}\hspace{1 cm}\mbox{and}\hspace{1 cm}T\geq \tau_{\varepsilon},
	\end{equation}
	any optimal control satisfies the turnpike estimate
	\begin{equation}\label{prop1_eq66}
	\|u^T(t)-\overline{u}\|_{\xLinfty\left(\omega\right)}+\|y^T(t)-\overline{y}\|_{\xLinfty\left(\Omega\right)}\leq \varepsilon\left[\exp\left(-\mu(t-t_s)\right)+\exp\left(-\mu(T-(t-t_s))\right)\right],\hspace{0.6 cm}\forall t\in [t_s,T],
	\end{equation}
	with $t_s\leq \tau_{\varepsilon}$.
	
	Now, as in step 1 of the proof of Lemma \ref{lemma_opt_est}, we can follow the arguments of \cite[subsection 3.2]{PZ2} getting
	\begin{equation}\label{}
	\|\overline{u}\|_{\xLinfty\left(\omega\right)}+\|\overline{y}\|_{\xLinfty\left(\Omega\right)}\leq K_1\left\|z\right\|_{\xLinfty\left(\omega_0\right)},\hspace{0.6 cm}\forall t\in [t_s,T],
	\end{equation}
	with $K_1=K_1\left(\Omega,\beta\right)$.
	
	Set
	\begin{equation*}
	K_{\varepsilon}\coloneqq \exp\left(\mu \tau_{\varepsilon}\right)\max\left\{\left(K+K_1\right)\left[\|y_0\|_{\xLinfty\left(\Omega\right)}+\rho_{\varepsilon}\right], \varepsilon\right\},
	\end{equation*}
	with $\mu >0$ the exponential rate defined in \eqref{formulas_22_PZ2} and $K$ is given by \cref{lemma_bound_optima_estimate_1}. Note that $K_{\varepsilon}=K_{\varepsilon}\left(\Omega,\beta,\varepsilon\right)$ and, in particular, it is independent of the time horizon. By the above definition, for every $T>0$ and for each $t\in [0,\tau_{\varepsilon}]\cap [0,T]$
	\begin{equation}\label{th_slt_1_turnpike_1}
	\|u^T(t)-\overline{u}\|_{\xLinfty\left(\omega\right)}+\|y^T(t)-\overline{y}\|_{\xLinfty\left(\Omega\right)}\leq K_{\varepsilon}\exp\left(-\mu\tau_{\varepsilon}\right)\leq K_{\varepsilon}\exp\left(-\mu t\right).
	\end{equation}
	On the other hand, for $t\geq t_s$, \cref{prop1_eq66} holds, whence
	\begin{align}\label{}
	\|u^T(t)-\overline{u}\|_{\xLinfty\left(\omega\right)}+\|y^T(t)-\overline{y}\|_{\xLinfty\left(\Omega\right)}&\leq\varepsilon\left[\exp\left(-\mu(t-t_s)\right)+\exp\left(-\mu(T-(t-t_s))\right)\right]\nonumber\\
	&=\varepsilon\exp\left(-\mu(t-t_s)\right)+\varepsilon\exp\left(-\mu(T-(t-t_s))\right)\nonumber\\
	&=\varepsilon\exp\left(-\mu t \right)\exp\left(\mu t_s\right)+\varepsilon\exp\left(-\mu(T-t)\right)\exp\left(-\mu t_s\right)\nonumber\\
	&\leq\varepsilon\exp\left(-\mu t \right)\exp\left(\mu \tau_{\varepsilon}\right)+\varepsilon\exp\left(-\mu(T-t)\right)\nonumber\\
	&\leq K_{\varepsilon}\exp\left(-\mu t \right)+\varepsilon\exp\left(-\mu(T-t)\right).
	\end{align}
	Then, \cref{exp_turnpike_2} follows.
\end{proof}

\section{Control acting everywhere: convergence of averages}
\label{sec:5.3}

In this section, we suppose that the control acts everywhere, namely $\omega=\Omega$ in the state equation \cref{semilinear_internal_1_slt}. %This assumption may be removed , employing \cite[Lemma 1.3 paghe 1402]{FCH}.
Our purpose is to prove \Cref{theorem_convergence_averages}, valid for any data and targets.

In Lemma \ref{lemma_upper_bound_value}, we observed that, even in the more general case $\omega\subsetneq \Omega$, we have an estimate from above of the infimum of the time-evolution functional in terms of the steady functional. This is the easier task obtained by plugging the steady optimal control in the time-evolution functional. The complicated task is to estimate from below the infimum of the time-evolution functional, in terms of the steady functional. Indeed, the lower bound indicates that the time-evolution strategies cannot perform significantly better than the steady one and this is in general the hardest task in the proof of turnpike results. The key idea is indicated in Lemma \ref{lemma_functional rewritten}.

The main idea for the proof of \Cref{theorem_convergence_averages} is in the following Lemma, where an alternative representation formula for the time-evolution functional is obtained.

\begin{lmm}\label{lemma_functional rewritten}
	Consider the functional introduced in \cref{functional_slt}-\cref{semilinear_internal_1_slt} and its steady version \cref{semilinear_internal_elliptic_1_slt}-\cref{steady_functional_slt}. Set $F\left(y\right)\coloneqq \int_0^yf\left(\xi\right)d\xi$. Assume $\omega=\Omega$. Suppose the initial datum $y_0\in \xLinfty\left(\Omega\right)\cap \xHone_0\left(\Omega\right)$. Then, for any control $u\in \xLtwo((0,T)\times \omega)$, we can rewrite the functional as
	\begin{align}\label{lemma_functional rewritten_eq1}
	J_{T}(u)&= \int_0^T J_s\big(-\Delta y(t,\cdot) +f\left(y(t,\cdot)\right)\big) dt\nonumber\\
	&\;\hspace{0.33 cm}+\frac12 \int_0^T\int_{\Omega} \left|y_t(t,x)\right|^2 dxdt\nonumber\\
	&\;\hspace{0.33 cm}	+\frac12\int_{\Omega}\left[\left\|\nabla y(T,x)\right\|^2+2F\left(y(T,x)\right)-\left\|\nabla y_0(x)\right\|^2-2F\left(y_0(x)\right)\right]dx,
	\end{align}
	where, for a.e. $t\in (0,T)$, $J_s\big(-\Delta y(t,\cdot) +f\left(y(t,\cdot)\right)\big)$ denotes the evaluation of the steady functional $J_s$ at control $u_s(\cdot)\coloneqq -\Delta y(t,\cdot) +f\left(y(t,\cdot)\right)$ and $y$ is the state associated to control $u$ solution to
	\begin{equation}\label{semilinear_internal_6_slt}
	\begin{dcases}
	y_t-\Delta y+f\left(y\right)=u\hspace{2.8 cm} & \mbox{in} \hspace{0.10 cm}(0,T)\times\Omega\\
	y=0  & \mbox{on}\hspace{0.10 cm} (0,T)\times \partial \Omega\\
	y(0,x)=y_0(x)  & \mbox{in}\hspace{0.10 cm}  \Omega.
	\end{dcases}
	\end{equation}
\end{lmm}

In \cref{lemma_functional rewritten_eq1}, the term $\int_0^T\int_{\Omega} \left|y_t(t,x)\right|^2 dxdt$ emerges. This means that the time derivative of optimal states has to be small, whence the time-evolving optimal strategies for \cref{functional_slt}-\cref{semilinear_internal_1_slt} are in fact close to the steady ones.\\
The proof of Lemma \ref{lemma_functional rewritten} is based on the following PDE result, which basically asserts that the squared right hand side of the equation
\begin{equation*}
\begin{dcases}
y_t-\Delta y+f\left(y\right)=h\hspace{2.8 cm} & \mbox{in} \hspace{0.10 cm}(0,T)\times\Omega\\
y=0  & \mbox{on}\hspace{0.10 cm} (0,T)\times \partial \Omega
\end{dcases}
\end{equation*}
can be written as
\begin{equation}
\left\|h\right\|_{\xLtwo\left(\left(0,T\right)\times\Omega\right)}^2=\left\|y_t\right\|_{\xLtwo((0,T)\times \Omega)}^2+\left\|-\Delta y+f\left(y\right)\right\|_{\xLtwo((0,T)\times \Omega)}^2+\mbox{remainder},
\end{equation}
where the remainder depends on the value of the solution at times $t=0$ and $t=T$.

\begin{lmm}\label{lemma_parab_L2_equality}
	Let $\Omega$ be a bounded open set of $\xR^n$, $n \in \left\{1,2,3\right\}$, with $\xCinfty$ boundary. Let $f\in \xCn{3}\left(\xR;\xR\right)$ be nondecreasing, with $f(0)=0$.
	%[CHECK], in \cref{lemma_parab_L2_equality_eq12}, $\frac{\partial \nabla y}{\partial t}$ exists in \xLtwo((0,T)\times \Omega). See \cite[Theorem 5 page 382]{PDE}
	Set $F\left(y\right)\coloneqq \int_0^y f(\xi)d\xi$. Let $y_0\in \xLinfty\left(\Omega\right)\cap \xHone_0\left(\Omega\right)$ be an initial datum and let $h\in \xLinfty((0,T)\times \Omega)$ be a source term. Let $y$ be the solution to
	\begin{equation}\label{semilinear_general_1}
	\begin{dcases}
	y_t-\Delta y+f\left(y\right)=h\hspace{2.8 cm} & \mbox{in} \hspace{0.10 cm}(0,T)\times\Omega\\
	y=0  & \mbox{on}\hspace{0.10 cm} (0,T)\times \partial \Omega\\
	y(0,x)=y_0(x)  & \mbox{in}\hspace{0.10 cm}  \Omega.\\
	\end{dcases}
	\end{equation}
	Then, the following identity holds
	\begin{align}\label{lemma_parab_L2_equality_eq1}
	\int_0^T\int_{\Omega}\left|h\right|^2dxdt&=\int_0^T\int_{\Omega}\left[\left|y_t\right|^2+\left|-\Delta y+f\left(y\right)\right|^2\right]dxdt\\
	&\;\hspace{0.33 cm}+\int_{\Omega}\left[\left\|\nabla y(T,x)\right\|^2+2F\left(y(T,x)\right)-\left\|\nabla y_0(x)\right\|^2-2F\left(y_0(x)\right)\right]dx.\nonumber
	\end{align}
\end{lmm}
\begin{proof}[Proof of Lemma \ref{lemma_parab_L2_equality}]
	%\textit{Step 1} \  \textbf{Conclusion for $y_0\in \xCinfty_{c}\left(\Omega\right)$ and $h\in \xCinfty_{c}\left((0,T)\times \Omega\right)$}\\
	We start by proving our assertion for $\xCinfty$-smooth data, with compact support. By \cref{semilinear_general_1}, we have
	\begin{align}\label{lemma_parab_L2_equality_eq2}
	\int_0^T\int_{\Omega}\left|h\right|^2dxdt&=\int_0^T\int_{\Omega}\left|y_t-\Delta y+f\left(y\right)\right|^2dxdt\nonumber\\
	&=\int_0^T\int_{\Omega}\left[\left|y_t\right|^2+\left|-\Delta y+f\left(y\right)\right|^2\right]dxdt\nonumber\\
	&\;\hspace{0.33 cm}+2\int_0^T\int_{\Omega} y_t\left[-\Delta y+f\left(y\right)\right]dxdt.
	\end{align}
	We now concentrate on the terms $2\int_0^T\int_{\Omega}y_t\left[-\Delta y\right]dxdt$ and $2\int_0^T\int_{\Omega}y_tf\left(y\right)dxdt$. Integrating by parts in space,
	%and by employing the Th(Fubini) and Th(Fundamental Calculus)
	we get
	\begin{align}\label{lemma_parab_L2_equality_eq12}
	2\int_0^T\int_{\Omega}y_t\left[-\Delta y\right]dxdt&=\int_0^T\int_{\Omega}2\frac{\partial \nabla y}{\partial t}\cdot \nabla y dxdt\nonumber\\
	&=\int_{\Omega}\left[\left\|\nabla y(T,x)\right\|^2-\left\|\nabla y_0(x)\right\|^2\right]dx.
	\end{align}
	By using the chain rule 
	%the Th(Fubini) and Th(Fundamental Calculus)
	and the definition $F\left(y\right)\coloneqq \int_0^y f(\xi)d\xi$, we have
	\begin{equation}\label{lemma_parab_L2_equality_eq14}
	\int_0^T\int_{\Omega}y_tf\left(y\right)dxdt=\int_0^T\int_{\Omega}\frac{\partial}{\partial t}\left[F\left(y\right)\right]dxdt=\int_{\Omega}\left[F\left(y(T,x)\right)-F\left(y_0(x)\right)\right]dx.
	\end{equation}
	By \cref{lemma_parab_L2_equality_eq2}, \cref{lemma_parab_L2_equality_eq12} and \cref{lemma_parab_L2_equality_eq14}, we get \cref{lemma_parab_L2_equality_eq1}.\\
	%\textit{Step 2} \  \textbf{Conclusion}\\
	The conclusion for general data follows from a density argument based on parabolic regularity (see \cite[Theorem 7.32 page 182]{lieberman1996second}, \cite[Theorem 9.1 page 341]{PEL} or \cite[Theorem 9.2.5 page 275]{wu2006elliptic}).
\end{proof}

We proceed now with the proof of Lemma \ref{lemma_functional rewritten}.

\begin{proof}[Proof of Lemma \ref{lemma_functional rewritten}]
	%\textit{Step 1} \  \textbf{Conclusion}\\
	For any control $u\in \xLtwo((0,T)\times \omega)$, by Lemma \ref{lemma_parab_L2_equality} applied to \cref{semilinear_internal_6_slt}, we have
	\begin{align*}\label{lemma_functional rewritten_eq2}
	\frac12 \int_{0}^T\int_{\omega}\left|u\right|^2dxdt&=\int_0^T\int_{\Omega}\left[\left|y_t\right|^2+\left|-\Delta y+f\left(y\right)\right|^2\right]dxdt\\
	&\;\hspace{0.33 cm}+\int_{\Omega}\left[\left\|\nabla y(T,x)\right\|^2+2F\left(y(T,x)\right)-\left\|\nabla y_0(x)\right\|^2-2F\left(y_0(x)\right)\right]dx.
	\end{align*}
	whence
	\begin{align*}
	J_{T}(u)&=\frac12 \int_0^T\int_{\Omega} \left|-\Delta y +f\left(y\right)\right|^2 dxdt+\frac{\beta}{2}\int_0^T\int_{\omega_0}\left|y-z\right|^2dxdt\nonumber\\
	&\;\hspace{0.33 cm}+\frac12 \int_0^T\int_{\Omega} \left|y_t\right|^2 dxdt\nonumber\\
	&\;\hspace{0.33 cm}+\frac12\int_{\Omega}\left[\left\|\nabla y(T,x)\right\|^2+2F\left(y(T,x)\right)-\left\|\nabla y_0(x)\right\|^2-2F\left(y_0(x)\right)\right]dx.
	\end{align*}
	By the above equality and the definition of $J_s$ \cref{semilinear_internal_elliptic_1_slt}-\cref{steady_functional_slt}, formula \cref{lemma_functional rewritten_eq1} follows.
\end{proof}

The last Lemma needed to prove \Cref{theorem_convergence_averages} is the following one.

\begin{lmm}\label{lemma_bound_value_and_statetimederivative}
	Consider the time-evolution control problem \cref{semilinear_internal_1_slt}-\cref{functional_slt} and its steady version \cref{semilinear_internal_elliptic_1_slt}-\cref{steady_functional_slt}. Assume $\omega=\Omega$. Arbitrarily fix $y_0\in \xLinfty\left(\Omega\right)\cap \xHone_0\left(\Omega\right)$ an initial datum and $z\in \xLinfty(\omega_0)$ a target. Let $u^T$ be an optimal control for \cref{semilinear_internal_1_slt}-\cref{functional_slt} and let $y^T$ be the corresponding state, solution to \cref{semilinear_internal_1_slt}, with control $u^T$ and initial datum $y_0$. Then,
	\begin{enumerate}
		\item there exists a $T$-independent constant $K$ such that
		%$K=K(\Omega,\omega,\omega_0,\beta,y_0,z,\dots)$
		\begin{equation}\label{lemma_bound_value_and_statetimederivative_eq1}
		\left|\inf_{\xLtwo((0,T)\times \Omega)} J_{T}-T\inf_{\xLtwo\left(\Omega\right)}J_s\right|\leq K;
		\end{equation}
		\item the $\xLtwo$ norm of the time derivative of the optimal state is bounded uniformly in $T$
		\begin{equation}\label{lemma_bound_value_and_statetimederivative_eq2}
		\left\|y^T_t\right\|_{\xLtwo((0,T)\times \Omega)}\leq K,
		\end{equation}
		with $K$ independent of $T>0$.
		%$K=K(\Omega,\omega,\omega_0,\beta,y_0,z,\dots)$
	\end{enumerate}
\end{lmm}
\begin{proof}[Proof of Lemma \ref{lemma_bound_value_and_statetimederivative}]
	\textit{Step 1} \  \textbf{Proof of
		\begin{equation*}
		\inf_{\xLtwo((0,T)\times \Omega)}J_{T}=J_{T}\left(u^T\right)\geq T \inf_{\xLtwo\left(\Omega\right)}J_s +\frac12 \int_0^T\int_{\Omega} \left|y_t^T(t,x)\right|^2 dxdt-\frac12\int_{\Omega}\left[\left\|\nabla y_0(x)\right\|^2+2F\left(y_0(x)\right)\right]dx.
		\end{equation*}}\\
	We start observing that, since the nonlinearity $f$ is nondecreasing and $f(0)=0$, the primitive $F$ is nonnegative
	\begin{equation}\label{lemma_bound_value_and_statetimederivative_eq3}
	F\left(y\right)\geq 0, \hspace{2.6 cm}\forall \ y\in\mathbb{R}.
	\end{equation}
	
	Let $u^T$ be an optimal control for \cref{semilinear_internal_1_slt}-\cref{functional_slt} and let $y^T$ be the corresponding state, solution to \cref{semilinear_internal_1_slt}, with control $u^T$ and initial datum $y_0$. By Lemma \ref{lemma_functional rewritten} and \cref{lemma_bound_value_and_statetimederivative_eq3}, we have
	\begin{align}\label{lemma_bound_value_and_statetimederivative_eq4}
	J_{T}\left(u^T\right)&= \int_0^T J_s\big(-\Delta y^T(t,\cdot) +f\left(y^T(t,\cdot)\right)\big) dt\nonumber\\
	&\;\hspace{0.33 cm}+\frac12 \int_0^T\int_{\Omega} \left|y_t^T(t,x)\right|^2 dxdt\nonumber\\
	&\;\hspace{0.33 cm}+\frac12\int_{\Omega}\left[\left\|\nabla y^T(T,x)\right\|^2+2F\left(y^T(T,x)\right)-\left\|\nabla y_0(x)\right\|^2-2F\left(y_0(x)\right)\right]dx\nonumber\\
	&\geq \int_0^T J_s\big(-\Delta y^T(t,\cdot) +f\left(y^T(t,\cdot)\right)\big) dt\nonumber\\
	&\;\hspace{0.33 cm}+\frac12 \int_0^T\int_{\Omega} \left|y_t^T(t,x)\right|^2 dxdt\nonumber\\
	&\;\hspace{0.33 cm}-\frac12\int_{\Omega}\left[\left\|\nabla y_0(x)\right\|^2+2F\left(y_0(x)\right)\right]dx.\nonumber\\
	%	&\geq \int_0^T J_s\big(-\Delta y^T(t,\cdot) +f\left(y^T(t,\cdot)\right)\big) dt\nonumber\\
	%	&\;\hspace{0.33 cm}+\frac12 \int_0^T\int_{\Omega} \left|y_t^T(t,x)\right|^2 dxdt-K,\nonumber\\
	\end{align}
	Now, for a.e. $t\in (0,T)$, by definition of infimum
	\begin{equation*}
	J_s\big(-\Delta y^T(t,\cdot) +f\left(y^T(t,\cdot)\right)\big)\geq \inf_{\xLtwo\left(\Omega\right)}J_s.
	\end{equation*}
	The above inequality and \cref{lemma_bound_value_and_statetimederivative_eq4} yield
	\begin{align*}
	J_{T}\left(u^T\right)&\geq \int_0^T J_s\big(-\Delta y^T(t,\cdot) +f\left(y^T(t,\cdot)\right)\big) dt\nonumber\\
	&\;\hspace{0.33 cm}+\frac12 \int_0^T\int_{\Omega} \left|y_t^T(t,x)\right|^2 dxdt-\frac12\int_{\Omega}\left[\left\|\nabla y_0(x)\right\|^2+2F\left(y_0(x)\right)\right]dx\nonumber\\
	&\geq \int_0^T \left[\inf_{\xLtwo\left(\Omega\right)}J_s\right] dt+\frac12 \int_0^T\int_{\Omega} \left|y_t^T(t,x)\right|^2 dxdt-\frac12\int_{\Omega}\left[\left\|\nabla y_0(x)\right\|^2+2F\left(y_0(x)\right)\right]dx\nonumber\\
	&= T \inf_{\xLtwo\left(\Omega\right)}J_s +\frac12 \int_0^T\int_{\Omega} \left|y_t^T(t,x)\right|^2 dxdt-\frac12\int_{\Omega}\left[\left\|\nabla y_0(x)\right\|^2+2F\left(y_0(x)\right)\right]dx,\nonumber\\
	\end{align*}
	whence
	\begin{equation}\label{lemma_bound_value_and_statetimederivative_eq6}
	\inf_{\xLtwo((0,T)\times \Omega)}J_{T}=J_{T}\left(u^T\right)\geq T \inf_{\xLtwo\left(\Omega\right)}J_s +\frac12 \int_0^T\int_{\Omega} \left|y_t^T(t,x)\right|^2 dxdt-\frac12\int_{\Omega}\left[\left\|\nabla y_0(x)\right\|^2+2F\left(y_0(x)\right)\right]dx.
	\end{equation}
	\textit{Step 2} \  \textbf{Conclusion}\\
	On the one hand, by Lemma \ref{lemma_upper_bound_value}, we have
	\begin{equation}\label{lemma_bound_value_and_statetimederivative_eq7}
	\inf_{\xLtwo((0,T)\times \Omega)} J_{T}-T\inf_{\xLtwo\left(\Omega\right)}J_s\leq  K,
	\end{equation}
	the constant $K$ being independent of $T>0$.
	%$K=K(\Omega,\omega,\omega_0,\beta,y_0,z,\dots)$
	On the other hand, by \cref{lemma_bound_value_and_statetimederivative_eq6}, we get
	\begin{equation}\label{lemma_bound_value_and_statetimederivative_eq8}
	\inf_{\xLtwo((0,T)\times \Omega)} J_{T}-T\inf_{\xLtwo\left(\Omega\right)}J_s\geq  -K.
	\end{equation}
	%$K=K(\Omega,\omega,\omega_0,\beta,y_0,z,\dots)$
	By \cref{lemma_bound_value_and_statetimederivative_eq7} and \cref{lemma_bound_value_and_statetimederivative_eq8}, inequality \cref{lemma_bound_value_and_statetimederivative_eq1} follows.
	
	It remains to prove \cref{lemma_bound_value_and_statetimederivative_eq2}. By \cref{lemma_bound_value_and_statetimederivative_eq6} and Lemma \ref{lemma_upper_bound_value}, we have
	\begin{equation*}
	T \inf_{\xLtwo\left(\Omega\right)}J_s +\frac12 \int_0^T\int_{\Omega} \left|y_t^T(t,x)\right|^2 dxdt-K\leq \inf_{\xLtwo((0,T)\times \Omega)}J_{T}\leq T\inf_{\xLtwo\left(\Omega\right)}J_s + K,
	\end{equation*}
	whence
	\begin{equation*}
	\frac12 \int_0^T\int_{\Omega} \left|y_t^T(t,x)\right|^2 dxdt\leq K,
	\end{equation*}
	as required.
\end{proof}

We are now ready to prove \Cref{theorem_convergence_averages}.

\begin{proof}[Proof of \Cref{theorem_convergence_averages}]
	%\textit{Step 1} \  \textbf{Conclusion}\\
	Estimate \cref{theorem_convergence_averages_eq2} follows directly from Lemma \ref{lemma_bound_value_and_statetimederivative} (2.).
	
	It remains to prove the convergence of the averages. By the regularizing effect of the state equation \cref{semilinear_internal_1_slt} and Lemma \ref{lemma_bound_optima}, we can reduce to the case of initial datum $y_0\in \xLinfty\left(\Omega\right)\cap \xHone_0\left(\Omega\right)$. By Lemma \ref{lemma_bound_value_and_statetimederivative}, we have
	\begin{equation}\label{theorem_convergence_averages_eq3}
	\left|\inf_{\xLtwo((0,T)\times \Omega)} J_{T}-T\inf_{\xLtwo\left(\Omega\right)}J_s\right|\leq K.
	\end{equation}
	Then,
	\begin{align*}
	\left|\frac{1}{T}\inf_{\xLtwo((0,T)\times \Omega)} J_{T}-\inf_{\xLtwo\left(\Omega\right)}J_s\right|&=\frac{1}{T}\left|\inf_{\xLtwo((0,T)\times \Omega)} J_{T}-T\inf_{\xLtwo\left(\Omega\right)}J_s\right|\nonumber\\
	&\leq \frac{K}{T} \underset{T\to +\infty}{\longrightarrow}0,
	\end{align*}
	as required.
\end{proof}

\section{Numerical simulations}
\label{sec:5.4}

This section is devoted to a numerical illustration of \Cref{th_slt_1}. Our goal is to check that the turnpike property is fulfilled for small target, regardless of the size of the initial datum.

We deal with the optimal control problem
\begin{equation*}\label{functional_example}
\min_{u\in \xLtwo((0,T)\times (0,\frac12))}J_{T}(u)=\frac12 \int_0^T\int_{0}^{\frac12} |u|^2 dxdt+\frac{\beta}{2}\int_0^T\int_{0}^1 |y-z|^2 dxdt,
\end{equation*}
where:
\begin{equation*}\label{semilinear_internal_1_example}
\begin{dcases}
y_t-y_{xx}+y^3=u\chi_{(0,\frac12)}\hspace{2.8 cm} & (t,x)\in (0,T)\times (0,1)\\
y(t,0)=y(t,1)=0  & t\in (0,T)\\
y(0,x)=y_0(x)  & x\in  (0,1).
\end{dcases}
\end{equation*}
We choose as initial datum $y_0\equiv 10$, as weighting parameter $\beta = 1000$ and as target $z\equiv 1$.

We solve the above semilinear heat equation by using the semi-implicit method:
\begin{equation*}
\begin{dcases}
\frac{Y_{i+1}-Y_{i}}{\Delta t}-\Delta Y_{i+1}+Y_{i}^3=U_{i}\chi_{(0,\frac12)}\hspace{0.6 cm} &  i=0,\dots,N_t-1\\
Y_{0}=y_0, &
\end{dcases}
\end{equation*}
where $Y_i$ and $U_i$ denote resp. a time discretization of the state and the control.

The optimal control is determined by a Gradient Descent method, with constant stepsize. The optimal state is depicted in \cref{yopty010}.

\begin{figure}[htp]
	\begin{center}
		% replace aims_logo.eps by your figure file name
		\centerline{\includegraphics[width=17 cm]{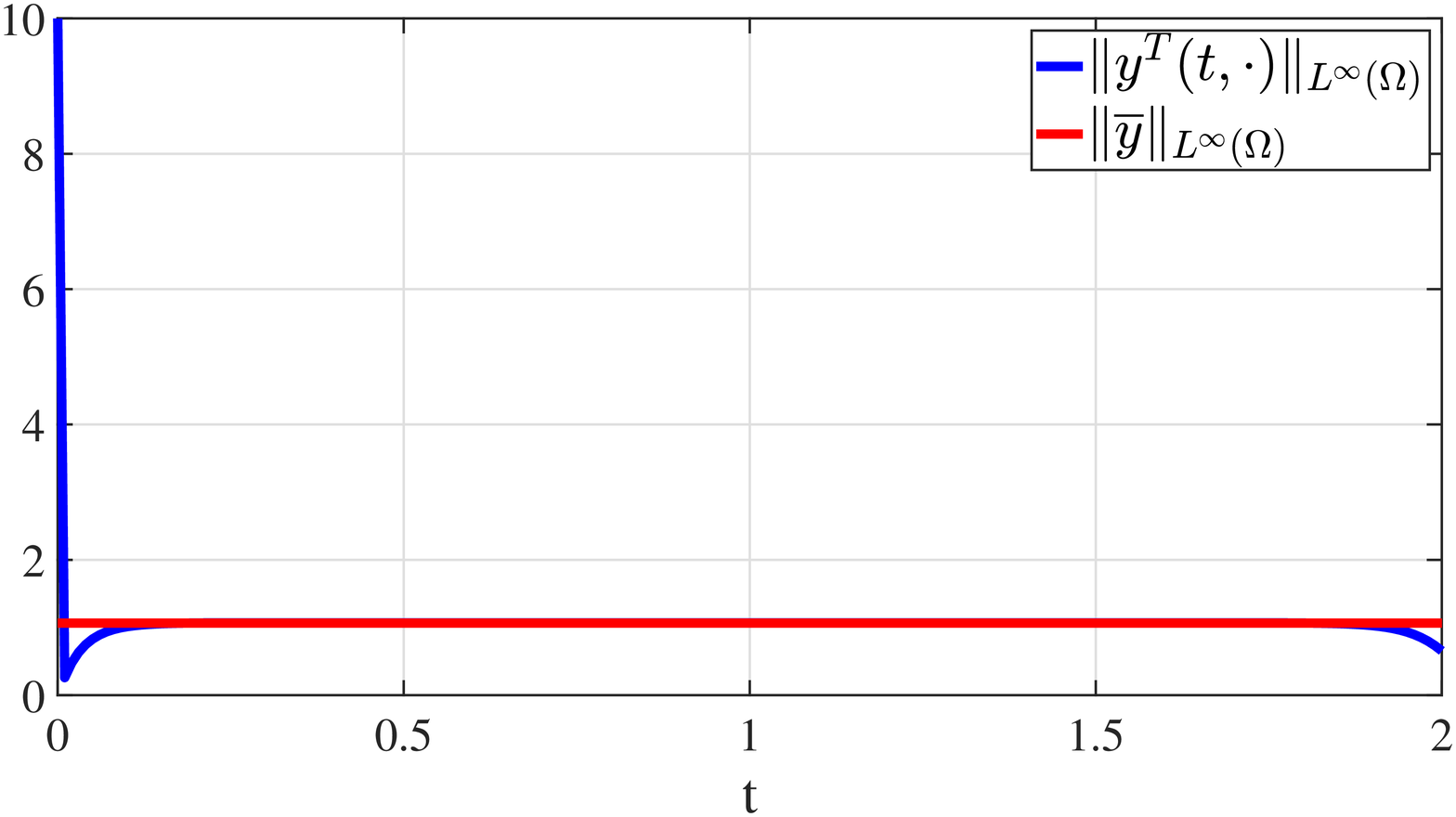}}
		\caption{graph of the function $t\longrightarrow \|y^T(t)\|_{\xLinfty\left(\Omega\right)}$ (in blue) and $\|\overline{y}\|_{\xLinfty\left(\Omega\right)}$ (in red), where $y^T$ denotes an optimal state, whereas $\overline{y}$ stands for an optimal steady state.}\label{yopty010}
	\end{center}
\end{figure}

\begin{figure}[htp]
	\begin{center}
		% replace aims_logo.eps by your figure file name
		\centerline{\includegraphics[width=17 cm]{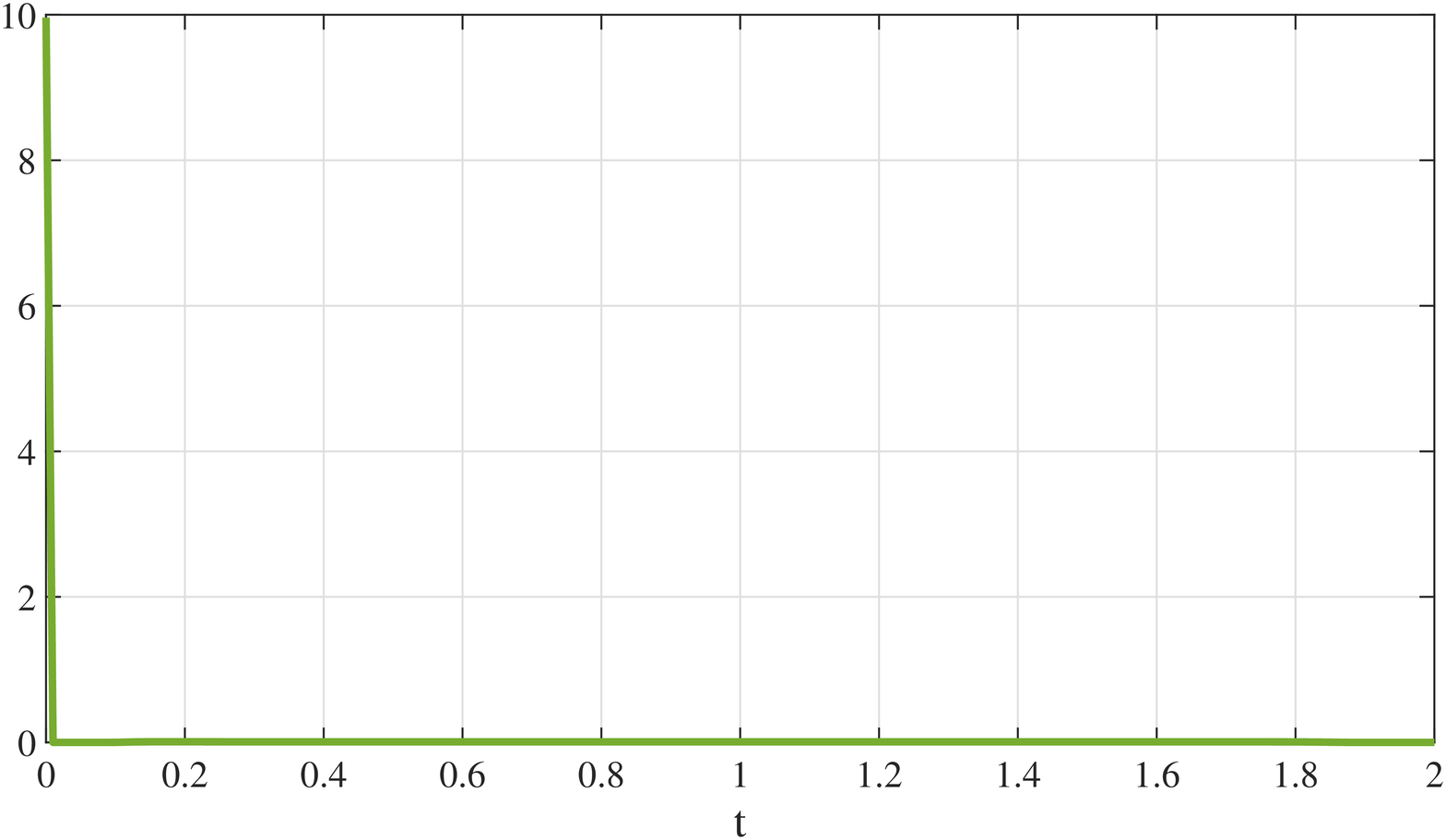}}
		\caption{graph of the function $t\longrightarrow \left\|y^T(t)-\overline{y}\right\|_{\xLinfty\left(\Omega\right)}$, where $y^T$ denotes an optimal state, whereas $\overline{y}$ stands for an optimal steady state.}\label{evostatesteadystate}
	\end{center}
\end{figure}

\section{Conclusions and open problems}
\label{sec:5.5}

In this manuscript we have obtained some global turnpike results for an optimal control problem governed by a nonlinear state equation. For any data and small targets, we have shown that the exponential turnpike property holds (\Cref{th_slt_1}). For arbitrary targets, we have proved the convergence of averages (\Cref{theorem_convergence_averages}), under the added assumption of controlling everywhere. One of the main tools employed for our analysis is an $\xLinfty$ bound of the norm of the optima, uniform in the time horizon (Lemma \ref{lemma_bound_optima}). Numerical simulation have been performed, which confirms the theoretical results.

We present now some interesting open problems in the field.

\subsection{General targets with any control domain}
\label{subsec:5.5.1}

In \Cref{theorem_convergence_averages} we have proved the convergence of averages for large targets, in the context of control everywhere. An interesting challenge is to prove the exponential turnpike property, even if the control is local (namely $\omega\subsetneq \Omega$). The challenge is to prove the following conjecture.

\begin{cnjctr}\label{conj_turnpike}
	Consider the control problem \cref{semilinear_internal_1_slt}-\cref{functional_slt}. Take any initial datum $y_0\in \xLinfty\left(\Omega\right)$ and any target $z\in \xLinfty(\omega_0)$. Let $u^T$ be a minimizer of \cref{functional_slt}. There exists an optimal pair $(\overline{u},\overline{y})$ for \cref{semilinear_internal_elliptic_1_slt}-\cref{steady_functional_slt} such that
	\begin{equation}
	\|u^T(t)-\overline{u}\|_{\xLinfty\left(\omega\right)}+\|y^T(t)-\overline{y}\|_{\xLinfty\left(\Omega\right)}\leq K\left[\exp\left(-\mu t\right)+\exp\left(-\mu(T-t)\right)\right],\hspace{0.6 cm}\forall t\in [0,T],
	\end{equation}
	the constants $K$ and $\mu>0$ being independent of the time horizon $T$.
	%$\mu=\mu(\Omega,\omega,\omega_0,\dots,\dots, )$
	%$K=K(\Omega,\omega,\omega_0,y_0,\dots,\dots, )$.
\end{cnjctr}

In \cite{pighin2020nonuniqueness} special large targets $z$ are constructed, such that the optimal control for the steady problem \cref{semilinear_internal_elliptic_1_slt}-\cref{steady_functional_slt} is not unique. For those targets, a question arises: if the turnpike property is satisfied, which minimizer for \cref{semilinear_internal_elliptic_1_slt}-\cref{steady_functional_slt} attracts the optimal solutions to \cref{semilinear_internal_1_slt}-\cref{functional_slt}?\\
Note that, in the context of internal control, the counterexample to uniqueness in \cite{pighin2020nonuniqueness} is valid in case of local control $\omega\subsetneq \Omega$.

Generally speaking a further investigation is required for the linearized optimality system determined in \cite[subsection 3.1]{PZ2}. We introduce the problem. As in \cref{semilinear_internal_parabolic_1}, consider the optimality system for \cref{semilinear_internal_1_slt}-\cref{functional_slt}
\begin{equation}\label{semilinear_internal_parabolic_1_intro}
\begin{cases}
y^T_t-\Delta y^T+f(y^T)=-q^T\chi_{\omega}\hspace{2.8 cm} & \mbox{in} \hspace{0.10 cm}(0,T)\times\Omega\\
y^T=0  & \mbox{on}\hspace{0.10 cm} (0,T)\times \partial \Omega\\
y^T(0,x)=y_0(x)  & \mbox{in}\hspace{0.10 cm}  \Omega\\
-q^T_t-\Delta q^T+f^{\prime}(y^T)q^T=\beta(y^T-z)\chi_{\omega_0}\hspace{2.8 cm} & \mbox{in} \hspace{0.10 cm}(0,T)\times \Omega\\
q^T=0  & \mbox{on}\hspace{0.10 cm} (0,T)\times\partial \Omega\\
q^T(T,x)=0 & \mbox{in} \hspace{0.10 cm}\Omega.\\
\end{cases}
\end{equation}

Pick any optimal pair $(\overline{u},\overline{y})$ for \cref{semilinear_internal_elliptic_1_slt}-\cref{steady_functional_slt}. By the first order optimality conditions, the steady optimal control reads as $\overline{u}=-\overline{q}\chi_{\omega}$, with
\begin{equation}\label{semilinear_internal_elliptic_2_intro}
\begin{cases}
-\Delta \overline{y}+f(\overline{y})=-\overline{q}\chi_{\omega}\hspace{2.8 cm} & \mbox{in} \hspace{0.10 cm}\Omega\\
\overline{y}=0  & \mbox{on}\hspace{0.10 cm} \partial \Omega\\
-\Delta \overline{q}+f^{\prime}(\overline{y})\overline{q}=\beta(\overline{y}-z)\chi_{\omega_0}\hspace{2.8 cm} & \mbox{in} \hspace{0.10 cm}\Omega\\
\overline{q}=0  & \mbox{on}\hspace{0.10 cm} \partial \Omega.
\end{cases}
\end{equation}

As in \cite{PZ2}, we introduce the perturbation variables
\begin{equation}
\eta^T\coloneqq y^T-\overline{y}\hspace{0.3 cm}\mbox{and}\hspace{0.3 cm}\varphi^T\coloneqq q^T-\overline{q}
\end{equation}
and we write down the linearized optimality system around $(\overline{u},\overline{y})$
\begin{equation}\label{semilinear_internal_parabolic_linearized_1_intro}
\begin{cases}
\eta^T_t-\Delta \eta^T+f^{\prime}(\overline{y})\eta^T=-\varphi^T\chi_{\omega}\hspace{2.8 cm} & \mbox{in} \hspace{0.10 cm}(0,T)\times\Omega\\
\eta^T=0  & \mbox{on}\hspace{0.10 cm} (0,T)\times \partial \Omega\\
\eta^T(0,x)=y_0(x)-\overline{y}(x)  & \mbox{in}\hspace{0.10 cm}  \Omega\\
-\varphi^T_t-\Delta \varphi^T+f^{\prime}(\overline{y})\varphi^T=\left(\beta\chi_{\omega_0}-f^{\prime\prime}\left(\overline{y}\right)\overline{q}\right)\eta^T\hspace{2.8 cm} & \mbox{in} \hspace{0.10 cm}(0,T)\times \Omega\\
\varphi^T=0  & \mbox{on}\hspace{0.10 cm} (0,T)\times\partial \Omega\\
\varphi^T(T,x)=-\overline{q}(x) & \mbox{in} \hspace{0.10 cm}\Omega.\\
\end{cases}
\end{equation}

As pointed out in \cite[Theorem 1 in subsection 3.1]{PZ2}, a key point is to check the validity of the turnpike property for the linearized optimality system \cref{semilinear_internal_parabolic_linearized_1_intro}. This is complicated because of the term $\beta\chi_{\omega_0}-f^{\prime\prime}\left(\overline{y}\right)\overline{q}$, whose sign is unknown for general large targets. Furthermore, in case of nonuniqueness of steady optimum, it would be interesting to compute the spectrum of the linearized system around any steady optima to check if among them one is a better attractor.

We conclude this subsection observing that, even in case the control acts everywhere ($\omega=\Omega$), theory is not conclusive. Indeed, our results (Theorem \ref{theorem_convergence_averages} and Lemma \ref{lemma_bound_value_and_statetimederivative}) provides information on the performances of the steady controls and the estimate of the $\xLtwo$ norm of the time derivative of the optimal state. The proof of Conjecture \ref{conj_turnpike} in this case would require the use of Hamilton-Jacobi techniques (see e.g. \cite[Theorem 7.4.17]{cannarsa2004semiconcave}) to carry over our results to the optimal control and states by a feedback operator.

\subsection{Different nonlinear state equations}
\label{subsec:5.5.2}

It would be interesting to check the validity of the turnpike property, for different state equations, e.g. hyperbolic PDEs. This has been done in the linear case \cite{porretta2013long,zuazua2017large,grune2019sensitivity}. To address the nonlinear case the scheme we have employed can be used (uniform estimates for the optima, linearization and global-local argument). However, appropriate modifications have to be made to the proofs according to regularity properties of the state equation.

\appendix

%\subsection*{Regularity for the state equation}
%
%Our purpose is now to investigate the regularity of the state equation \cref{semilinear_internal_1_slt}.
%
%We have the following result.

%The first part of the Appendix is devoted to some regularity results for parabolic equations. The second part is devoted to the proof of Lemma \ref{lemma_bound_optima}.

\section{Parabolic regularity results}
\label{appendixsec:Parabolic regularity results}

One of the key tool to carry on the proof of Lemma \ref{lemma_bound_optima} is the following regularity result.

\begin{lmm}\label{lemma_reg_nonnegativepotential}
	Let $\Omega$ be a bounded open set of $\xR^n$, $n \in \left\{1,2,3\right\}$, with $\xCtwo$ boundary. Let $c:(0,T)\times \Omega\longrightarrow \mathbb{R}$ be measurable and nonnegative. Let $y_0\in \xLinfty\left(\Omega\right)$ be an initial datum and let $h\in \xLtwo((0,T)\times \Omega)$ be a source term. Let $y$ be the solution to
	\begin{equation*}
	\begin{dcases}
	y_t-\Delta y+cy=h\hspace{2.8 cm} & \mbox{in} \hspace{0.10 cm}(0,T)\times\Omega\\
	y=0  & \mbox{on}\hspace{0.10 cm} (0,T)\times \partial \Omega\\
	y(0,x)=y_0(x)  & \mbox{in}\hspace{0.10 cm}  \Omega.\\
	\end{dcases}
	\end{equation*}
	Choose $y_0$ and $h$ so that $cy\in \xLtwo\left(\left(0,T\right)\times \Omega\right)$. Then, $y\in \xLtwo\left((0,T); \xLinfty\left(\Omega\right)\right)$ and we have
	\begin{equation}\label{estimate_norm_L2Linf}
	\|y\|_{\xLtwo\left((0,T); \xLinfty\left(\Omega\right)\right)}\leq K\left[\|y_0\|_{\xLinfty\left(\Omega\right)}+\|h\|_{\xLtwo\left(\left(0,T\right)\times\Omega\right)}\right],
	\end{equation}
	where $K=K\left(\Omega\right)$.
	%$K=K(\Omega,\dots,\dots,  )$
\end{lmm}
\begin{proof}[Proof of Lemma \ref{lemma_reg_nonnegativepotential}.]
	\textit{Step 1} \ \textbf{Comparison}\\
	Let $\psi$ be the solution to:
	\begin{equation}\label{lemma_reg_nonnegativepotential_proof_eq1}
	\begin{dcases}
	\psi_t-\Delta \psi =|h|\hspace{0.6 cm} & \mbox{in} \hspace{0.10 cm}(0,T)\times \Omega\\
	\psi=0  & \mbox{on}\hspace{0.10 cm} (0,T)\times \partial \Omega\\
	\psi(0,x)=|y_0|.  & \mbox{in} \hspace{0.10 cm}\Omega\\
	\end{dcases}
	\end{equation}
	Since $c\geq 0$, a.e. in $(0,T)\times \Omega$, by a comparison argument, for each $t\in [0,T]$:
	%Which is proved by energy estimates,
	\begin{equation}\label{lemma_reg_nonnegativepotential_proof_eq2}
	|y(t,x)|\leq \psi(t,x), \quad \mbox{a.e. }x \in \Omega.
	\end{equation}
	Now, since $y_0$ and $h$ are bounded, again by comparison principle applied to \cref{lemma_reg_nonnegativepotential_proof_eq1}, $\psi$ is bounded. Hence, by \cref{lemma_reg_nonnegativepotential_proof_eq2}, $y$ is bounded as well and
	\begin{equation}\label{lemma_reg_nonnegativepotential_proof_eq3}
	\int_0^T\|y(t)\|_{\xLinfty\left(\Omega\right)}^2dt\leq \int_0^T\|\psi(t)\|_{\xLinfty\left(\Omega\right)}^2 dt.
	\end{equation}
	Then, to conclude it suffices to show
	\begin{equation*}
	\|\psi\|_{\xLtwo(0,T;\xLinfty\left(\Omega\right))}\leq K\left[\|y_0\|_{\xLinfty\left(\Omega\right)}+\|h\|_{\xLtwo((0,T)\times \Omega)}\right],
	\end{equation*}
	the constant $K$ being independent of $T$.\\
	%$K=K(\Omega,\dots,\dots,  )$
	\textit{Step 2} \ \textbf{Splitting}\\
	Split $\psi=\xi+\chi$, where $\xi$ solves:
	\begin{equation}\label{lemma_reg_nonnegativepotential_proof_eq4}
	\begin{dcases}
	\xi_t-\Delta \xi =0\hspace{0.6 cm} & \mbox{in} \hspace{0.10 cm}(0,T)\times \Omega\\
	\xi=0  & \mbox{on}\hspace{0.10 cm} (0,T)\times \partial \Omega\\
	\xi(0,x)=|y_0|  & \mbox{in} \hspace{0.10 cm}\Omega\\
	\end{dcases}
	\end{equation}
	while $\chi$ satisfies:
	\begin{equation}\label{lemma_reg_nonnegativepotential_proof_eq5}
	\begin{dcases}
	\chi_t-\Delta \chi=|h|\hspace{0.6 cm} & \mbox{in} \hspace{0.10 cm}(0,T)\times \Omega\\
	\chi=0  & \mbox{on}\hspace{0.10 cm} (0,T)\times \partial \Omega\\
	\chi(0,x)=0  & \mbox{in} \hspace{0.10 cm}\Omega.\\
	\end{dcases}
	\end{equation}
	
	First of all, we prove an estimate like \cref{estimate_norm_L2Linf} for $\xi$. We start by employing maximum principle (see \cite{MPD}) to \cref{lemma_reg_nonnegativepotential_proof_eq3}, getting
	\begin{equation}\label{lemma_reg_nonnegativepotential_proof_eq5.5}
	\|\xi\|_{\xLinfty\left((0,T)\times \Omega\right)}\leq \|y_0\|_{\xLinfty\left(\Omega\right)}.
	\end{equation}
	Now, if $T\geq 1$, by the regularizing effect and the exponential stability of the heat equation, for any $t\in [1,T]$, we have
	\begin{equation}\label{lemma_reg_nonnegativepotential_proof_eq6}
	\|\xi(t)\|_{\xLinfty\left(\Omega\right)}\leq K\|\xi(t-1)\|_{\xLtwo\left(\Omega\right)}\leq K \exp\left(-\lambda_1 (t-1)\right)\|y_0\|_{\xLtwo\left(\Omega\right)},
	\end{equation}
	the constant $K$ depending only on the domain $\Omega$.
	Then, by \cref{lemma_reg_nonnegativepotential_proof_eq5.5} and \cref{lemma_reg_nonnegativepotential_proof_eq6}, for any $T>0$, for every $t\in [0,T]$,
	\begin{equation}\label{lemma_reg_nonnegativepotential_proof_eq6.5}
	\|\xi(t)\|_{\xLinfty\left(\Omega\right)}\leq K\min\left\{1,\exp\left(-\lambda_1 (t-1)\right)\right\}\|y_0\|_{\xLinfty\left(\Omega\right)},
	\end{equation}
	with $K=K\left(\Omega\right)$.
	
	Now, we focus on \cref{lemma_reg_nonnegativepotential_proof_eq5}. By parabolic regularity (see e.g. \cite[Theorem 5 subsection 7.1.3]{PDE}), $\chi \in \xLtwo(0,T;\xHtwo\left(\Omega\right))$, with $\chi_t \in \xLtwo((0,T)\times \Omega)$. Then, by multiplying \cref{lemma_reg_nonnegativepotential_proof_eq5} by $-\Delta \chi$ and integrating over $[0,T]\times \Omega$, we obtain
	\begin{equation*}
	\frac12\|\nabla \chi(T)\|_{\xLtwo\left(\Omega\right)}^2+\int_0^T\int_{\Omega}|\Delta \chi|^2dxdt\leq \|h\|_{\xLtwo\left((0,T)\times \Omega\right)}\|\Delta \chi\|_{\xLtwo\left((0,T)\times \Omega\right)}.
	\end{equation*}
	By Young's Inequality,
	\begin{equation*}
	\int_0^T\int_{\Omega}|\Delta \chi|^2dxdt\leq \frac12\|h\|_{\xLtwo\left((0,T)\times \Omega\right)}^2+\frac12\|\Delta \chi\|_{\xLtwo\left((0,T)\times \Omega\right)}^2,
	\end{equation*}
	which leads to
	\begin{equation*}
	\int_0^T\int_{\Omega}|\Delta \chi|^2dxdt\leq \|h\|_{\xLtwo\left((0,T)\times \Omega\right)}^2.
	\end{equation*}
	Now, by \cite[Theorem 6 subsection 5.6.3]{PDE} and \cite[Theorem 4 subsection 6.3.2]{PDE},
	\begin{equation}\label{lemma_reg_nonnegativepotential_proof_eq7}
	\int_0^T\|\chi\|_{\xLinfty\left(\Omega\right)}^2dt\leq K\int_0^T\|\chi\|_{\xHtwo\left(\Omega\right)}^2dt\leq K\int_0^T\int_{\Omega}|\Delta \chi|^2dxdt\leq K\|h\|_{\xLtwo\left((0,T)\times \Omega\right)}^2.		
	\end{equation}
	
	Finally, by \cref{lemma_reg_nonnegativepotential_proof_eq3}, \cref{lemma_reg_nonnegativepotential_proof_eq6.5} and \cref{lemma_reg_nonnegativepotential_proof_eq7},
	\begin{equation*}
	\int_0^T\|y\|_{\xLinfty\left(\Omega\right)}^2dt\leq 2\int_0^T\|\xi\|_{\xLinfty\left(\Omega\right)}^2dt+2 K\int_0^T\|\chi\|_{\xLinfty\left(\Omega\right)}^2dt\leq  K\left[\|y_0\|_{\xLinfty\left(\Omega\right)}^2+\|h\|_{\xLtwo\left((0,T)\times \Omega\right)}^2\right],
	\end{equation*}
	as required.
\end{proof}

The following regularity result is employed in the proof of Lemma \ref{lemma_bound_optima}.

\begin{lmm}\label{lemma_subinterval_Linf}
	Let $\Omega\subset\xR^n$ be a bounded open set, with $\partial \Omega\in \xCinfty$. Let $c\in \xLinfty((0,T)\times \Omega)$ be nonnegative. Let $y_0\in \xLinfty\left(\Omega\right)$ an initial datum and $h\in \xLinfty((0,T)\times \Omega)$ a source term. Let $\overline{T}\in (0,T)$  and set $N\coloneqq \left\lfloor{T/\overline{T}}\right\rfloor$. Let $y$ be the solution to
	\begin{equation*}
	\begin{dcases}
	y_t-\Delta y+cy=h\hspace{2.8 cm} & \mbox{in} \hspace{0.10 cm}(0,T)\times\Omega\\
	y=0  & \mbox{on}\hspace{0.10 cm} (0,T)\times \partial \Omega\\
	y(0,x)=y_0(x)  & \mbox{in}\hspace{0.10 cm}  \Omega.\\
	\end{dcases}
	\end{equation*}
	Then, $y\in \xLinfty((0,T)\times\Omega)$ and we have
	\begin{equation}\label{estimate_subinterval_Linf}
	\|y\|_{\xLinfty((0,T)\times\Omega)}\leq K\left[\|y_0\|_{\xLinfty\left(\Omega\right)}+\max_{i=1,\dots,N}\|h\|_{\xLtwo(((i-1)\overline{T},i\overline{T});\xLinfty\left(\Omega\right))}+\|h\|_{\xLtwo(N\overline{T},T;\xLinfty\left(\Omega\right))}\right],
	\end{equation}
	where $K=K\left(\Omega, \overline{T}\right)$ is independent of the potential $c\geq 0$ and the time horizon $T$.
\end{lmm}
\begin{proof}[Proof of Lemma \ref{lemma_subinterval_Linf}]
	\textit{Step 1} \ \textbf{Comparison argument}\\
	Let $\psi$ be the solution to:
	\begin{equation}\label{lemma_subinterval_Linf_proof_eq1}
	\begin{dcases}
	\psi_t-\Delta \psi =|h|\hspace{0.6 cm} & \mbox{in} \hspace{0.10 cm}(0,T)\times \Omega\\
	\psi=0  & \mbox{on}\hspace{0.10 cm} (0,T)\times \partial \Omega\\
	\psi(0,x)=|y_0|.  & \mbox{in} \hspace{0.10 cm}\Omega\\
	\end{dcases}
	\end{equation}
	Since $c\geq 0$, a.e. in $(0,T)\times \Omega$, by a comparison argument, for each $t\in [0,T]$:
	%Which is proved by energy estimates,
	\begin{equation}\label{lemma_subinterval_Linf_proof_eq2}
	|y(t,x)|\leq \psi(t,x), \quad \mbox{a.e. }x\in  \Omega.
	\end{equation}
	Now, since $y_0$ and $h$ are bounded, again by comparison principle applied to \cref{lemma_subinterval_Linf_proof_eq1}, $\psi$ is bounded. Hence, by \cref{lemma_subinterval_Linf_proof_eq2}, $y$ is bounded as well and
	\begin{equation}\label{lemma_subinterval_Linf_proof_eq3}
	\|y\|_{\xLinfty((0,T)\times\Omega)}\leq \|\psi\|_{\xLinfty((0,T)\times\Omega)}.
	\end{equation}
	Then, to conclude it suffices to show
	\begin{equation*}
	\|\psi\|_{\xLinfty((0,T)\times\Omega)}\leq K\left[\|y_0\|_{\xLinfty\left(\Omega\right)}+\max_{i=1,\dots,N}\|h\|_{\xLtwo(((i-1)\overline{T},i\overline{T});\xLinfty\left(\Omega\right))}+\|h\|_{\xLtwo(N\overline{T},T;\xLinfty\left(\Omega\right))}\right],
	\end{equation*}
	the constant $K=K\left(\Omega, \overline{T}\right)$ being independent of $T$.\\
	%$K=K(\Omega, \overline{T}, \dots,\dots,  )$
	\textit{Step 2} \  \textbf{Conclusion}\\
	Let $\left\{S(t)\right\}_{t\in\xR^+}$ be the heat semigroup on $\Omega$, with zero Dirichlet boundary conditions. Fix $\varepsilon \in (0,\overline{T})$. By the regularizing effect of the heat equation (see, e.g. \cite[Theorem 10.1, section 10.1]{BFP}), for any $t\geq \varepsilon$,
	\begin{equation*}
	\|S(t)y_0\|_{\xLinfty\left(\Omega\right)}\leq K\exp(-\mu(t-\varepsilon))\|y_0\|_{\xLtwo\left(\Omega\right)}\leq K\exp(-\mu(t-\varepsilon))\|y_0\|_{\xLinfty\left(\Omega\right)}.
	\end{equation*}
	For $t\in [0,\varepsilon]$, by comparison principle, we have
	\begin{equation*}
	\|S(t)y_0\|_{\xLinfty\left(\Omega\right)}\leq K\|y_0\|_{\xLinfty\left(\Omega\right)}\leq K\exp(-\mu(t-\varepsilon))\|y_0\|_{\xLinfty\left(\Omega\right)},
	\end{equation*}
	being $\exp(-\mu(t-\varepsilon))\geq 1$. Hence, for any $t\geq 0$,
	\begin{equation}\label{Linf_Linf_est}
	\|S(t)y_0\|_{\xLinfty\left(\Omega\right)}\leq K\exp(-\mu(t-\varepsilon))\|y_0\|_{\xLinfty\left(\Omega\right)}.
	\end{equation}
	Then, by the Duhamel formula, for any $t\in [0,T]$, we have
	\begin{equation}\label{repr_formula}
	\psi(t)=S(t)\left(|y_0|\right)+\int_0^tS(t-s)\left(|h(s)|\right)ds.
	\end{equation}
	Now, by \cref{Linf_Linf_est}, for any $t\geq 0$,
	\begin{equation}\label{ineq_init_dat}
	\|S(t)\left(|y_0|\right)\|_{\xLinfty\left(\Omega\right)}\leq K\exp\left(-\mu(t-\varepsilon)\right)\|y_0\|_{\xLinfty\left(\Omega\right)}.
	\end{equation}
	Besides, by applying \cref{Linf_Linf_est} to the integral term $\eta(t)\coloneqq \int_0^tS(t-s)\left(|h(s)|\right)ds$ in \cref{repr_formula}, we obtain
	\begin{eqnarray}\label{ineq_source_term}
	\|\eta(t)\|_{\xLinfty}&\leq&\int_0^t\|S(t-s)\left(|h(s)|\right)\|_{\xLinfty}ds\nonumber\\
	&\leq&K\int_0^t\exp(-\mu(t-s-\varepsilon))\|h(s)\|_{\xLinfty}ds\nonumber\\
	&\leq&K\left[\sum_{i=1}^{\left\lfloor{\frac{t}{\overline{T}}}\right\rfloor}\exp(-\mu(t-\varepsilon-i\overline{T}))\int_{(i-1)\overline{T}}^{i\overline{T}}\exp(-\mu(i\overline{T}-s))\|h(s)\|_{\xLinfty}ds\right.\nonumber\\
	&\;&\left.+\int_{\left(\left\lfloor{\frac{t}{\overline{T}}}\right\rfloor-1\right)\overline{T}}^{t}\exp(-\mu(t-s-\varepsilon))\|h(s)\|_{\xLinfty}ds\right]\nonumber\\
	&\leq&K\left\{\sum_{i=1}^{\left\lfloor{\frac{t}{\overline{T}}}\right\rfloor}\exp(-\mu(t-\varepsilon-i\overline{T}))\left[\int_{(i-1)\overline{T}}^{i\overline{T}}\exp(-2\mu(i\overline{T}-s))ds\right]^{\frac12}\left[\int_{(i-1)\overline{T}}^{i\overline{T}}\|h(s)\|_{\xLinfty}^2ds\right]^{\frac12}\right.\nonumber\\
	&\;&\left.+\left[\int_{\left(\left\lfloor{\frac{t}{\overline{T}}}\right\rfloor-1\right)\overline{T}}^{t}\exp(-2\mu(t-s-\varepsilon))ds\right]^{\frac12}\left[\int_{\left(\left\lfloor{\frac{t}{\overline{T}}}\right\rfloor-1\right)\overline{T}}^{t}\|h(s)\|_{\xLinfty}^2ds\right]^{\frac12}\right\}\nonumber\\
	&\leq&K\left[\sum_{i=1}^{\left\lfloor{\frac{t}{\overline{T}}}\right\rfloor}\exp(-\mu(t-\varepsilon-i\overline{T}))\|h\|_{\xLtwo((i-1)\overline{T},i\overline{T};\xLinfty\left(\Omega\right))}\right.\nonumber\\
	&\;&\left.+\|h\|_{\xLtwo(N\overline{T},T;\xLinfty\left(\Omega\right))}\right]\nonumber\\
	&\leq&K\left[\sum_{i=1}^{\left\lfloor{\frac{t}{\overline{T}}}\right\rfloor}\exp(-\mu(t-\varepsilon-i\overline{T}))\max_{i=1,\dots,N}\|h\|_{\xLtwo(((i-1)\overline{T},i\overline{T});\xLinfty\left(\Omega\right))}\right.\nonumber\\
	&\;&\left.+\|h\|_{\xLtwo(N\overline{T},T;\xLinfty\left(\Omega\right))}\right].
	\end{eqnarray}
	Now, the sum
	\begin{align}
	\sum_{i=1}^{\left\lfloor{\frac{t}{\overline{T}}}\right\rfloor}\exp(-\mu(t-\varepsilon-i\overline{T}))&\leq \exp(\mu \varepsilon)\sum_{i=1}^{\left\lfloor{\frac{t}{\overline{T}}}\right\rfloor}\exp\left(-\mu\left(\left\lfloor{\frac{t}{\overline{T}}}\right\rfloor\overline{T}-i\overline{T}\right)\right) \nonumber\\
	&= \exp(\mu \varepsilon)\sum_{i=1}^{\left\lfloor{\frac{t}{\overline{T}}}\right\rfloor}\exp\left(-\mu\left(\left\lfloor{\frac{t}{\overline{T}}}\right\rfloor-i\right)\overline{T}\right) \label{sum_eq_1}\\
	&= \exp(\mu \varepsilon)\sum_{j=0}^{\left\lfloor{\frac{t}{\overline{T}}}\right\rfloor-1}\exp\left(-\mu j \overline{T}\right) \label{sum_eq_2}\\
	&\leq \exp(\mu \varepsilon)\sum_{j=0}^{+\infty}\exp\left(-\mu j \overline{T}\right) \\
	%&= \exp(\mu \varepsilon)\frac{1}{1-\exp\left(-\mu \overline{T}\right)} \\
	&= \frac{\exp(\mu \varepsilon)}{1-\exp\left(-\mu \overline{T}\right)},  \label{sum_eq_3}\\
	\end{align}
	where in \eqref{sum_eq_1} we have used $t\leq \left\lfloor{\frac{t}{\overline{T}}}\right\rfloor$, in \eqref{sum_eq_2} we have performed the change of variable $j\coloneqq \left\lfloor{\frac{t}{\overline{T}}}\right\rfloor - i$ and in \eqref{sum_eq_3} we have computed the geometric series. Hence, by \eqref{ineq_source_term} and the above computation, we have\\
	\begin{eqnarray}\label{}
	\|\eta(t)\|_{\xLinfty}&\leq&K\left[\sum_{i=1}^{\left\lfloor{\frac{t}{\overline{T}}}\right\rfloor}\exp(-\mu(t-\varepsilon-i\overline{T}))\|h\|_{\xLtwo((i-1)\overline{T},i\overline{T};\xLinfty\left(\Omega\right))}\right.\nonumber\\
	&\;&\left.+\|h\|_{\xLtwo(N\overline{T},T;\xLinfty\left(\Omega\right))}\right]\nonumber\\
	&\leq&K\left[\frac{\exp(\mu \varepsilon)}{1-\exp\left(-\mu \overline{T}\right)}\max_{i=1,\dots,N}\|h\|_{\xLtwo(((i-1)\overline{T},i\overline{T});\xLinfty\left(\Omega\right))}\right.\nonumber\\
	&\;&\left.+\|h\|_{\xLtwo(N\overline{T},T;\xLinfty\left(\Omega\right))}\right]\nonumber\\
	&\leq&K\left[\max_{i=1,\dots,N}\|h\|_{\xLtwo((i-1)\overline{T},i\overline{T};\xLinfty\left(\Omega\right))}+\|h\|_{\xLtwo(N\overline{T},T;\xLinfty\left(\Omega\right))}\right],
	\end{eqnarray}
	with $K=K\left(\Omega, \overline{T}\right)$
	
	Then, by \cref{ineq_init_dat} and \cref{ineq_source_term}, for each $t\in [0,T]$
	\begin{equation*}
	\|\psi(t)\|_{\xLinfty\left(\Omega\right)}\leq K\exp\left(-\mu(t-\varepsilon)\right)\left[\|y_0\|_{\xLinfty\left(\Omega\right)}+\max_{i=1,\dots,N}\|h\|_{\xLtwo((i-1)\overline{T},i\overline{T};\xLinfty\left(\Omega\right))}+\|h\|_{\xLtwo(N\overline{T},T;\xLinfty\left(\Omega\right))}\right]
	\end{equation*}
	as desired.
\end{proof}

%\begin{rmrk}\label{remark_application_lemma_reg_nonnegativepotential}
%	Lemma \ref{lemma_reg_nonnegativepotential} can be applied to a bounded solution $y$ to \cref{semilinear_internal_1_slt}. Indeed, set
%	\begin{equation*}c_y(t,x)\coloneqq 
%	\begin{dcases}
%	\frac{f(y(t,x))}{y(t,x)} \hspace{1.3 cm} & y(t,x)\neq 0\\
%	f^{\prime}(0) & y(t,x)= 0.\\
%	\end{dcases}
%	\end{equation*}
%	Since $f$ is increasing and $f(0)=0$, we have $c_y\geq 0$. Hence, we are in position to apply Lemma \ref{lemma_reg_nonnegativepotential}, with potential $c_y$.
%\end{rmrk}

\section{Well posedeness and regularity of the state equation}
\label{appendixsec:sml_heat_well_posedeness}

In this subsection we study the well posedeness and regularity properties of the state equation \eqref{semilinear_internal_1_slt}.

Let $\Omega$ be a bounded open subset of $\mathbb{R}^n$, with boundary $\partial \Omega \in C^{2}$. The nonlinearity $f\in \xCone(\mathbb{R})$ is nondecreasing and $f(0)=0$. Let us also mention that, among the equivalent definitions for $\xHone_0$ norm, we choose $\left\|y\right\|_{\xHone_0\left(\Omega\right)}\coloneqq \sqrt{\int_{\Omega} \left\|\nabla y\right\|^2dx}$ as suggested by Poincar\'e’s inequality \cite[Corollary 9.19 page 290]{BFP}. To define the notion of solution, we introduce the class of test functions
\begin{equation*}
\mathscr{C}\coloneqq \xLtwo\left(\left(0,T\right);\xHone_0\left(\Omega\right)\right)\cap \xLinfty\left(\left(0,T\right)\times \Omega\right),
\end{equation*}
the Hilbert space (see e.g. \cite[(1.8)-(1.9) page 102]{LIO})
\begin{equation*}
W\left(0,T\right)\coloneqq \left\{y\in \xLtwo\left(\left(0,T\right);\xHone_0\left(\Omega\right)\right) \ | \ y_t\in \xLtwo\left(\left(0,T\right);\xHmone\left(\Omega\right)\right)\right\}
\end{equation*}
endowed with the norm
\begin{equation*}
\left\|y\right\|_{W\left(0,T\right)}\coloneqq \left\|y\right\|_{\xLtwo\left(\left(0,T\right);\xHone_0\left(\Omega\right)\right)}+\|y_t\|_{\xLtwo\left(\left(0,T\right);\xHmone\left(\Omega\right)\right)}.
\end{equation*}
and the Hilbert space
\begin{equation}\label{def_W_R}
W_R\left(0,T\right)\coloneqq \left\{y\in \xLtwo((0,T); \xHone_0\left(\Omega\right)\cap \xHtwo\left(\Omega\right)) \ | \ y_t\in \xLtwo\left(\left(0,T\right)\times\Omega\right) \right\}
\end{equation}
endowed with the norm:
\begin{equation*}
\|y\|_{W_R\left(0,T\right)}\coloneqq \|y\|_{\xLtwo((0,T); \xHtwo\left(\Omega\right))}+\|y_t\|_{\xLtwo\left(\left(0,T\right)\times\Omega\right)}.
\end{equation*}

\begin{dfntn}\label{def_sol_semilinear}
	Let $y_0\in \xLinfty\left(\Omega\right)$ and $h\in \xLtwo\left(\left(0,T\right)\times \Omega\right)$. Then, $y\in W\left(0,T\right)$ is said to be a solution to the Cauchy problem
	\begin{equation}\label{semilinear_internal_1_slt_wp}
	\begin{dcases}
	y_t-\Delta y+f\left(y\right)=h\hspace{2.8 cm} & \mbox{in} \hspace{0.10 cm}(0,T)\times\Omega\\
	y=0  & \mbox{on}\hspace{0.10 cm} (0,T)\times \partial \Omega\\
	y(0,x)=y_0(x)  & \mbox{in}\hspace{0.10 cm}  \Omega.
	\end{dcases}
	\end{equation}
	if $y\left(0,\cdot\right)=y_0$ in $\xLtwo\left(\Omega\right)$, $f\left(y\right)\in \xLone\left(\left(0,T\right)\times \Omega\right)$ and for any test function $\varphi\in \mathscr{C}$, we have
	\begin{equation}\label{weak_sol_id}
	\int_0^T\left\{\langle y_t,\varphi\rangle+\int_{\Omega}\left[\nabla y\nabla \varphi+f\left(y\right)\varphi\right]dx\right\}dt=\int_0^T\int_{\Omega}h\varphi dxdt,
	\end{equation}
	where $\langle\cdot, \cdot\rangle$ denotes the duality product $\left(\xHone_0\left(\Omega\right), \xHmone\left(\Omega\right)\right)$.
\end{dfntn}

To state the next Proposition, we introduce the following class of initial data
\begin{equation}\label{def_L_1F^2}
\xLoneFtwo \coloneqq \left\{y_0\in \xLtwo\left(\Omega\right) \ | \ F\left(y_0\right)\in \xLone\left(\Omega\right) \right\},
\end{equation}
where $F\left(y\right)\coloneqq \int_0^y f\left(\xi\right)d\xi$. Next Proposition is devoted to the well posedeness of \eqref{semilinear_internal_1_slt_wp}, while Proposition \ref{prop_sml_heat_improved_regularity} deals with regularity properties.

\begin{prpstn}[Well posedeness]\label{prop_sml_heat_well_posedeness}
	Let $y_0\in \xLoneFtwo$ be an initial datum and $h\in \xLtwo\left(\left(0,T\right)\times \Omega\right)$ be a source term. There exists a unique solution $y$ to \eqref{semilinear_internal_1_slt_wp}. Moreover, $f\left(y\right)$ is square integrable, i.e. $f\left(y\right)\in \xLtwo\left(\left(0,T\right)\times \Omega\right)$, with estimates
	\begin{equation}\label{prop_sml_heat_well_posedeness_est_2}
	\|y\|_{\xLtwo\left(\left(0,T\right);\xHone_0\left(\Omega\right)\right)}\leq \|y_0\|_{\xLtwo\left(\Omega\right)}+\|h\|_{\xLtwo\left(\left(0,T\right);\xHmone\left(\Omega\right)\right)},
	\end{equation}
	\begin{equation}\label{prop_sml_heat_well_posedeness_est_3}
	\left\|f\left(y\right)\right\|_{\xLtwo\left(\left(0,T\right)\times\Omega\right)}\leq \sqrt{2\left\|F\left(y_0\right)\right\|_{\xLone\left(\Omega\right)}}+\|h\|_{\xLtwo\left(\left(0,T\right)\times\Omega\right)}
	\end{equation}
	and
	\begin{equation}\label{prop_sml_heat_well_posedeness_est_4}
	\left\|y_t\right\|_{\xLtwo\left(\left(0,T\right);\xHmone\left(\Omega\right)\right)}\leq K\left[\left\|y_0\right\|_{\xLtwo\left(\Omega\right)}+\sqrt{\left\|F\left(y_0\right)\right\|_{\xLone\left(\Omega\right)}}+\|h\|_{\xLtwo\left(\left(0,T\right)\times\Omega\right)}\right],
	\end{equation}
	where $F\left(y\right) = \int_0^y f\left(\xi\right)d\xi$ and $K=K\left(\Omega\right)$.
\end{prpstn}

\begin{rmrk}
	In the literature, well posedeness and regularity of dissipative semilinear heat equations have been treated extensively (see e.g. \cite{casas1995optimal,raymond1999hamiltonian,barbu2010nonlinear}). However, in \cite{casas1995optimal,raymond1999hamiltonian} the source is required to be $\xLn{p}$, with $p$ large enough to have bounded solutions. Our goal is to show how to prove well posedeness and regularity with $\xLtwo$ source. As we shall see in Proposition \ref{prop_sml_heat_improved_regularity}, if the initial datum is bounded, the solution is bounded in space (of class $\xLtwo\left(\left(0,T\right);\xLinfty\left(\Omega\right)\right)$), but possibly unbounded in space-time. In \cite[Proposition 5.1 page 195]{barbu2010nonlinear} unbounded sources are considered, but to have solutions in $W\left(0,T\right)$ the initial datum is supposed to be in $\xHone_0\left(\Omega\right)\cap \xLoneFtwo$. In our result, we ask only $y_0\in \xLoneFtwo$. Note that, in case the source is only space dependent $h=h(x)$, \eqref{semilinear_internal_1_slt_wp} can be seen as a gradient flow of the convex functional
	\begin{equation}\label{functional_I}
	I:D(I)\subset \xLtwo\left(\Omega\right)\longrightarrow \mathbb{R},\hspace{0.3 cm}I\left(y\right)\coloneqq \int_{\Omega}\left[\frac12\left\|\nabla y\right\|^2+F\left(y\right)-hy\right]dx,
	\end{equation}
	where $F\left(y\right)=\int_0^yf\left(\xi\right)d\xi$ (see \cite[section 9.6]{PDE} and references therein). This explains the condition $F\left(y_0\right)\in \xLone\left(\Omega\right)$ we imposed, which requires that the term $\int_{\Omega}F\left(y_0\right)dx$ appearing in the functional \eqref{functional_I} is finite.
	
	The initial datum $y_0$ is supposed to be in $\xLoneFtwo$. However, this condition may be relaxed to weaker integrability conditions by looking for a solution in a larger Banach space.
	%	Like
	%	\begin{itemize}
	%		\item $y_0\in \xLtwo \left(\Omega\right)$, thus maybe loosing some regularity;
	%		\item $y_0\in \xLtwo \left(\Omega\right)$, with $F\left(y_0\right)\in \xLone\left(\Omega\right)$ to keep the same regularity.
	%	\end{itemize}
\end{rmrk}

The proof of Proposition \ref{prop_sml_heat_well_posedeness} is inspired by \cite[Theorem 4.7, page 29]{boccardo2013elliptic}.
\begin{proof}[Proof of Proposition \ref{prop_sml_heat_well_posedeness}.]
	Uniqueness follows from energy estimates. We focus on the proof of existence and estimates. Along the proof, we will use that $f$ is nondecreasing and $f(0)=0$, which yield $f^{\prime}(y)\geq 0$ and $F\left(y\right)=\int_0^yf\left(\xi\right)d\xi \geq 0$, for any $y\in \mathbb{R}$. Integration by parts will be employed. We will denote by $K$ a large enough constant depending only on the domain $\Omega$.
	
	\textit{Step 1} \ \textbf{Existence for bounded data}\\
	Take initial datum $y_0\in \xLinfty(\Omega)$ and source $h\in \xLinfty((0,T)\times \Omega)$. By \cite[Theorem 2.1 page 547]{casas1995optimal}, there exists a bounded solution $y$ to \eqref{semilinear_internal_1_slt_wp}. Truncation methods are employed to deal with possibly non Lipschitz nonlinearity.
	
	\textit{Step 2} \ \textbf{Estimates for bounded data}\\
	We now aim at proving \eqref{prop_sml_heat_well_posedeness_est_2}, \eqref{prop_sml_heat_well_posedeness_est_3} and \eqref{prop_sml_heat_well_posedeness_est_4} for $\xLinfty$ data. In order to show \eqref{prop_sml_heat_well_posedeness_est_2}, in \eqref{weak_sol_id} let us choose as test function $\varphi\coloneqq y$, getting
	\begin{eqnarray}\label{}
	\int_0^T\int_{\Omega}hy dxdt&=&\int_0^T\int_{\Omega}\left[y_ty+\left\|\nabla y\right\|^2+f\left(y\right)y\right]dxdt\nonumber\\
	&=&\frac12\left[\left\|y(T,\cdot)\right\|_{\xLtwo(\Omega)}^2-\left\|y_0\right\|_{\xLtwo(\Omega)}^2\right]+\|y\|_{\xLtwo\left(\left(0,T\right);\xHone_0\left(\Omega\right)\right)}^2+\int_0^T\int_{\Omega}f\left(y\right)ydxdt.\nonumber\\
	\end{eqnarray}
	Then, by Young's inequality
	\begin{eqnarray}\label{}
	\|y\|_{\xLtwo\left(\left(0,T\right);\xHone_0\left(\Omega\right)\right)}^2+\int_0^T\int_{\Omega}f\left(y\right)ydxdt&\leq&\int_0^T\int_{\Omega}hy dxdt+\left\|y_0\right\|_{\xLtwo(\Omega)}^2\nonumber\\
	&\leq&\frac12\left\|h\right\|_{\xLtwo\left(\left(0,T\right);\xHmone\left(\Omega\right)\right)}^2+\frac12\left\|y\right\|_{\xLtwo\left(\left(0,T\right);\xHone_0\left(\Omega\right)\right)}^2+\frac12\left\|y_0\right\|_{\xLtwo(\Omega)}^2,\nonumber\\
	\end{eqnarray}
	which yields
	\begin{equation}\label{prop_sml_heat_well_posedeness_proof_eq6}
	\frac12\|y\|_{\xLtwo\left(\left(0,T\right);\xHone_0\left(\Omega\right)\right)}^2+\int_0^T\int_{\Omega}f\left(y\right)ydxdt\leq \frac12\left\|h\right\|_{\xLtwo\left(\left(0,T\right);\xHmone\left(\Omega\right)\right)}^2+\frac12\left\|y_0\right\|_{\xLtwo(\Omega)}^2,
	\end{equation}
	whence \eqref{prop_sml_heat_well_posedeness_est_2} follows. For the proof of \eqref{prop_sml_heat_well_posedeness_est_3}, let us observe that, since $y$ is bounded and $f$ is $\xCone$, $f^{\prime}\left(y\right)$ is bounded. Then, $\nabla \left(f\left(y\right)\right)=f^{\prime}\left(y\right)\nabla y\in \xLtwo\left(\left(0,T\right)\times\Omega\right)$, whence $f\left(y\right)$ is eligible as test function in \eqref{weak_sol_id}, thus obtaining
	\begin{eqnarray}\label{}
	\int_0^T\int_{\Omega}hf\left(y\right) dxdt&=&\int_0^T\int_{\Omega}\left[y_tf\left(y\right)+\nabla y\nabla \left(f\left(y\right)\right)+\left|f\left(y\right)\right|^2\right]dxdt\nonumber\\
	&=&\left\|F\left(y\left(T,\cdot\right)\right)\right\|_{\xLone\left(\Omega\right)}-\left\|F\left(y_0\right)\right\|_{\xLone\left(\Omega\right)}+\int_0^T\int_{\Omega}\left[f^{\prime}\left(y\right)\left\|\nabla y\right\|^2+\left|f\left(y\right)\right|^2\right]dxdt\nonumber\\
	&\geq&-\left\|F\left(y_0\right)\right\|_{\xLone\left(\Omega\right)}+\int_0^T\int_{\Omega}\left|f\left(y\right)\right|^2dxdt.\nonumber\\
	\end{eqnarray}
	Therefore, by Cauchy–Schwarz and Young's inequality
	\begin{eqnarray}\label{}
	\int_0^T\int_{\Omega}\left|f\left(y\right)\right|^2dxdt&\leq&\left\|F\left(y_0\right)\right\|_{\xLone\left(\Omega\right)}+\int_0^T\int_{\Omega}hf\left(y\right) dxdt\nonumber\\
	&\leq&\left\|F\left(y_0\right)\right\|_{\xLone\left(\Omega\right)}+\frac12\left\|h\right\|_{\xLtwo\left(\left(0,T\right)\times\Omega\right)}^2+\frac12\left\|f\left(y\right)\right\|_{\xLtwo\left(\left(0,T\right)\times\Omega\right)}^2,\nonumber\\
	\end{eqnarray}
	whence
	\begin{equation}\label{prop_sml_heat_well_posedeness_proof_eq7}
	\frac12\int_0^T\int_{\Omega}\left|f\left(y\right)\right|^2dxdt\leq \left\|F\left(y_0\right)\right\|_{\xLone\left(\Omega\right)}+\frac12\left\|h\right\|_{\xLtwo\left(\left(0,T\right)\times\Omega\right)}^2,
	\end{equation}
	which leads to \eqref{prop_sml_heat_well_posedeness_est_3}. For the proof of \eqref{prop_sml_heat_well_posedeness_est_4}, let us arbitrarily choose a test function $\varphi$ in the class $\mathscr{C}$ of test functions. Then, from \eqref{weak_sol_id} we have
	\begin{eqnarray}\label{prop_sml_heat_well_posedeness_proof_ytdual}
	\left|\int_0^T\langle y_t,\varphi\rangle dt\right|&=&\left|\int_0^T\int_{\Omega}\left[-\nabla y\nabla \varphi-f\left(y\right)\varphi+h\varphi\right]dxdt\right|\nonumber\\
	&\leq&\|y\|_{\xLtwo\left(\left(0,T\right);\xHone_0\left(\Omega\right)\right)}\|\varphi\|_{\xLtwo\left(\left(0,T\right);\xHone_0\left(\Omega\right)\right)}+\left\|f\left(y\right)\right\|_{\xLtwo\left(\left(0,T\right)\times\Omega\right)}\left\|\varphi\right\|_{\xLtwo\left(\left(0,T\right)\times\Omega\right)}\nonumber\\
	&\;&+\left\|h\right\|_{\xLtwo\left(\left(0,T\right)\times\Omega\right)}\left\|\varphi\right\|_{\xLtwo\left(\left(0,T\right)\times\Omega\right)}\nonumber\\
	&\leq&K\left[\left\|y_0\right\|_{\xLtwo\left(\Omega\right)}+\sqrt{\left\|F\left(y_0\right)\right\|_{\xLone\left(\Omega\right)}}+\left\|h\right\|_{\xLtwo\left(\left(0,T\right)\times\Omega\right)}\right]\left\|\varphi\right\|_{\xLtwo\left(\left(0,T\right);\xHone_0\left(\Omega\right)\right)},
	\end{eqnarray}
	as required. For the sake of uniform integrability needed in next steps, we will estimate the $\xLone$ norm of $f\left(y\right)$ over an arbitrary measurable set, by adapting the arguments of the proof of \cite[Theorem 4.7, page 29]{boccardo2013elliptic} based on the Vitali Convergence Theorem \cite[page $94$]{RMI}. Arbitrarily choose $M>0$ and $E$ Lebesgue measurable subset of $(0,T)\times \Omega$. Set
	\begin{equation}
	E_M\coloneqq \left\{(t,x)\in (0,T)\times \Omega \ | \ \left|y(t,x)\right|>M\right\}.
	\end{equation}
	
	Now, on the one hand,
	\begin{align}
	\int_{E_M} \left|f\left(y\right)\right|d\left(x,t\right)&\leq \frac{1}{M}\int_{E_M} \left|f\left(y\right)y\right|d\left(x,t\right)\label{prop_sml_heat_well_posedeness_proof_eq18}\\
	&=\frac{1}{M}\int_{0}^{T}\int_{\Omega} f\left(y\right)y \ dxdt\label{prop_sml_heat_well_posedeness_proof_eq19}\\
	&\leq\frac{1}{2M}\left[\left\|h\right\|_{\xLtwo\left(\left(0,T\right);\xHmone\left(\Omega\right)\right)}^2+\left\|y_0\right\|_{\xLtwo(\Omega)}^2\right]\label{prop_sml_heat_well_posedeness_proof_eq20}\\
	\end{align}
	where in \eqref{prop_sml_heat_well_posedeness_proof_eq18} the definition of $E_M$ is used, in \eqref{prop_sml_heat_well_posedeness_proof_eq19} the fact that $f$ is noncreasing together with $f(0)=0$ is employed and in \eqref{prop_sml_heat_well_posedeness_proof_eq20} inequality \eqref{prop_sml_heat_well_posedeness_proof_eq6} is used. Therefore
	\begin{equation}\label{prop_sml_heat_well_posedeness_proof_eq21}
	\int_{E_M} \left|f\left(y\right)\right|d\left(x,t\right)\leq \frac{1}{2M}\left[\left\|h\right\|_{\xLtwo\left(\left(0,T\right);\xHmone\left(\Omega\right)\right)}^2+\left\|y_0\right\|_{\xLtwo(\Omega)}^2\right].
	\end{equation}
	
	On the other hand, since $f$ is increasing and $f(0)=0$, we have
	\begin{align}\label{prop_sml_heat_well_posedeness_proof_eq24}
	\int_{E \setminus E_M} \left|f\left(y\right)\right|d\left(x,t\right)&\leq \int_{E \setminus E_M} \max\left\{\left|f\left(-M\right)\right|,\left|f\left(M\right)\right|\right\}d\left(x,t\right)\nonumber\\
	&\leq \int_{E} \max\left\{\left|f\left(-M\right)\right|,\left|f\left(M\right)\right|\right\}d\left(x,t\right)\nonumber\\
	&=\mu_{leb}\left(E\right)\max\left\{\left|f\left(-M\right)\right|,\left|f\left(M\right)\right|\right\}.\nonumber\\
	\end{align}
	
	Putting together \eqref{prop_sml_heat_well_posedeness_proof_eq21} and \eqref{prop_sml_heat_well_posedeness_proof_eq24}, we obtain
	\begin{eqnarray}\label{prop_sml_heat_well_posedeness_proof_eq32}
	\int_{E} \left|f\left(y\right)\right|d\left(x,t\right)&=&\int_{E_M} \left|f\left(y\right)\right|d\left(x,t\right) + \int_{E \setminus E_M} \left|f\left(y\right)\right|d\left(x,t\right)\nonumber\\
	&\leq &\frac{1}{2M}\left[\left\|h\right\|_{\xLtwo\left(\left(0,T\right);\xHmone\left(\Omega\right)\right)}^2+\left\|y_0\right\|_{\xLtwo(\Omega)}^2\right]\nonumber\\
	&\;&+\mu_{leb}\left(E\right)\max\left\{\left|f\left(-M\right)\right|,\left|f\left(M\right)\right|\right\}.\nonumber\\
	\end{eqnarray}
	This will allow us to apply Vitali Convergence Theorem \cite[page $94$]{RMI} to perform the density argument in step 3.
	
	\textit{Step 3} \ \textbf{Existence of solution for unbounded data}\\
	Let $y_0\in \xLoneFtwo$ be an initial datum and let $h$ be a source in $\xLtwo\left(\left(0,T\right)\times\Omega\right)$. Take sequences $\left\{y_{0,m}\coloneqq y_0\chi_{\left\{\left|y_0\right|\leq m\right\}}\right\}_{m\in \mathbb{N}}\subset \xLinfty\left(\Omega\right)$ and $\left\{h_m \coloneqq h\chi_{\left\{\left|h\right|\leq m\right\}} \right\}_{m\in \mathbb{N}}\subset \xLinfty\left(\left(0,T\right)\times\Omega\right)$, where $\chi_E$ denotes the characteristic function of a set $E$. By Dominated Convergence Theorem
	%Lemma(01/01/2021_03)
	, we have
	\begin{equation}
	\left\|y_{0,m}-y_0\right\|_{\xLtwo\left(\Omega\right)}+\left\|F\left(y_{0,m}\right)-F\left(y_0\right)\right\|_{\xLone\left(\Omega\right)}+\left\|h_m-h\right\|_{\xLtwo\left(\left(0,T\right)\times\Omega\right)}\underset{m\to +\infty}{\longrightarrow}0.
	\end{equation}
	For each $m\in \mathbb{N}$, set $y_m$ solution to \eqref{semilinear_internal_1_slt_wp} with initial datum $y_{0,m}$ and source $h_m$. By \eqref{prop_sml_heat_well_posedeness_proof_eq6}, \eqref{prop_sml_heat_well_posedeness_proof_eq7} and \eqref{prop_sml_heat_well_posedeness_proof_ytdual}, the sequences $\left\{y_m\right\}_{m\in \mathbb{N}}\subset W\left(0,T\right)$ and $\left\{f\left(y_m\right)\right\}_{m\in \mathbb{N}}\subset \xLtwo\left(\left(0,T\right)\times \Omega\right)$ are bounded. Hence, by Banach-Alaoglu Theorem, there exists $y\in W\left(0,T\right)$ with $f\left(y\right)\in \xLtwo\left(\left(0,T\right)\times \Omega\right)$ such that, up to subsequences,
	\begin{equation}\label{prop_sml_heat_well_posedeness_proof_weakconv}
	y_m\underset{m\to +\infty}{\rightharpoonup}y,
	\end{equation}
	weakly in $W\left(0,T\right)$ and
	\begin{equation}
	f\left(y_m\right)\underset{m\to +\infty}{\rightharpoonup}f\left(y\right)
	\end{equation}
	weakly in $\xLtwo\left(\left(0,T\right)\times \Omega\right)$. By parabolic compactness (\cite{Simon1986} or \cite[Th\'eor\`eme 2.4.1 page 51]{ISV}), the sequence $\left\{y_m\right\}_{m\in \mathbb{N}}$ is relatively compact in $\xLtwo\left(\left(0,T\right)\times\Omega\right)$, whence, up to subsequences
	\begin{equation}\label{prop_sml_heat_well_posedeness_proof_eq216}
	y_m\underset{m\to +\infty}{\longrightarrow}y,\hspace{0.36 cm}\mbox{a.e. in}\hspace{0.06 cm}(0,T)\times \Omega,
	\end{equation}
	which, together with the continuity of $f$, yields
	\begin{equation}\label{prop_sml_heat_well_posedeness_proof_eq218}
	f\left(y_m\right)\underset{m\to +\infty}{\longrightarrow}f\left(y\right),\hspace{0.36 cm}\mbox{a.e. in}\hspace{0.06 cm}(0,T)\times \Omega.
	\end{equation}
	Moreover, \eqref{prop_sml_heat_well_posedeness_proof_eq32} gives uniform integrability of $\left\{y_m\right\}_{m\in \mathbb{N}}$. Then, we are allowed to apply Vitali Convergence Theorem \cite[page $94$]{RMI}, getting $f\left(y\right)\in \xLone(\Omega)$ and
	\begin{equation}\label{prop_sml_heat_well_posedeness_proof_eq174}
	f\left(y_m\right)\underset{m\to +\infty}{\longrightarrow}f\left(y\right),
	\end{equation}
	in $\xLone\left(\left(0,T\right)\times \Omega\right)$. Now, the weak convergence \eqref{prop_sml_heat_well_posedeness_proof_weakconv} and the strong convergence \eqref{prop_sml_heat_well_posedeness_proof_eq174} enable us to pass to the limit as $m\to +\infty$ in \eqref{weak_sol_id}, thus showing that in fact $y$ is a solution to \eqref{semilinear_internal_1_slt_wp}. To finish the proof, we apply \eqref{prop_sml_heat_well_posedeness_proof_eq6}, \eqref{prop_sml_heat_well_posedeness_proof_eq7} and \eqref{prop_sml_heat_well_posedeness_proof_ytdual} to $\left\{y_m\right\}_{m\in \mathbb{N}}$. Then, using the Lower Semicontinuity of the norm with respect to the weak convergence \cite[Proposition 3.13 (iii) page 63]{BFP}, we take the limit as $m\to +\infty$, getting respectively \eqref{prop_sml_heat_well_posedeness_est_2}, \eqref{prop_sml_heat_well_posedeness_est_3} and \eqref{prop_sml_heat_well_posedeness_est_4}. This finishes the proof.
\end{proof}

\begin{prpstn}[Improved regularity]\label{prop_sml_heat_improved_regularity}
	Let $y_0\in \xLoneFtwo\cap \xHone_0\left(\Omega\right)$ be an initial datum and $h\in \xLtwo\left(\left(0,T\right)\times \Omega\right)$ be a source term. Let $y$ be the corresponding solution to \eqref{semilinear_internal_1_slt_wp}. Then, in fact, $y\in \xLtwo\left(\left(0,T\right);\xHone_0\left(\Omega\right)\cap \xHtwo\left(\Omega\right)\right)$ and $y_t\in \xLtwo\left(\left(0,T\right)\times\Omega\right)$, with
	\begin{equation}\label{prop_sml_heat_well_posedeness_regest_1}
	\|y\|_{\xLtwo((0,T); \xHtwo\left(\Omega\right))}\leq K\left[\|y_0\|_{\xHone_0\left(\Omega\right)}+\|h\|_{\xLtwo\left(\left(0,T\right)\times\Omega\right)}\right]
	\end{equation}
	and
	\begin{equation}\label{prop_sml_heat_well_posedeness_regest_2}
	\|y_t\|_{\xLtwo\left(\left(0,T\right)\times\Omega\right)}\leq \|y_0\|_{\xHone_0\left(\Omega\right)}+\sqrt{2\left\|F\left(y_0\right)\right\|_{\xLone\left(\Omega\right)}}+\|h\|_{\xLtwo\left(\left(0,T\right)\times\Omega\right)},
	\end{equation}
	where $K=K\left(\Omega\right)$.
	
	Suppose, the initial datum $y_0\in \xLinfty\left(\Omega\right)$ (not required to be of class $\xHone_0\left(\Omega\right)$) and the source $h\in \xLtwo\left(\left(0,T\right)\times \Omega\right)$. Assume the space dimension $n=1,2,3$. Then, $y$ is bounded in space, namely $y\in \xLtwo\left(\left(0,T\right);\xHone_0\left(\Omega\right)\cap \xLinfty\left(\Omega\right)\right)$, with estimates
	\begin{equation}\label{prop_sml_heat_well_posedeness_est_1}
	\|y\|_{\xLtwo\left((0,T); \xLinfty\left(\Omega\right)\right)}\leq K\left[\|y_0\|_{\xLinfty\left(\Omega\right)}+\|h\|_{\xLtwo\left(\left(0,T\right)\times\Omega\right)}\right],
	\end{equation}
	with $K=K\left(\Omega\right)$
\end{prpstn}
\begin{proof}[Proof of Proposition \ref{prop_sml_heat_improved_regularity}]
	As in the proof of Proposition \ref{prop_sml_heat_well_posedeness}, we will use that $f$ is nondecreasing and $f(0)=0$, which yield $f^{\prime}(y)\geq 0$ and $F\left(y\right)=\int_0^yf\left(\xi\right)d\xi \geq 0$, for any $y\in \mathbb{R}$. Integration by parts will be employed. We will denote by $K$ a large enough constant depending only on the domain $\Omega$.
	
	\textit{Step 1} \ \textbf{Estimates for bounded data}\\
	We now aim at proving \eqref{prop_sml_heat_well_posedeness_regest_1} and \eqref{prop_sml_heat_well_posedeness_regest_2} for $\xLinfty$ data working with $y_0\in \xLinfty\left(\Omega\right)\cap \xHone_0\left(\Omega\right)$ and $h\in \xLinfty\left(\left(0,T\right)\times \Omega\right)$. By Proposition \ref{prop_sml_heat_well_posedeness}, $f\left(y\right)\in \xLtwo\left(\left(0,T\right)\times \Omega\right)$. Then, by applying \cite[Theorem 5 page 382]{PDE} to
	\begin{equation*}
	\begin{dcases}
	y_t-\Delta y=-f\left(y\right)+h\hspace{2.8 cm} & \mbox{in} \hspace{0.10 cm}(0,T)\times\Omega\\
	y=0  & \mbox{on}\hspace{0.10 cm} (0,T)\times \partial \Omega\\
	y(0,x)=y_0(x)  & \mbox{in}\hspace{0.10 cm}  \Omega,
	\end{dcases}
	\end{equation*}
	we get $y\in \xLtwo\left(\left(0,T\right);\xHone_0\left(\Omega\right)\cap \xHtwo\left(\Omega\right)\right)$ and $y_t\in \xLtwo\left(\left(0,T\right)\times\Omega\right)$. Let us choose as test function $\varphi\coloneqq -\Delta y$ in \eqref{weak_sol_id}\footnote{
		$\varphi= -\Delta y$ may not be in $\mathscr{C}$. However, in case $y\in \xLtwo\left(\left(0,T\right);\xHone_0\left(\Omega\right)\cap \xHtwo\left(\Omega\right)\right)\cap \xLinfty\left(\left(0,T\right)\times \Omega\right)$, $y_t\in \xLtwo\left(\left(0,T\right)\times\Omega\right)$ and $f\left(y\right)\in \xLtwo\left(\left(0,T\right)\times \Omega\right)$, \eqref{weak_sol_id} is valid for each test function in $\mathscr{C}$ if and only if it holds for any test function in $\xLtwo\left(\left(0,T\right)\times \Omega\right)$.
	}, obtaining
	\begin{eqnarray}\label{}
	\int_0^T\int_{\Omega}h\left(-\Delta y\right) dxdt%&=&\int_0^T\int_{\Omega}\left[y_t\left(-\Delta y\right)+\nabla y\nabla\left(-\Delta y\right)+f\left(y\right)\left(-\Delta y\right)\right]dxdt\nonumber\\
	&=&\int_0^T\int_{\Omega}\left[y_t\left(-\Delta y\right)+\left|-\Delta y\right|^2+f\left(y\right)\left(-\Delta y\right)\right]dxdt\nonumber\\
	&=&\int_0^T\int_{\Omega}\left[\nabla y_t\nabla y+\left|-\Delta y\right|^2+f^{\prime}\left(y\right)\left\|\nabla y\right\|^2\right]dxdt\nonumber\\
	%&\geq&\int_0^T\int_{\Omega}\left[\frac{\partial}{\partial t}\left(\frac12\left\|\nabla y\right\|^2\right)+\left|-\Delta y\right|^2\right]dxdt\nonumber\\
	&\geq&\frac12\left\|\nabla y(T,\cdot)\right\|_{\xLtwo(\Omega)}^2-\frac12\left\|\nabla y_0\right\|_{\xLtwo(\Omega)}^2+\int_0^T\int_{\Omega}\left|-\Delta y\right|^2dxdt\nonumber\\
	\end{eqnarray}
	Therefore, by Cauchy-Schwarz and Young's inequality
	\begin{eqnarray}\label{}
	\left\|\Delta y\right\|_{\xLtwo\left(\left(0,T\right)\times \Omega\right)}^2&\leq &\int_0^T\int_{\Omega}h\left(-\Delta y\right) dxdt + \frac12\left\|\nabla y_0\right\|_{\xLtwo(\Omega)}^2\nonumber\\
	&\leq&\frac12 \left\|h\right\|_{\xLtwo\left(\left(0,T\right)\times \Omega\right)}^2+\frac12\left\|\Delta y\right\|_{\xLtwo\left(\left(0,T\right)\times \Omega\right)}^2 + \frac12\left\|\nabla y_0\right\|_{\xLtwo(\Omega)}^2,\nonumber\\
	\end{eqnarray}
	whence, by elliptic regularity \cite[Theorem 4 page 334]{PDE} we get \eqref{prop_sml_heat_well_posedeness_regest_1}. To prove \eqref{prop_sml_heat_well_posedeness_regest_2} for bounded data. Let us choose as test function $\varphi\coloneqq y_t$ in \eqref{weak_sol_id}\footnote{
		$\varphi= y_t$ may not be in $\mathscr{C}$. As in the former footnote.}, getting
	\begin{eqnarray}\label{}
	\int_0^T\int_{\Omega}hy_t dxdt&=&\int_0^T\int_{\Omega}\left[\left|y_t\right|^2+\nabla y\nabla\left(y_t\right)+f\left(y\right)y_t\right]dxdt\nonumber\\
	&=&\int_0^T\int_{\Omega}\left|y_t\right|^2dxdt+\frac12\left\|\nabla y(T,\cdot)\right\|_{\xLtwo(\Omega)}^2-\frac12\left\|\nabla y_0\right\|_{\xLtwo(\Omega)}^2+\left\|F\left(y\left(T,\cdot\right)\right)\right\|_{\xLone(\Omega)}-\left\|F\left(y_0\right)\right\|_{\xLone(\Omega)}.\nonumber\\
	\end{eqnarray}
	Then, by Cauchy-Schwarz and Young's inequality
	\begin{eqnarray}\label{}
	\left\|y_t\right\|_{\xLtwo\left(\left(0,T\right)\times \Omega\right)}^2&\leq &\int_0^T\int_{\Omega}hy_t dxdt + \frac12\left\|\nabla y_0\right\|_{\xLtwo(\Omega)}^2+\left\|F\left(y_0\right)\right\|_{\xLone(\Omega)}\nonumber\\
	&\leq&\frac12 \left\|h\right\|_{\xLtwo\left(\left(0,T\right)\times \Omega\right)}^2+\frac12\left\|y_t\right\|_{\xLtwo\left(\left(0,T\right)\times \Omega\right)}^2 + \frac12\left\|\nabla y_0\right\|_{\xLtwo(\Omega)}^2+\left\|F\left(y_0\right)\right\|_{\xLone(\Omega)},\nonumber\\
	\end{eqnarray}
	which leads to \eqref{prop_sml_heat_well_posedeness_regest_2}.
	
	\textit{Step 2} \ \textbf{Estimates for unbounded data}\\
	Suppose the initial datum $y_0\in \xLinfty\left(\Omega\right)\cap \xHone_0\left(\Omega\right)$. Consider the sequence $\left\{y_m\right\}_{m\in \mathbb{N}}$ constructed in step 4 of the proof of Proposition \ref{prop_sml_heat_well_posedeness}. In step 1 we have obtained \eqref{prop_sml_heat_well_posedeness_regest_1} and \eqref{prop_sml_heat_well_posedeness_regest_2} for bounded data. We apply them to $\left\{y_m\right\}_{m\in \mathbb{N}}$, getting respectively
	\begin{equation}\label{prop_sml_heat_well_posedeness_regest_1_y_m}
	\|y_m\|_{\xLtwo((0,T); \xHtwo\left(\Omega\right))}\leq K\left[\|y_0\|_{\xHone_0\left(\Omega\right)}+\|h_m\|_{\xLtwo\left(\left(0,T\right)\times\Omega\right)}\right]
	\end{equation}
	and
	\begin{equation}\label{prop_sml_heat_well_posedeness_regest_2_y_m}
	\|\left(y_m\right)_t\|_{\xLtwo\left(\left(0,T\right)\times\Omega\right)}\leq \|y_0\|_{\xHone_0\left(\Omega\right)}+\sqrt{2\left\|F\left(y_0\right)\right\|_{\xLone\left(\Omega\right)}}+\|h_m\|_{\xLtwo\left(\left(0,T\right)\times\Omega\right)}.
	\end{equation}
	Then, by Banach-Alaoglu Theorem, $\left\{y_m\right\}_{m\in \mathbb{N}}$ is weakly precompact in $W_R\left(0,T\right)$ (defined in \eqref{def_W_R}), whence $y\in W_R\left(0,T\right)$ and the above inequalities hold for $y$.
	
	\textit{Step 3} \ \textbf{Boundedness in space}\\
	Let us now assume $y_0\in \xLinfty\left(\Omega\right)$ and $h\in \xLtwo\left(\left(0,T\right)\times\Omega\right)$. The boundedness in space of $y$ and \eqref{prop_sml_heat_well_posedeness_est_1} follows from Lemma \ref{lemma_reg_nonnegativepotential}, with potential
	\begin{equation*}c(t,x)\coloneqq 
	\begin{dcases}
	\frac{f(y(t,x))}{y(t,x)} \hspace{1.3 cm} & y(t,x)\neq 0\\
	f^{\prime}(0) & y(t,x)= 0.\\
	\end{dcases}
	\end{equation*}
	This concludes the proof of the Proposition.
\end{proof}

\section{Uniform bounds of the optima}
\label{appendixsec:Uniform bounds of the optima}

As pointed out in \cite[subsection 3.2]{PZ2}, the norms of optimal controls and states can be estimated in terms of the initial datum for \cref{semilinear_internal_1_slt} and the running target in an averaged sense, using the inequality
\begin{equation}
J_{T}\left(u^T\right)\leq J_{T}\left(0\right),
\end{equation}
where $u^T$ is any optimal control for the time-evolution problem. We have to ensure that the bounds actually holds for any time, i.e. we need to show that optimal controls and states do not oscillate too much.

The proof of Lemma \ref{lemma_bound_optima} follows the scheme:
\begin{itemize}
	\item divide the interval $[0,T]$ into subintervals of $T$-independent length;
	\item estimate the magnitude of the optima in each subinterval by using controllability (\cref{oscillations_estimate}).
\end{itemize}
%We fix a threshold and we consider those subintervals where the optima are larger than the threshold. In such subintervals, we employ controllability to bound the oscillations.

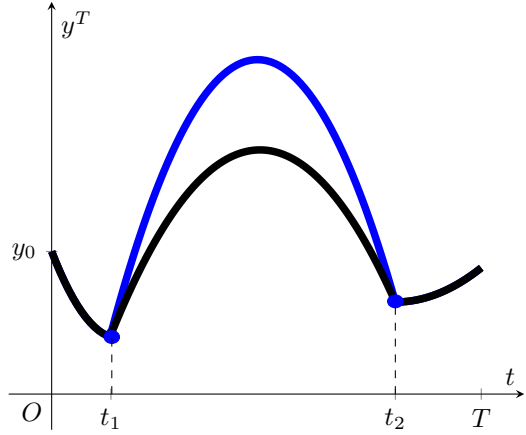
\begin{figure}
	\begin{center}
		\begin{tikzpicture}[
		declare function={
			func(\x)= (\x < 1.36) * (pow(\x,2)-3.12*\x+4)   +
			and(\x >= 1.36, \x < 8) * (-0.66*pow(\x,2)+6.3237*\x -5.7695)     +
			(\x >= 8) * (0.24*pow(\x-8,2)+2.58)
			;
		}
		]
		\begin{axis}[
		axis x line=middle, axis y line=middle,
		ymin=-1, ymax=11, ytick={4}, yticklabels={$y_0$}, ylabel=$y^T$,
		xmin=-1, xmax=11, xtick={1.38,8,10}, xticklabels={$t_1$,$t_2$,$T$}, xlabel=$t$,
		domain=0:10,samples=101, % added
		]
		
		{\addplot [blue,line width=0.1cm] {func(x)};}
		%\addlegendentry{quasi-optimum}
		{\addplot [black,line width=0.1cm] {(\x < 1.36) * (pow(\x,2)-3.12*\x+4)   +
				and(\x >= 1.36, \x < 8) * (-0.43*pow(\x,2)+4.1709*\x -3.2671)     +
				(\x >= 8) * (0.24*pow(\x-8,2)+2.58)};}
		\draw [fill, blue] (24,26) circle [radius=1.8];
		\draw [dashed] (24,10) -- (24,26);
		\draw [fill, blue] (90,36) circle [radius=1.8];
		\draw [dashed] (90,10) -- (90,36);
		%\addlegendentry{optimum}
		\node [below left, black] at (10,10) {$O$};
		\end{axis}
		\end{tikzpicture}
		\caption{The idea of the proof of Lemma \ref{lemma_bound_optima} is to use controllability for \cref{semilinear_internal_1_slt} to show that optima for \cref{semilinear_internal_1_slt}-\cref{functional_slt} cannot oscillate too much. Indeed, consider a the time interval $[t_1,t_2]$. By controllability, we can link $y^T\left(t_1,\cdot\right)$ and $y^T\left(t_2,\cdot\right)$ by a controlled trajectory (in blue). By optimality, the optimum (in black) is bounded by the constructed trajectory.}
		\label{oscillations_estimate}
	\end{center}
\end{figure}

In order to carry out the proof of Lemma \ref{lemma_bound_optima}, we need some preliminary lemmas. We start by stating some results on the controllability of a dissipative semilinear heat equation.

\subsection{Controllability of dissipative semilinear heat equation}
\label{appendixsubsec:Controllability dissipative semilinear heat equation}

\begin{lemma}\label{lemma_controllab_semil}
	Let $\hat{y}\in \xLinfty((0,+\infty)\times \Omega)$ be a target trajectory, solution to
	\begin{equation}\label{}
	\begin{dcases}
	\hat{y}_t-\Delta \hat{y}+f\left(\hat{y}\right)=\hat{u}\chi_{\omega}\hspace{2.8 cm} & \mbox{in} \hspace{0.10 cm}(0,T)\times\Omega\\
	\hat{y}=0  & \mbox{on}\hspace{0.10 cm} (0,T)\times \partial \Omega,
	\end{dcases}
	\end{equation}
	with control $\hat{u}\in \xLinfty((0,T)\times \omega)$. Let $y_0\in \xLinfty\left(\Omega\right)$ be an initial datum. Let $R>0$. Suppose $\|y_0\|_{\xLinfty\left(\Omega\right)}\leq R$ and $\|\hat{y}\|_{\xLinfty((0,+\infty)\times \Omega)}\leq R$. Then, there exists $T_R=T_R(\Omega,f,\omega,R)$, such that for any $T\geq T_R$ there exists $u\in \xLinfty((0,T)\times \omega)$ such that the solution $y$ to the controlled equation \cref{semilinear_internal_1_slt}, with initial datum $y_0$ and control $u$, verifies the final condition
	\begin{equation}\label{final_condition}
	y(T,x)=\hat{y}(T,x)\hspace{0.3 cm}\mbox{in} \ \Omega
	\end{equation}
	and
	\begin{equation}\label{est_Linf_contr}
	\left\|u-\hat{u}\right\|_{\xLinfty((0,T)\times \omega)}\leq K\|y_0-\hat{y}(0)\|_{\xLinfty\left(\Omega\right)},
	\end{equation}
	where the constant $K$ depends only on $\Omega$, $f$, $\omega$ and $R$.
\end{lemma}
The proof of the above lemma is classical (see, e.g. \cite{EFR,DTA}).

In order to prove Lemma \ref{lemma_bound_optima}, we introduce an optimal control problem, with specified terminal states. Let $t_1<t_2$. Let $\hat{y}$ be a target trajectory, bounded solution to \cref{semilinear_internal_1_slt} in $(t_1,t_2)$, i.e.
\begin{equation}\label{semilinear_internal_3}
\begin{dcases}
\hat{y}_t-\Delta \hat{y}+f(\hat{y})=\hat{u}\chi_{\omega}\hspace{2.8 cm} & \mbox{in} \hspace{0.10 cm}(t_1,t_2)\times\Omega\\
\hat{y}=0  & \mbox{on}\hspace{0.10 cm} (t_1,t_2)\times \partial \Omega\\
\hat{y}(t_1,x)=\hat{y}_0(x)  & \mbox{in}\hspace{0.10 cm}  \Omega,
\end{dcases}
\end{equation}
with initial datum $\hat{y}_0\in \xLinfty\left(\Omega\right)$ and control $\hat{u}\in \xLinfty\left(\left(t_1,t_2\right)\times \omega\right)$.

For any control $u\in \xLtwo((t_1,t_2)\times \omega)$, the corresponding state $y$ is the solution to:
\begin{equation}\label{semilinear_internal_2}
\begin{dcases}
y_t-\Delta y+f\left(y\right)=u\chi_{\omega}\hspace{2.8 cm} & \mbox{in} \hspace{0.10 cm}(t_1,t_2)\times\Omega\\
y=0  & \mbox{on}\hspace{0.10 cm} (t_1,t_2)\times \partial \Omega\\
y(t_1,x)=\hat{y}(t_1,x)  & \mbox{in}\hspace{0.10 cm}  \Omega.
\end{dcases}
\end{equation}
We introduce the set of admissible controls\\
\begin{equation*}
\mathscr{U}_{\mbox{\tiny{ad}}}\coloneqq \left\{u\in \xLtwo((t_1,t_2)\times \omega) \ | \ y(t_2,\cdot)=\hat{y}(t_2,\cdot)\right\}.
\end{equation*}
By definition, $\hat{u}\in \mathscr{U}_{\mbox{\tiny{ad}}}$. Hence, $\mathscr{U}_{\mbox{\tiny{ad}}}\neq \varnothing$. We consider the optimal control problem
\begin{equation}\label{functional_specified_terminal_states}
\min_{u\in \mathscr{U}_{\mbox{\tiny{ad}}}}J_{t_1,t_2}(u)=\frac12 \int_{t_1}^{t_2}\int_{\omega} |u|^2 dxdt+\frac{\beta}{2}\int_{t_1}^{t_2}\int_{\omega_0} |y-z|^2 dxdt,
\end{equation}
with running target $z\in \xLinfty(\omega_0)$. By the direct methods in the calculus of variations, the functional $J_{t_1,t_2}$ admits a global minimizer in the set of admissible controls $\mathscr{U}_{\mbox{\tiny{ad}}}$.

We now bound the minimal value of the functional \cref{functional_specified_terminal_states}, showing that the magnitude of the control $\hat{u}$ in the time interval $[t_1,t_2-T_R]$ can be neglected when estimating the cost of controllability. Namely, what matters is the norm of $\hat{u}$ in the final time interval $[t_2-T_R,t_2]$.

\begin{lemma}\label{lemma_optima_bounds_specified_terminal_states}
	Consider the optimal control problem \cref{semilinear_internal_2}-\cref{functional_specified_terminal_states}, with $t_2-t_1\geq T_R$. Then,
	\begin{eqnarray}\label{lemma_optima_bounds_specified_terminal_states_eq_1}
	\min_{\mathscr{U}_{\mbox{\tiny{ad}}}}J_{t_1,t_2}&\leq&K\left[\|\hat{y}(t_1,\cdot)\|_{\xLinfty\left(\Omega\right)}^2+(t_2-t_1)\|z\|_{\xLinfty(\omega_0)}^2\right.\nonumber\\
	&\;&\left.+\|\hat{u}\|_{\xLinfty((t_2-T_R,t_2)\times \omega)}^2+\|\hat{y}(t_2-T_R,\cdot)\|_{\xLinfty\left(\Omega\right)}^2\right],\nonumber\\
	\end{eqnarray}
	the constant $K$ being independent of the time horizon $t_2-t_1\geq T_R$.
\end{lemma}
\begin{proof}[Proof of Lemma \ref{lemma_optima_bounds_specified_terminal_states}]
	\textit{Step 1} \ \textbf{A quasi-optimal control}\\
	Let $u_{\mbox{\tiny{opt}}}$ be an optimal control for \eqref{semilinear_internal_2}-\eqref{functional_specified_terminal_states} and let $y_{\mbox{\tiny{opt}}}$ be its corresponding state. To get the desired bound, we introduce a quasi-optimal control $u$ for \cref{semilinear_internal_2}-\cref{functional_specified_terminal_states}, linking $\hat{y}(t_1,\cdot)$ and $y_{\mbox{\tiny{opt}}}\left(t_2,\cdot\right)$.
	The control strategy is the following
	\begin{enumerate}
		\item employ null control for time $t\in [t_1,t_2-T_R]$;
		\item match the final condition by control $w$, for $t\in [t_2-T_R,t_2]$.
	\end{enumerate}
	Let us denote by $y^0$ the solution to the semilinear problem with null control
	\begin{equation}\label{semilinear_internal_4}
	\begin{dcases}
	y^0_t-\Delta y^0+f\left(y^0\right)=0\hspace{2.8 cm} & \mbox{in} \hspace{0.10 cm}(t_1,t_2)\times\Omega\\
	y^0=0  & \mbox{on}\hspace{0.10 cm} (t_1,t_2)\times \partial \Omega\\
	y^0(t_1,x)=\hat{y}(t_1,x)  & \mbox{in}\hspace{0.10 cm}  \Omega.
	\end{dcases}
	\end{equation}
	By Lemma \ref{lemma_controllab_semil}, there exists $w\in \xLinfty((t_2-T_R,t_2)\times \omega)$, steering \cref{semilinear_internal_2} from $y^0(t_2-T_R,\cdot)$ to $\hat{y}\left(t_2,\cdot\right)$ in the time interval $(t_2-T_R,t_2)$, with estimate
	\begin{equation}\label{est_Linf_contr_3}
	\left\|w-\hat{u}\right\|_{\xLinfty\left(\left(t_2-T_R,t_2\right)\times \omega\right)}\leq K\left\|y^0\left(t_2-T_R\right)-\hat{y}\left(t_2-T_R\right)\right\|_{\xLinfty\left(\Omega\right)},
	\end{equation}
	Then, set
	\begin{equation}\label{special_control}
	u\coloneqq\begin{dcases}
	0 \quad &\mbox{in} \ \left(0,t_2-T_R\right)\\
	w \quad &\mbox{in} \ \left(t_2-T_R,t_2\right).
	\end{dcases}
	\end{equation}
	By \cref{est_Linf_contr_3}, we can bound the norm of the control,
	\begin{equation}\label{est_control}
	\left\|u\right\|_{\xLinfty((t_1,t_2)\times \omega)}\leq K\left[\left\|y^0\left(t_2-T_R\right)-\hat{y}\left(t_2-T_R\right)\right\|_{\xLinfty\left(\Omega\right)}+\left\|\hat{u}\right\|_{\xLinfty\left(\left(t_2-T_R,t_2\right)\times \omega\right)}\right].
	\end{equation}
	\textit{Step 2} \ \textbf{Conclusion}\\
	Consider the control $u$ introduced in \cref{special_control} and let $y$ be the solution to \cref{semilinear_internal_2}, with initial datum $y_0$ and control $u$. Then, we have
	\begin{eqnarray}
	\min_{\mathscr{U}_{\mbox{\tiny{ad}}}}J_{t_1,t_2}&\leq&J_{t_1,t_2}(u)\nonumber\\
	&=&\frac12 \int_{t_1}^{t_2}\int_{\omega} |u|^2 dxdt+\frac{\beta}{2}\int_{t_1}^{t_2}\int_{\omega_0} |y-z|^2 dxdt\nonumber\\
	&=&\frac12 \int_{t_2-T_R}^{t_2}\int_{\omega} |w|^2 dxdt+\frac{\beta}{2}\int_{t_1}^{t_2}\int_{\omega_0} |y-z|^2 dxdt\nonumber\\
	&\leq &\frac12 \int_{t_2-T_R}^{t_2}\int_{\omega} |w|^2 dxdt+{\beta}\int_{t_1}^{t_2}\int_{\omega_0} |y|^2 dxdt+{\beta}\int_{t_1}^{t_2}\int_{\omega_0} |z|^2 dxdt\nonumber\\
	&\leq &\frac12 \int_{t_2-T_R}^{t_2}\int_{\omega} |w|^2 dxdt+{\beta}\int_{t_1}^{t_2}\int_{\omega_0} |y|^2 dxdt+K (t_2-t_1)\|z\|_{\xLinfty(\omega_0)}^2\nonumber\\
	&\leq &K\left[\|w\|_{\xLinfty((t_2-T_R,t_2)\times \omega)}^2+(t_2-t_1)\|z\|_{\xLinfty(\omega_0)}^2\right. \nonumber\\
	&\;&\left.+{\beta}\int_{t_1}^{t_2-T_R}\left\|y^0(t,\cdot)\right\|_{\xLtwo\left(\Omega\right)}^2 dt+\left\|y\right\|_{\xLtwo((t_2-T_R,t_2)\times \Omega)}^2\right]\nonumber\\
	&\leq &K\left[\left\|y^0\left(t_2-T_R,\cdot\right)-\hat{y}\left(t_2-T_R,\cdot\right)\right\|_{\xLinfty\left(\Omega\right)}^2+\left\|\hat{u}\right\|_{\xLinfty\left(\left(t_2-T_R,t_2\right)\times \omega\right)}^2\right.\label{weemploy_dissip_1}\\
	&\; &\left.+(t_2-t_1)\|z\|_{\xLinfty(\omega_0)}^2+\left\|\hat{y}(t_1,\cdot)\right\|_{\xLinfty\left(\Omega\right)}^2\right]\nonumber\\
	&\leq &K\left[\|\hat{y}(t_1,\cdot)\|_{\xLinfty\left(\Omega\right)}^2+(t_2-t_1)\|z\|_{\xLinfty(\omega_0)}^2\right.\label{weemploy_dissip_2}\\
	&\;&\left.+\|\hat{u}\|_{\xLinfty((t_2-T_R,t_2)\times \omega)}^2+\|\hat{y}(t_2-T_R,\cdot)\|_{\xLinfty\left(\Omega\right)}^2\right],\nonumber\\
	\end{eqnarray}
	where in \eqref{weemploy_dissip_1} and in \eqref{weemploy_dissip_2} we have employed the dissipativity of \cref{semilinear_internal_4}. This concludes the proof.
\end{proof}

\subsection{A mean value result for integrals}
\label{appendixsubsec:A mean value result for integrals}

In the following Lemma we estimate the value of a function at some point, with the value of its integral.

\begin{lemma}\label{lemma_int_point}
	Let $h\in \xLone(c,d)\cap \xCzero(c,d)$, with $-\infty<c<d<+\infty$. Assume $h\geq0$ a.e. in $(c,d)$. Then,
	\begin{enumerate}
		\item there exists $t_c\in \left(c,c+\frac{d-c}{3}\right)$, such that
		\begin{equation*}
		h(t_c)\leq\frac{3}{d-c}\int_c^dhdt;
		\end{equation*}
		\item there exists $t_d\in \left(d-\frac{d-c}{3},d\right)$, such that
		\begin{equation*}
		h(t_d)\leq\frac{3}{d-c}\int_c^dhdt.
		\end{equation*}
	\end{enumerate}
\end{lemma}
\begin{proof}[Proof of Lemma \ref{lemma_int_point}]
	%\textit{Step 1} \  \textbf{Proof of (1.)}
	By contradiction, for any $t\in \left(c,c+\frac{d-c}{3}\right)$, $h(t)>\frac{3}{d-c}\int_c^dhds$. Then, we have
	\begin{equation*}
	\int_{c}^{d}hdt\geq \int_{c}^{c+\frac{d-c}{3}}hdt>\int_{c}^{c+\frac{d-c}{3}}\left[\frac{3}{d-c}\int_c^dhds\right]dt=\int_c^dhds,
	\end{equation*}
	so obtaining a contradiction. The proof of (2.) is similar.
\end{proof}

\subsection{Proof of Lemma \ref{lemma_bound_optima}}
\label{appendixsubsec:Proof of Lemma reflemma_bound_optima}
We are now in position to prove Lemma \ref{lemma_bound_optima}.

\begin{proof}[Proof of Lemma \ref{lemma_bound_optima}]
	%\textit{Step 0} \  \textbf{Boundedness of $u^T$ and $y^T$}\\
	%	By bootstrapping in the optimality system \cref{semilinear_internal_parabolic_1}, we have the boundedness of the optimal control $u^T$, the optimal state $y^T$ and the adjoint state $q^T$.\\
	%	\textit{Step 1} \  \textbf{Upper bound for the averages}\\
	%	By Lemma \ref{lemma_reg_nonnegativepotential},
	%	\begin{eqnarray}
	%	\int_0^T\left[\|y^T(t,\cdot)\|_{\xLinfty\left(\Omega\right)}^2+\|q^T(t,\cdot)\|_{\xLinfty\left(\Omega\right)}^2\right]dt&\leq&K\left[ \|y_0\|_{\xLinfty\left(\Omega\right)}^2+\left\|y^T\right\|_{\xLtwo((0,T)\times \omega_0)}^2+\|u^T\|_{\xLtwo\left(\Omega\right)}^2 \right]\nonumber\\
	%	&\leq&K\left[ \|y_0\|_{\xLinfty\left(\Omega\right)}^2+J_{T}\left(u^T\right) \right]\nonumber\\
	%	&\leq&K\left[ \|y_0\|_{\xLinfty\left(\Omega\right)}^2+J_{T}(0) \right]\nonumber\\
	%	&\leq&K\left[ \|y_0\|_{\xLinfty\left(\Omega\right)}^2+T\|z\|_{\xLinfty(\omega_0)}^2 \right]\nonumber,
	%	\end{eqnarray}
	%	the constant $K$ being independent of the time horizon $T>0$. At this point, we divide the above inequality by $T$, getting
	%	\begin{equation}
	%	\frac{1}{T}\int_0^T\left[\|y^T(t,\cdot)\|_{\xLinfty\left(\Omega\right)}^2+\|q^T(t,\cdot)\|_{\xLinfty\left(\Omega\right)}^2\right]dt\leq K\left[ \|y_0\|_{\xLinfty\left(\Omega\right)}^2+\|z\|_{\xLinfty(\omega_0)}^2 \right].
	%	\end{equation}
	\textit{Step 1} \ \textbf{Estimates on subintervals}\\
	Let $T_R$ be given by Lemma \ref{lemma_controllab_semil}.
	
	The case $T\leq 6T_R$ can be addressed by employing the inequality $J_T\left(u^T\right)\leq J_T\left(0\right)$ and bootstrapping in the optimality system \cref{semilinear_internal_parabolic_1}, as in \cite[subsection 3.2]{PZ2}.
	
	We address now the case $T> 6T_R$.
	
	Set $N_T\coloneqq \left\lfloor\frac{T}{3T_R}\right\rfloor$. Arbitrarily fix $\theta >0$, a degree of freedom, to be made precise later. Consider the indexes $i\in \left\{1,\dots,N_T\right\}$, such that
	\begin{equation}\label{theta_smallness}
	\int_{(i-1)3T_R}^{i3T_R}\left[\|q^T(t)\|_{\xLinfty\left(\Omega\right)}^2+\|y^T(t)\|_{\xLinfty\left(\Omega\right)}^2\right]dt\leq \theta \left[\|y_0\|_{\xLinfty\left(\Omega\right)}^2+\|z\|_{\xLinfty(\omega_0)}^2\right].
	\end{equation}
	Set
	\begin{equation}\label{definition_I}
	\mathscr{I}_T\coloneqq \left\{i\in \left\{1,\dots,N_T\right\} \ \bigg| \ \mbox{the estimate}\hspace{0.16 cm}\cref{theta_smallness}\hspace{0.16 cm}\mbox{is not verified}\right\}.
	\end{equation}
	On the one hand, for any $i\in \left\{1,\dots,N_T\right\}\setminus \mathscr{I}_T$, by definition of $\mathscr{I}_T$
	\begin{equation*}
	\int_{(i-1)3T_R}^{i3T_R}\left[\|q^T(t)\|_{\xLinfty\left(\Omega\right)}^2+\|y^T(t)\|_{\xLinfty\left(\Omega\right)}^2\right]dt\leq \theta \left[\|y_0\|_{\xLinfty\left(\Omega\right)}^2+\|z\|_{\xLinfty(\omega_0)}^2\right].
	\end{equation*}
	On the other hand, for every $i\in \mathscr{I}_T$, we seek to prove the existence of a constant $K_{\theta}=K_{\theta}(\Omega,f,R,\theta)$, possibly larger than $\theta$, such that
	\begin{equation}\label{lemma_bound_optima_eq60}
	\int_{(i-1)3T_R}^{i3T_R}\left[\|q^T(t)\|_{\xLinfty\left(\Omega\right)}^2+\|y^T(t)\|_{\xLinfty\left(\Omega\right)}^2\right]dt\leq K_{\theta} \left[\|y_0\|_{\xLinfty\left(\Omega\right)}^2+\|z\|_{\xLinfty(\omega_0)}^2\right].	
	\end{equation}
	We start by considering the union of time intervals, where \cref{theta_smallness} is not verified
	\begin{equation*}
	\mathscr{W}_T\coloneqq \bigcup_{i\in \mathscr{I}_T}[(i-1)3T_R,i3T_R].
	\end{equation*}
	The above set is made of a finite union of disjoint closed intervals, namely there exists a natural $M$ and $\left\{\left(a_j,b_j\right)\right\}_{j=1,\dots,M}$, such that
	\begin{equation*}
	b_j<a_{j+1},\hspace{0.3 cm}j=1,\dots,M-1
	\end{equation*}
	and
	\begin{equation*}
	\mathscr{W}_T=\bigcup_{i\in \mathscr{I}_T}[(i-1)3T_R,i3T_R]=\bigcup_{j=1,\dots,M}[a_j,b_j].
	\end{equation*}
	For any $j=1,\dots,M$, set
	\begin{equation}\label{def_Cj}
	C_j\coloneqq \left\{i\in \mathscr{I}_T \ | \ [(i-1)3T_R,i3T_R]\subseteq[a_j,b_j]\right\}.
	\end{equation}
	We are going to prove \cref{lemma_bound_optima_eq60}, studying the optima in a neighbourhood of $[a_j,b_j]$, for $j=1,\dots,M$. Three different cases may occur:
	\begin{itemize}
		\item \textbf{Case 1.} $a_1=0$ and $b_1<3T_RN_T$, namely the left end of the interval $[a_1,b_1]$ coincides with $t=0$, while the right end is far from $t=T$;
		\item \textbf{Case 2.} $a_j>0$ and $b_j<3T_RN_T$, i.e. the left end of the interval $[a_j,b_j]$ is far from $t=0$ and the right end is far from $t=T$;
		\item \textbf{Case 3.} $a_j>0$ and $b_j=3T_RN_T$, i.e. the left end of the interval $[a_j,b_j]$ is far from $t=0$, while the right end is close to $t=T$.
		%case $a_j=0$ and $b_j=3T_RN_T$ cannot occur because $T\in $\left(6T_R,+\infty\right)$.
	\end{itemize}

	\vspace*{4pt}\noindent\textbf{Case 1.} $a_1=0$ and $b_1<3T_RN_T$.
	
	Since $b_1<3T_RN_T$, we have $[b_1,b_1+3T_RN_T]\subseteq [0,T]\setminus \mathscr{W}_T$. Hence, by \cref{definition_I},
	\begin{equation*}
	\int_{b_1}^{b_1+3T_R}\left[\|q^T(t)\|_{\xLinfty\left(\Omega\right)}^2+\|y^T(t)\|_{\xLinfty\left(\Omega\right)}^2\right]dt\leq\theta\left[\|y_0\|_{\xLinfty\left(\Omega\right)}^2+\|z\|_{\xLinfty(\omega_0)}^2\right].
	\end{equation*}
	Set $c\coloneqq b_1$, $d\coloneqq b_1+3T_R$ and $h(t)\coloneqq \|q^T(t)\|_{\xLinfty\left(\Omega\right)}^2+\|y^T(t)\|_{\xLinfty\left(\Omega\right)}^2$. By Lemma \ref{lemma_int_point}, there exist $t_c$ and $t_d$,
	\begin{equation}\label{Case1_tctd}
	b_1<t_c<b_1+T_R\hspace{0.3 cm}\mbox{and}\hspace{0.3 cm}b_1+2T_R<t_d<b_1+3T_R,
	\end{equation}
	such that
	\begin{eqnarray*}
		\|q^T(t_c)\|_{\xLinfty\left(\Omega\right)}^2+\|y^T(t_c)\|_{\xLinfty\left(\Omega\right)}^2&\leq&\frac{1}{T_R}\int_{b_1}^{b_1+3T_R}\left[\|q^T(t)\|_{\xLinfty\left(\Omega\right)}^2+\|y^T(t)\|_{\xLinfty\left(\Omega\right)}^2\right]dt\nonumber\\
		&\leq&\frac{\theta}{T_R}\left[\|y_0\|_{\xLinfty\left(\Omega\right)}^2+\|z\|_{\xLinfty(\omega_0)}^2\right]
	\end{eqnarray*}
	and
	\begin{eqnarray*}
		\|q^T(t_d)\|_{\xLinfty\left(\Omega\right)}^2+\|y^T(t_d)\|_{\xLinfty\left(\Omega\right)}^2&\leq&\frac{1}{T_R}\int_{b_1}^{b_1+3T_R}\left[\|q^T(t)\|_{\xLinfty\left(\Omega\right)}^2+\|y^T(t)\|_{\xLinfty\left(\Omega\right)}^2\right]dt\nonumber\\
		&\leq&\frac{\theta}{T_R}\left[\|y_0\|_{\xLinfty\left(\Omega\right)}^2+\|z\|_{\xLinfty(\omega_0)}^2\right].
	\end{eqnarray*}
	Parabolic regularity
	%[ADD]
	in the optimality system \cref{semilinear_internal_parabolic_1} in the interval $[t_c,t_d]$ gives
	\begin{eqnarray}\label{Linf_bound_nice_part_case1}
	\left\|y^T\right\|_{\xLinfty((t_c,t_d)\times \Omega)}^2+\left\|q^T\right\|_{\xLinfty((t_c,t_d)\times \Omega)}^2&\leq&K\left\{\|q^T(t_d)\|_{\xLinfty\left(\Omega\right)}^2+\|y^T(t_c)\|_{\xLinfty\left(\Omega\right)}^2+\|z\|_{\xLinfty(\omega_0)}^2\right.\nonumber\\
	&\;&\left.+\int_{b_1}^{b_1+3T_R}\left[\|q^T(t)\|_{\xLinfty\left(\Omega\right)}^2+\|y^T(t)\|_{\xLinfty\left(\Omega\right)}^2\right]dt\right\}\nonumber\\
	&\leq &K_{\theta}\left[\|y_0\|_{\xLinfty\left(\Omega\right)}^2+\|z\|_{\xLinfty(\omega_0)}^2\right].
	\end{eqnarray}
	where the constant $K_{\theta}$ is independent of the time horizon $T$, but it depends on $\theta$. At this point, we want to apply Lemma \ref{lemma_optima_bounds_specified_terminal_states}. To this purpose, we set up a control problem like \cref{semilinear_internal_2}-\cref{functional_specified_terminal_states} with specified final state
	\begin{eqnarray*}
		\hat{y}&\coloneqq&y^T\nonumber\\
		t_1&\coloneqq &0\nonumber\\
		t_2&\coloneqq &t_d.
	\end{eqnarray*}
	Since $t_d> T_R$, assumptions of Lemma \ref{lemma_optima_bounds_specified_terminal_states} are satisfied. Then, by \eqref{lemma_optima_bounds_specified_terminal_states_eq_1} and \eqref{Linf_bound_nice_part_case1},
	\begin{eqnarray}\label{lemma_optima_bounds_specified_terminal_states_eq_1_appl_1}
	\min_{\mathscr{U}_{\mbox{\tiny{ad}}}}J_{t_1,t_2}&\leq& K\left[\|y_0\|_{\xLinfty\left(\Omega\right)}^2+t_d\|z\|_{\xLinfty(\omega_0)}^2\right.\nonumber\\
	&\;&\left.+\|u^T\|_{\xLinfty((t_d-T_R,t_d)\times \omega)}^2+\|y^T(t_d-T_R)\|_{\xLinfty\left(\Omega\right)}^2\right]\nonumber\\
	&\leq&K_{\theta}\left[\|y_0\|_{\xLinfty\left(\Omega\right)}^2+\|z\|_{\xLinfty(\omega_0)}^2\right]+\gamma t_d\|z\|_{\xLinfty(\omega_0)}^2,
	\end{eqnarray}
	where $K_{\theta}=K_{\theta}(\Omega,f,R,\theta)$ and $\gamma=\gamma(\Omega,f,R)$. In our case the target trajectory for \cref{semilinear_internal_2}-\cref{functional_specified_terminal_states} is the state $y^T$ associated to an optimal control $u^T$ for \cref{semilinear_internal_1_slt}-\cref{functional_slt}. Then, by definition of \cref{semilinear_internal_2}-\cref{functional_specified_terminal_states},
	\begin{equation*}
	J_{t_1,t_2}\left(u^T\right)\leq J_{t_1,t_2}(u),\hspace{0.3 cm}\forall \ u\in \mathscr{U}_{\mbox{\tiny{ad}}}.
	\end{equation*}
	Hence, by \eqref{lemma_optima_bounds_specified_terminal_states_eq_1_appl_1},
	\begin{eqnarray}\label{case1_eq16}
	J_{t_1,t_2}\left(u^T\right)&\leq&\min_{\mathscr{U}_{\mbox{\tiny{ad}}}}J_{t_1,t_2}\nonumber\\
	&\leq&K_{\theta}\left[\|y_0\|_{\xLinfty\left(\Omega\right)}^2+\|z\|_{\xLinfty(\omega_0)}^2\right]+\gamma t_d\|z\|_{\xLinfty(\omega_0)}^2.
	\end{eqnarray}
	By definition of $\mathscr{I}_T$ \cref{definition_I} and $C_1$ \cref{def_Cj}, we have
	\begin{eqnarray*}
		\int_{0}^{b_1}\left[\|q^T(t)\|_{\xLinfty\left(\Omega\right)}^2+\|y^T(t)\|_{\xLinfty\left(\Omega\right)}^2\right]dt&\geq&\sum_{i\in C_1}\theta\left[\|y_0\|_{\xLinfty\left(\Omega\right)}^2+\|z\|_{\xLinfty(\omega_0)}^2\right]\nonumber\\
		&=&\frac{\theta b_1}{3T_R}\left[\|y_0\|_{\xLinfty\left(\Omega\right)}^2+\|z\|_{\xLinfty(\omega_0)}^2\right]\nonumber\\
		&>&\frac{\theta (t_d-3T_R)}{3T_R}\left[\|y_0\|_{\xLinfty\left(\Omega\right)}^2+\|z\|_{\xLinfty(\omega_0)}^2\right],
	\end{eqnarray*}
	where in the last inequality we have used \cref{Case1_tctd}, which yields $b_1>t_d-3T_R$. By the above inequality, Lemma \ref{lemma_reg_nonnegativepotential}, \eqref{Linf_bound_nice_part_case1} and \eqref{case1_eq16},
	\begin{eqnarray*}
		\frac{\theta(t_d-3T_R)}{6T_R}\left[\|y_0\|_{\xLinfty\left(\Omega\right)}^2+\|z\|_{\xLinfty(\omega_0)}^2\right]&\;&\nonumber\\
		+\frac12\int_{0}^{b_1}\left[\|q^T(t)\|_{\xLinfty\left(\Omega\right)}^2+\|y^T(t)\|_{\xLinfty\left(\Omega\right)}^2\right]dt&\leq&\int_{0}^{b_1}\left[\|q^T(t)\|_{\xLinfty\left(\Omega\right)}^2+\|y^T(t)\|_{\xLinfty\left(\Omega\right)}^2\right]dt\nonumber\\
		&\leq&K\left[J_{t_1,t_2}\left(u^T\right)+\|y_0\|_{\xLinfty\left(\Omega\right)}^2+\|q^T(t_d)\|_{\xLinfty\left(\Omega\right)}^2\right]\nonumber\\
		&\leq&K_{\theta}\left[\|y_0\|_{\xLinfty\left(\Omega\right)}^2+\|z\|_{\xLinfty(\omega_0)}^2\right]+\gamma t_d\|z\|_{\xLinfty\left(\omega_0\right)}^2,
	\end{eqnarray*}
	whence
	\begin{eqnarray*}
		\int_{0}^{b_1}\left[\|q^T(t)\|_{\xLinfty\left(\Omega\right)}^2+\|y^T(t)\|_{\xLinfty\left(\Omega\right)}^2\right]dt&\leq& K_{\theta}\left[\|y_0\|_{\xLinfty\left(\Omega\right)}^2+\|z\|_{\xLinfty(\omega_0)}^2\right]\nonumber\\
		&\;&+2\left(\gamma t_d-\frac{\theta(t_d-3T_R)}{6T_R}\right) \|z\|_{\xLinfty\left(\omega_0\right)}^2\nonumber\\
		&\leq & K_{\theta}\left[\|y_0\|_{\xLinfty\left(\Omega\right)}^2+\|z\|_{\xLinfty(\omega_0)}^2\right]\nonumber\\
		&\;&+2t_d\left(\gamma -\frac{\theta}{6T_R}\right) \|z\|_{\xLinfty\left(\omega_0\right)}^2.
	\end{eqnarray*}
	If $\theta$ is large enough, we have $\gamma -\frac{\theta}{6T_R}<0$. Hence, choosing $\theta$ large enough, we obtain the estimate
	\begin{equation*}
	\int_{0}^{b_1}\left[\|q^T(t)\|_{\xLinfty\left(\Omega\right)}^2+\|y^T(t)\|_{\xLinfty\left(\Omega\right)}^2\right]dt\leq K_{\theta}\left[\|y_0\|_{\xLinfty\left(\Omega\right)}^2+\|z\|_{\xLinfty(\omega_0)}^2\right].
	\end{equation*}

	\vspace*{4pt}\noindent\textbf{Case 2.} $a_j>0$ and $b_j<3T_RN_T$.
	
	Since $a_j>0$ and $b_j<3T_RN_T$, we have
	\begin{equation}\label{case2_eq1}
	\int_{a_j-3T_R}^{a_j}\left[\|q^T(t)\|_{\xLinfty\left(\Omega\right)}^2+\|y^T(t)\|_{\xLinfty\left(\Omega\right)}^2\right]dt\leq\theta\left[\|y_0\|_{\xLinfty\left(\Omega\right)}^2+\|z\|_{\xLinfty(\omega_0)}^2\right]
	\end{equation}
	and
	\begin{equation}\label{case2_eq2}
	\int_{b_j}^{b_j+3T_R}\left[\|q^T(t)\|_{\xLinfty\left(\Omega\right)}^2+\|y^T(t)\|_{\xLinfty\left(\Omega\right)}^2\right]dt\leq\theta\left[\|y_0\|_{\xLinfty\left(\Omega\right)}^2+\|z\|_{\xLinfty(\omega_0)}^2\right].
	\end{equation}
	In Case 2, we apply Lemma \ref{lemma_int_point}:
	\begin{itemize}
		\item in the interval $[a_j-3T_R,a_j]$;
		\item in the interval $[b_j,b_j+3T_R]$.
	\end{itemize}
	
	We start by applying Lemma \ref{lemma_int_point} in $[a_j-3T_R,a_j]$. To this end, set $c\coloneqq a_j-3T_R$, $d\coloneqq a_j$ and $h(t)\coloneqq \|q^T(t)\|_{\xLinfty\left(\Omega\right)}^2+\|y^T(t)\|_{\xLinfty\left(\Omega\right)}^2$. By Lemma \ref{lemma_int_point}, there exist $t_{a,c}$ and $t_{a,d}$,
	\begin{equation}\label{case2_eq3}
	a_j-3T_R<t_{a,c}<a_j-2T_R\hspace{0.3 cm}\mbox{and}\hspace{0.3 cm}a_j-T_R<t_{a,d}<a_j,
	\end{equation}
	such that
	\begin{eqnarray*}
		\|q^T(t_{a,c})\|_{\xLinfty\left(\Omega\right)}^2+\|y^T(t_{a,c})\|_{\xLinfty\left(\Omega\right)}^2&\leq&\frac{1}{T_R}\int_{a_j-3T_R}^{a_j}\left[\|q^T(t)\|_{\xLinfty\left(\Omega\right)}^2+\|y^T(t)\|_{\xLinfty\left(\Omega\right)}^2\right]dt\nonumber\\
		&\leq&\frac{\theta}{T_R}\left[\|y_0\|_{\xLinfty\left(\Omega\right)}^2+\|z\|_{\xLinfty(\omega_0)}^2\right]
	\end{eqnarray*}
	and
	\begin{eqnarray*}
		\|q^T(t_{a,d})\|_{\xLinfty\left(\Omega\right)}^2+\|y^T(t_{a,d})\|_{\xLinfty\left(\Omega\right)}^2&\leq&\frac{1}{T_R}\int_{a_j-3T_R}^{a_j}\left[\|q^T(t)\|_{\xLinfty\left(\Omega\right)}^2+\|y^T(t)\|_{\xLinfty\left(\Omega\right)}^2\right]dt\nonumber\\
		&\leq&\frac{\theta}{T_R}\left[\|y_0\|_{\xLinfty\left(\Omega\right)}^2+\|z\|_{\xLinfty(\omega_0)}^2\right].
	\end{eqnarray*}
	By parabolic regularity
	%[ADD]
	in the optimality system \cref{semilinear_internal_parabolic_1} in the interval $[t_{a,c},t_{a,d}]$, we have
	\begin{eqnarray}\label{Linf_bound_nice_part_case2a}
	\left\|y^T\right\|_{\xLinfty((t_{a,c},t_{a,d})\times \Omega)}^2+\left\|q^T\right\|_{\xLinfty((t_{a,c},t_{a,d})\times \Omega)}^2&\leq&K\left\{\|q^T(t_{a,d})\|_{\xLinfty\left(\Omega\right)}^2+\|y^T(t_{a,c})\|_{\xLinfty\left(\Omega\right)}^2\right.\nonumber\\
	&\;&+\|z\|_{\xLinfty(\omega_0)}^2\nonumber\\
	&\;&\left.+\int_{a_j-3T_R}^{a_j}\left[\|q^T(t)\|_{\xLinfty\left(\Omega\right)}^2+\|y^T(t)\|_{\xLinfty\left(\Omega\right)}^2\right]dt\right\}\nonumber\\
	&\leq &K_{\theta}\left[\|y_0\|_{\xLinfty\left(\Omega\right)}^2+\|z\|_{\xLinfty(\omega_0)}^2\right].
	\end{eqnarray}
	where the constant $K_{\theta}$ is independent of the time horizon $T$, but it depends on $\theta$.
	
	We apply Lemma \ref{lemma_int_point} in $[b_j,b_j+3T_R]$. To this extent, set $c\coloneqq b_j$, $d\coloneqq b_j+3T_R$ and $h(t)\coloneqq \|q^T(t)\|_{\xLinfty\left(\Omega\right)}^2+\|y^T(t)\|_{\xLinfty\left(\Omega\right)}^2$. By Lemma \ref{lemma_int_point}, there exist $t_{b,c}$ and $t_{b,d}$,
	\begin{equation}\label{case2_eq6}
	b_j<t_{b,c}<b_j+T_R\hspace{0.3 cm}\mbox{and}\hspace{0.3 cm}b_j+2T_R<t_{b,d}<b_j+3T_R,
	\end{equation}
	such that
	\begin{eqnarray*}
		\|q^T(t_{b,c})\|_{\xLinfty\left(\Omega\right)}^2+\|y^T(t_{b,c})\|_{\xLinfty\left(\Omega\right)}^2&\leq&\frac{1}{T_R}\int_{b_j}^{b_j+3T_R}\left[\|q^T(t)\|_{\xLinfty\left(\Omega\right)}^2+\|y^T(t)\|_{\xLinfty\left(\Omega\right)}^2\right]dt\nonumber\\
		&\leq&\frac{\theta}{T_R}\left[\|y_0\|_{\xLinfty\left(\Omega\right)}^2+\|z\|_{\xLinfty(\omega_0)}^2\right]
	\end{eqnarray*}
	and
	\begin{eqnarray*}
		\|q^T(t_{b,d})\|_{\xLinfty\left(\Omega\right)}^2+\|y^T(t_{b,d})\|_{\xLinfty\left(\Omega\right)}^2&\leq&\frac{1}{T_R}\int_{b_j}^{b_j+3T_R}\left[\|q^T(t)\|_{\xLinfty\left(\Omega\right)}^2+\|y^T(t)\|_{\xLinfty\left(\Omega\right)}^2\right]dt\nonumber\\
		&\leq&\frac{\theta}{T_R}\left[\|y_0\|_{\xLinfty\left(\Omega\right)}^2+\|z\|_{\xLinfty(\omega_0)}^2\right].
	\end{eqnarray*}
	By parabolic regularity
	%[ADD]
	in the optimality system \cref{semilinear_internal_parabolic_1} in the interval $[t_{b,c},t_{b,d}]$, we have
	\begin{eqnarray}\label{Linf_bound_nice_part_case2b}
	\left\|y^T\right\|_{\xLinfty((t_{b,c},t_{b,d})\times \Omega)}^2+\left\|q^T\right\|_{\xLinfty((t_{b,c},t_{b,d})\times \Omega)}^2&\leq&K\left\{\|q^T(t_{b,d})\|_{\xLinfty\left(\Omega\right)}^2+\|y^T(t_{b,c})\|_{\xLinfty\left(\Omega\right)}^2\right.\nonumber\\
	&\;&+\|z\|_{\xLinfty(\omega_0)}^2+\int_{b_j}^{b_j+3T_R}\|q^T(t)\|_{\xLinfty\left(\Omega\right)}^2dt\nonumber\\
	&\;&\left.+\int_{b_j}^{b_j+3T_R}\|y^T(t)\|_{\xLinfty\left(\Omega\right)}^2dt\right\}\nonumber\\
	&\leq &K_{\theta}\left[\|y_0\|_{\xLinfty\left(\Omega\right)}^2+\|z\|_{\xLinfty(\omega_0)}^2\right].
	\end{eqnarray}
	where the constant $K_{\theta}$ is independent of the time horizon $T$, but it depends on $\theta$.
	
	%[CHECK] New figure.
	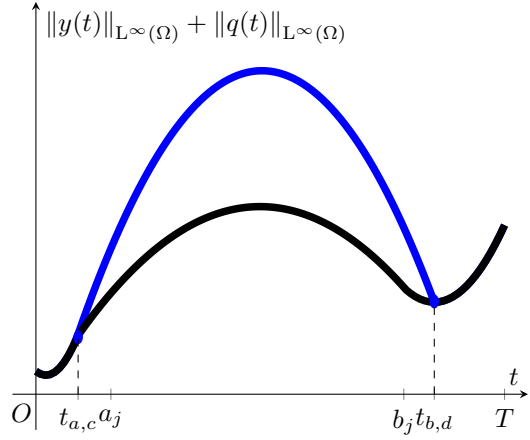
\begin{figure}
		\begin{center}
			\begin{tikzpicture}[
			declare function={
				func(\x)= (\x < 1.76) * (0.6*pow(\x - 1.76, 2) + 1.6*(\x - 1.76) + 1.6)   +
				and(\x >= 1.76, \x < 17) * (-0.12*(\x - 1.76)*(\x - 17) + (0.9701/(17-1.76))*(\x - 1.76) + 1.61)     +
				(\x >= 17) * (0.24*pow(\x-17,2)+2.58)
				;
			}
			]
			\begin{axis}[
			axis x line=middle, axis y line=middle,
			ymin=-1, ymax=11, ytick={0}, yticklabels={}, ylabel=$\left\|y(t)\right\|_{\xLinfty\left(\Omega\right)}+\left\|q(t)\right\|_{\xLinfty\left(\Omega\right)}$,
			xmin=-1, xmax=21, xtick={1.8,3.2,15.7,17,20}, xticklabels={$t_{a,c}$,$a_j$,$b_j$,$t_{b,d}$,$T$}, xlabel=$t$,
			domain=0:20,samples=101, % added
			]
			
			{\addplot [blue,line width=0.1cm] {func(x)};}
			%\addlegendentry{quasi-optimum}
			{\addplot [black,line width=0.1cm] {(\x < 1.76) * (0.6*pow(\x - 1.76, 2) + 1.6*(\x - 1.76) + 1.6)   +
					and(\x >= 1.76, \x < 15.8) * (-0.06*(\x - 1.76)*(\x - 15.8) + (1.3156/(15.8-1.76))*(\x - 1.76) + 1.61)     +
					(\x >= 15.8) * (0.24*pow(\x-17,2)+2.58};}
			\draw [fill, blue] (28,26) circle [radius=1.8];
			\draw [dashed] (28,10) -- (28,26);
			\draw [fill, blue] (180,36) circle [radius=1.8];
			\draw [dashed] (180,10) -- (180,36);
			%\addlegendentry{optimum}
			\node [below left, black] at (12.3,10) {$O$};
			\end{axis}
			\end{tikzpicture}
			\caption{Case $2$ of the proof of Lemma \ref{lemma_bound_optima}. By controllability, we can link $y^T\left(t_{a,c},\cdot\right)$ and $y^T\left(t_{b,d},\cdot\right)$ by a controlled trajectory. By optimality, the norm of $\left\|y^T(t)\right\|_{\xLinfty\left(\Omega\right)}+\left\|q^T(t)\right\|_{\xLinfty\left(\Omega\right)}$ (in black) is bounded by the corresponding norms (in blue) of the constructed trajectory.}
			\label{Proof of Lemmalemma_bound_optima drawing}
		\end{center}
	\end{figure}
	At this point (see figure \ref{Proof of Lemmalemma_bound_optima drawing}), we want to apply Lemma \ref{lemma_optima_bounds_specified_terminal_states}. To this purpose, we set up a control problem like \cref{semilinear_internal_2}-\cref{functional_specified_terminal_states} with specified final state
	\begin{eqnarray*}
		\hat{y}&\coloneqq&y^T\nonumber\\
		t_1&\coloneqq &t_{a,c}\nonumber\\
		t_2&\coloneqq &t_{b,d}.
	\end{eqnarray*}
	By \eqref{lemma_optima_bounds_specified_terminal_states_eq_1}, \eqref{Linf_bound_nice_part_case2a} and \eqref{Linf_bound_nice_part_case2b},
	\begin{eqnarray}\label{lemma_optima_bounds_specified_terminal_states_eq_1_appl_2}
	\min_{\mathscr{U}_{\mbox{\tiny{ad}}}}J_{t_1,t_2}&\leq& K\left[\|y^T(t_{a,c})\|_{\xLinfty\left(\Omega\right)}^2+(t_{b,d}-t_{a,c})\|z\|_{\xLinfty(\omega_0)}^2\right.,\nonumber\\
	&\;&\left.+\|u^T\|_{\xLinfty((t_{b,d}-T_R,t_{b,d})\times \omega)}^2+\|y^T(t_{b,d}-T_R)\|_{\xLinfty\left(\Omega\right)}^2\right]\nonumber\\
	&\leq&K_{\theta}\left[\|y_0\|_{\xLinfty\left(\Omega\right)}^2+\|z\|_{\xLinfty(\omega_0)}^2\right]+\gamma (t_{b,d}-t_{a,c})\|z\|_{\xLinfty(\omega_0)}^2,
	\end{eqnarray}
	where $K_{\theta}=K_{\theta}(\Omega,f,R,\theta)$ and $\gamma=\gamma(\Omega,f,R)$. In our case the target trajectory for \cref{semilinear_internal_2}-\cref{functional_specified_terminal_states} is the state $y^T$ associated to an optimal control $u^T$ for \cref{semilinear_internal_1_slt}-\cref{functional_slt}. Then, by definition of \cref{semilinear_internal_2}-\cref{functional_specified_terminal_states},
	\begin{equation*}
	J_{t_1,t_2}\left(u^T\right)\leq J_{t_1,t_2}(u),\hspace{0.3 cm}\forall \ u\in \mathscr{U}_{\mbox{\tiny{ad}}}.
	\end{equation*}
	Hence, by \eqref{lemma_optima_bounds_specified_terminal_states_eq_1_appl_2},
	\begin{eqnarray*}
		J_{t_1,t_2}\left(u^T\right)&\leq&\min_{\mathscr{U}_{\mbox{\tiny{ad}}}}J_{t_1,t_2}\nonumber\\
		&\leq&K_{\theta}\left[\|y_0\|_{\xLinfty\left(\Omega\right)}^2+\|z\|_{\xLinfty(\omega_0)}^2\right]+\gamma (t_{b,d}-t_{a,c})\|z\|_{\xLinfty(\omega_0)}^2.
	\end{eqnarray*}
	By definition of $\mathscr{I}_T$ \cref{definition_I} and $C_1$ \cref{def_Cj}, we have
	\begin{eqnarray}\label{case2_eq16}
	\int_{a_j}^{b_j}\left[\|q^T(t)\|_{\xLinfty\left(\Omega\right)}^2+\|y^T(t)\|_{\xLinfty\left(\Omega\right)}^2\right]dt&\geq&\sum_{i\in C_j}\theta\left[\|y_0\|_{\xLinfty\left(\Omega\right)}^2+\|z\|_{\xLinfty(\omega_0)}^2\right]\nonumber\\
	&=&\frac{\theta (b_j-a_j)}{3T_R}\left[\|y_0\|_{\xLinfty\left(\Omega\right)}^2+\|z\|_{\xLinfty(\omega_0)}^2\right]\nonumber\\
	&>&\frac{\theta (t_{b,d}-t_{a,c}-6T_R)}{3T_R}\left[\|y_0\|_{\xLinfty\left(\Omega\right)}^2+\|z\|_{\xLinfty\left(\omega_0\right)}^2\right],
	\end{eqnarray}
	where in the last inequality we have used \cref{case2_eq3} and \cref{case2_eq6} to get\\
	$b_j-a_j>t_{b,d}-t_{a,c}-6T_R$. By the above inequality, Lemma \ref{lemma_reg_nonnegativepotential} and \eqref{case2_eq16},
	\begin{eqnarray*}
		\frac{\theta(t_{b,d}-t_{a,c}-6T_R)}{6T_R}\left[\|y_0\|_{\xLinfty\left(\Omega\right)}^2+\|z\|_{\xLinfty(\omega_0)}^2\right]&\;&\nonumber\\
		+\frac12\int_{a_j}^{b_j}\left[\|q^T(t)\|_{\xLinfty\left(\Omega\right)}^2+\|y^T(t)\|_{\xLinfty\left(\Omega\right)}^2\right]dt&\leq&\int_{a_j}^{b_j}\left[\|q^T(t)\|_{\xLinfty\left(\Omega\right)}^2+\|y^T(t)\|_{\xLinfty\left(\Omega\right)}^2\right]dt\nonumber\\
		&\leq&K\left[J_{t_1,t_2}\left(u^T\right)+\|y_0\|_{\xLinfty\left(\Omega\right)}^2\right]\nonumber\\
		&\leq&K_{\theta}\left[\|y_0\|_{\xLinfty\left(\Omega\right)}^2+\|z\|_{\xLinfty(\omega_0)}^2\right]\nonumber\\
		&\;&+\gamma (t_{b,d}-t_{a,c})\|z\|_{\xLinfty(\omega_0)}^2,
	\end{eqnarray*}
	whence
	\begin{eqnarray*}
		\int_{a_j}^{b_j}\left[\|q^T(t)\|_{\xLinfty\left(\Omega\right)}^2+\|y^T(t)\|_{\xLinfty\left(\Omega\right)}^2\right]dt&\leq& K_{\theta}\left[\|y_0\|_{\xLinfty\left(\Omega\right)}^2+\|z\|_{\xLinfty(\omega_0)}^2\right]\nonumber\\
		&\;&+2\left(\gamma (t_{b,d}-t_{a,c})-\theta\frac{(t_{b,d}-t_{a,c}-6T_R)}{6T_R}\right) \|z\|_{\xLinfty\left(\omega_0\right)}^2\nonumber\\
		&\leq & K_{\theta}\left[\|y_0\|_{\xLinfty\left(\Omega\right)}^2+\|z\|_{\xLinfty(\omega_0)}^2\right]\nonumber\\
		&\;&+2(t_{b,d}-t_{a,c})\left(\gamma -\frac{\theta}{6T_R}\right) \|z\|_{\xLinfty(\omega_0)}^2.
	\end{eqnarray*}
	If $\theta$ is large enough, we have $\gamma -\frac{\theta}{6T_R}<0$. Hence, choosing $\theta$ large enough, we obtain the estimate
	\begin{equation*}
	\int_{a_j}^{b_j}\left[\|q^T(t)\|_{\xLinfty\left(\Omega\right)}^2+\|y^T(t)\|_{\xLinfty\left(\Omega\right)}^2\right]dt\leq K_{\theta}\left[\|y_0\|_{\xLinfty\left(\Omega\right)}^2+\|z\|_{\xLinfty(\omega_0)}^2\right].
	\end{equation*}

	\vspace*{4pt}\noindent\textbf{Case 3.} $a_j>0$ and $b_j=3T_RN_T$.
	
	We now work in case \cref{theta_smallness} is not satisfied in $[a_j,b_j]$, with $b_j=3T_RN_T$. We provide an estimate in the final interval $[a_j,T]$. As we shall see, in this case, we will not employ the exact controllability of \cref{semilinear_internal_1_slt}. We shall rather use the stability of the uncontrolled equation.
	
	Since $a_j>0$, we have
	\begin{equation}\label{case4_eq1}
	\int_{a_j-3T_R}^{a_j}\left[\|q^T(t)\|_{\xLinfty\left(\Omega\right)}^2+\|y^T(t)\|_{\xLinfty\left(\Omega\right)}^2\right]dt\leq\theta\left[\|y_0\|_{\xLinfty\left(\Omega\right)}^2+\|z\|_{\xLinfty(\omega_0)}^2\right].
	\end{equation}
	We apply Lemma \ref{lemma_int_point} in $[a_j-3T_R,a_j]$. To this end, set $c\coloneqq a_j-3T_R$, $d\coloneqq a_j$ and $h(t)\coloneqq \|q^T(t)\|_{\xLinfty\left(\Omega\right)}^2+\|y^T(t)\|_{\xLinfty\left(\Omega\right)}^2$. By Lemma \ref{lemma_int_point}, there exist $t_c$,
	\begin{equation}\label{case4_eq3}
	a_j-3T_R<t_c<a_j-2T_R
	\end{equation}
	such that
	\begin{eqnarray}\label{case4_eq4}
	\|q^T(t_c)\|_{\xLinfty\left(\Omega\right)}^2+\|y^T(t_c)\|_{\xLinfty\left(\Omega\right)}^2&\leq&\frac{1}{T_R}\int_{a_j-3T_R}^{a_j}\left[\|q^T(t)\|_{\xLinfty\left(\Omega\right)}^2+\|y^T(t)\|_{\xLinfty\left(\Omega\right)}^2\right]dt\nonumber\\
	&\leq&\frac{\theta}{T_R}\left[\|y_0\|_{\xLinfty\left(\Omega\right)}^2+\|z\|_{\xLinfty(\omega_0)}^2\right].
	\end{eqnarray}
	
	We introduce the control
	\begin{equation*}
	u^*\coloneqq \begin{dcases}
	u^T \quad &\mbox{in} \ (0,t_c)\\
	0 \quad &\mbox{in} \ (t_c,T)
	\end{dcases}
	\end{equation*}
	Let $y$ be the solution to \cref{semilinear_internal_1_slt}, with initial datum $y_0$ and control $u$ and $y^*$ be the solution to \cref{semilinear_internal_1_slt}, with initial datum $y_0$ and control $u^*$. By definition of minimizer, we have
	\begin{eqnarray*}
		J_{T}\left(u^T\right)&\leq&J_{T}(u^*)\nonumber\\
		&\leq&\frac12 \int_0^T\int_{\omega} |u^*|^2 dxdt+\frac{\beta}{2}\int_0^T\int_{\omega_0} |y^*-z|^2 dxdt\nonumber\\
		&=&\frac12 \int_0^{t_c}\int_{\omega} \left|u^T\right|^2 dxdt+\frac{\beta}{2}\int_0^{t_c}\int_{\omega_0} \left|y^T-z\right|^2 dxdt\nonumber\\
		&\;&+\frac{\beta}{2}\int_{t_c}^{T}\int_{\omega_0} |y^*-z|^2 dxdt,
	\end{eqnarray*}
	whence,
	\begin{eqnarray*}
		\frac12 \int_{t_c}^{T}\int_{\omega} \left|u^T\right|^2 dxdt+\frac{\beta}{2}\int_{t_c}^{T}\int_{\omega_0} \left|y^T-z\right|^2 dxdt&\leq&\frac{\beta}{2}\int_{t_c}^{T}\int_{\omega_0} |y^*-z|^2 dxdt\nonumber\\
		&\leq&K\left[\|y(t_c)\|_{\xLinfty\left(\Omega\right)}^2+(T-t_c)\|z\|_{\xLinfty(\omega_0)}^2\right]\nonumber\\
		&\leq&K_{\theta}\left[\|y_0\|_{\xLinfty\left(\Omega\right)}^2+\|z\|_{\xLinfty(\omega_0)}^2\right]\nonumber\\
		&\;&+\gamma(T-t_c)\|z\|_{\xLinfty(\omega_0)}^2,
	\end{eqnarray*}
	where we have used \eqref{case4_eq4} and $K_{\theta}=K_{\theta}(\Omega,f,R,\theta)$ and $\gamma=\gamma(\Omega,f,R)$.
	
	Now, on the one hand, by Lemma \ref{lemma_reg_nonnegativepotential} applied to the state and the adjoint equation in \cref{semilinear_internal_parabolic_1}, we have
	\begin{eqnarray}\label{Case_4_eq16}
	\int_{t_c}^{T}\left[\|q^T(t)\|_{\xLinfty\left(\Omega\right)}^2+\|y^T(t)\|_{\xLinfty\left(\Omega\right)}^2\right]dt&\leq&K_{\theta}\left[\|y_0\|_{\xLinfty\left(\Omega\right)}^2+\|z\|_{\xLinfty(\omega_0)}^2\right]\nonumber\\
	&\;&+\gamma(T-t_c)\|z\|_{\xLinfty(\omega_0)}^2.
	\end{eqnarray}
	
	On the other hand, by \cref{case4_eq3}, $-a_j>-t_c-3T_R$ and, since $b_j=3T_RN_T$, $b_j\geq T-3T_R$. Hence, $b_j-a_j>T-t_c-6T_R$. Then, by \cref{definition_I},
	\begin{eqnarray*}
		\int_{a_j}^{T}\left[\|q^T(t)\|_{\xLinfty\left(\Omega\right)}^2+\|y^T(t)\|_{\xLinfty\left(\Omega\right)}^2\right]dt&\geq&\int_{a_j}^{b_j}\left[\|q^T(t)\|_{\xLinfty\left(\Omega\right)}^2+\|y^T(t)\|_{\xLinfty\left(\Omega\right)}^2\right]dt\nonumber\\
		&\geq&\sum_{i\in C_j}\theta\left[\|y_0\|_{\xLinfty\left(\Omega\right)}^2+\|z\|_{\xLinfty(\omega_0)}^2\right]\nonumber\\
		&=&\frac{\theta (b_j-a_j)}{3T_R}\left[\|y_0\|_{\xLinfty\left(\Omega\right)}^2+\|z\|_{\xLinfty(\omega_0)}^2\right]\nonumber\\
		&>&\frac{\theta (T-t_{c}-6T_R)}{3T_R}\left[\|y_0\|_{\xLinfty\left(\Omega\right)}^2+\|z\|_{\xLinfty(\omega_0)}^2\right].
	\end{eqnarray*}
	By the above inequality and Lemma \ref{lemma_reg_nonnegativepotential} and \eqref{Case_4_eq16},
	\begin{eqnarray*}
		\frac{\theta (T-t_{c}-6T_R)}{6T_R}\left[\|y_0\|_{\xLinfty\left(\Omega\right)}^2+\|z\|_{\xLinfty(\omega_0)}^2\right]&\;&\nonumber\\
		+\frac12\int_{a_j}^{T}\left[\|q^T(t)\|_{\xLinfty\left(\Omega\right)}^2+\|y^T(t)\|_{\xLinfty\left(\Omega\right)}^2\right]dt&\leq&\int_{a_j}^{T}\left[\|q^T(t)\|_{\xLinfty\left(\Omega\right)}^2+\|y^T(t)\|_{\xLinfty\left(\Omega\right)}^2\right]dt\nonumber\\
		&\leq&K_{\theta}\left[\|y_0\|_{\xLinfty\left(\Omega\right)}^2+\|z\|_{\xLinfty(\omega_0)}^2\right]\nonumber\\
		&\;&+\gamma(T-t_c)\|z\|_{\xLinfty(\omega_0)}^2,
	\end{eqnarray*}
	whence
	\begin{eqnarray*}
		\int_{a_j}^{T}\left[\|q^T(t)\|_{\xLinfty\left(\Omega\right)}^2+\|y^T(t)\|_{\xLinfty\left(\Omega\right)}^2\right]dt&\leq& K_{\theta}\left[\|y_0\|_{\xLinfty\left(\Omega\right)}^2+\|z\|_{\xLinfty(\omega_0)}^2\right]\nonumber\\
		&\;&+2\left(\gamma (T-t_{c})-\theta\frac{(T-t_{c}-6T_R)}{6T_R}\right) \|z\|_{\xLinfty(\omega_0)}^2\nonumber\\
		&\leq & K_{\theta}\left[\|y_0\|_{\xLinfty\left(\Omega\right)}^2+\|z\|_{\xLinfty(\omega_0)}^2\right]\nonumber\\
		&\;&+2(T-t_{c})\left(\gamma -\frac{\theta}{6T_R}\right) \|z\|_{\xLinfty(\omega_0)}^2.
	\end{eqnarray*}
	If $\theta$ is large enough, we have $\gamma -\frac{\theta}{6T_R}<0$. Hence, choosing $\theta$ large enough, we obtain the estimate
	\begin{equation*}
	\int_{a_j}^{T}\left[\|q^T(t)\|_{\xLinfty\left(\Omega\right)}^2+\|y^T(t)\|_{\xLinfty\left(\Omega\right)}^2\right]dt\leq K_{\theta}\left[\|y_0\|_{\xLinfty\left(\Omega\right)}^2+\|z\|_{\xLinfty(\omega_0)}^2\right].
	\end{equation*}
	
	\textit{Step 2} \  \textbf{Conclusion}\\
	The proof is concluded, with an application of Lemma \ref{lemma_subinterval_Linf} to the state and the adjoint equation in \cref{semilinear_internal_parabolic_1}.
\end{proof}

\section{Upper bound for the minimal cost}
\label{appendixsec:Convergence of averages}

This section is devoted to the proof of Lemma \ref{lemma_upper_bound_value}.

\begin{proof}[Proof of Lemma \ref{lemma_upper_bound_value}]
	Let $\overline{u}\in \xLinfty\left(\Omega\right)$ be an optimal control for \cref{semilinear_internal_elliptic_1_slt}-\cref{steady_functional_slt} and let $\overline{y}$ be the corresponding solution to \cref{semilinear_internal_elliptic_1_slt} with control $\overline{u}$. Following step 1 of the proof of Lemma \ref{lemma_opt_est}, we obtain $\overline{u}\in \xCzero(\overline{\omega})$ and
	\begin{equation}\label{lemma_upper_bound_value_eq3}
	\|\overline{u}\|_{\xLinfty\left(\omega\right)}\leq K\|z\|_{\xLinfty\left(\omega_0\right)}.
	\end{equation}
	\textit{Step 1} \ \textbf{Proof of
		\begin{equation*}
		\left|J_T\left(\overline{u}\right)-T\inf_{\xLtwo\left(\omega\right)}J_s\right|\leq K,
		\end{equation*}with $K$ independent of $T$}\\
	Let $\hat{y}$ be the solution to
	\begin{equation}\label{semilinear_internal_3_slt}
	\begin{dcases}
	\hat{y}_t-\Delta \hat{y}+f\left(\hat{y}\right)=\overline{u}\chi_{\omega}\hspace{2.8 cm} & \mbox{in} \hspace{0.10 cm}(0,T)\times\Omega\\
	\hat{y}=0  & \mbox{on}\hspace{0.10 cm} (0,T)\times \partial \Omega\\
	\hat{y}(0,x)=y_0(x)  & \mbox{in}\hspace{0.10 cm}  \Omega.
	\end{dcases}
	\end{equation}
	Set $\eta\coloneqq \hat{y}-\overline{y}$ solution to
	\begin{equation}\label{semilinear_internal_4_slt}
	\begin{dcases}
	\eta_t-\Delta \eta+f\left(\hat{y}\right)-f\left(\overline{y}\right)=0\hspace{2.8 cm} & \mbox{in} \hspace{0.10 cm}(0,T)\times\Omega\\
	\eta=0  & \mbox{on}\hspace{0.10 cm} (0,T)\times \partial \Omega\\
	\eta(0,x)=y_0(x)-\overline{y}(x)  & \mbox{in}\hspace{0.10 cm}  \Omega.
	\end{dcases}
	\end{equation}	
	By multiplying \cref{semilinear_internal_4_slt} by $\eta$, since $f$ is increasing, for any $t\in [0,T]$ we have
	\begin{equation}\label{lemma_upper_bound_value_eq_15}
	\left\|\hat{y}(t,\cdot)-\overline{y}\right\|_{\xLtwo\left(\Omega\right)}\leq \exp(-\lambda_1 t)\left\|y_{0}-\overline{y}\right\|_{\xLtwo\left(\Omega\right)},
	\end{equation}
	where $\lambda_1$ is the first eigenvalue of $-\Delta :\xHone_0\left(\Omega\right)\longrightarrow \xHmone\left(\Omega\right)$.
	
	At this point, let us take the difference
	\begin{eqnarray}\label{lemma_upper_bound_value_eq_16}
	|J_T\left(\overline{u}\right)-T\inf_{\xLtwo\left(\omega\right)}J_s|&=&\frac12\left|\int_0^T\int_{\omega_0}\left[\left|\hat{y}-z\right|^2-\left|\overline{y}-z\right|^2\right]dxdt\right|\nonumber\\
	&\leq &\frac12\int_0^T\int_{\omega_0}\left|\hat{y}-\overline{y}\right|^2 dxdt+\int_{0}^T\int_{\omega_0}\left|\overline{y}-z\right|\left|\hat{y}-\overline{y}\right|dxdt\nonumber\\
	&\leq &K\left\|y_{0}-\overline{y}\right\|_{\xLtwo\left(\Omega\right)}^2+K\left\|y_{0}-\overline{y}\right\|_{\xLtwo\left(\Omega\right)}\leq K,\label{lemma_upper_bound_value_eq_16_1}\\
	\end{eqnarray}
	where in \eqref{lemma_upper_bound_value_eq_16_1} we have used \cref{lemma_upper_bound_value_eq_15} and \cref{lemma_upper_bound_value_eq3} and the constant $K$ is independent of the time horizon $T$.\\
	\textit{Step 2} \ \textbf{Conclusion}\\
	By the above reasoning, we have
	\begin{eqnarray*}
		\inf_{\xLtwo((0,T)\times \omega)}J_{T}&\leq&J_{T}\left(\overline{u}\right)\nonumber\\
		&=&T\inf_{\xLtwo\left(\omega\right)}J_s+J_{T}\left(\overline{u}\right)-T\inf_{\xLtwo\left(\omega\right)}J_s\nonumber\\
		&\leq &T\inf_{\xLtwo\left(\omega\right)}J_s+K.
	\end{eqnarray*}
	This finishes the proof.
\end{proof}
	
	\bibliography{my_references}
	\bibliographystyle{siam}
	
\end{document}